\documentclass[11pt]{article}
\usepackage[english]{babel} 
\usepackage[utf8]{inputenc} 
\usepackage[T1]{fontenc} 
\usepackage{amsmath}
\usepackage{amsfonts}
\usepackage{amssymb}
\usepackage{amsthm}
\usepackage{comment}
\usepackage{lmodern}
\usepackage{hyperref}
\usepackage{graphicx,subfigure,color}

\numberwithin{equation}{section} 

\newtheorem{theorem}{Theorem}[section]

\newtheorem{proposition}[theorem]{Proposition}
\newtheorem{lemma}[theorem]{Lemma}
\newtheorem{corollary}[theorem]{Corollary}
\newtheorem{definition}[theorem]{Definition}

\def\R{\mathbb{R}}

\def\C{\mathbb{C}}
\def\N{\mathbb{N}}

\author{Virginie Bonnaillie-No\"el, Benedetta Noris, Manon Nys, Susanna Terracini}

\title{On the eigenvalues of Aharonov-Bohm operators with varying poles 
\footnote{B. Noris and S. Terracini are partially supported by the PRIN2009 grant ``Critical Point Theory and Perturbative Methods for Nonlinear Differential Equations''. M. Nys is a Research Fellow of the Belgian Fonds de la Recherche Scientifique - FNRS. V. Bonnaillie-No\"el is supported by the ANR (Agence Nationale de la Recherche), project {\sc Optiform} n$^{\rm o}$ ANR-12-BS01-0007-02.}
}

\begin{document}

\maketitle

\begin{abstract}
We consider a magnetic operator of Aharonov-Bohm type with Dirichlet boundary conditions in a planar domain. We analyse the behavior of its eigenvalues as the singular pole moves in the domain. For any value of the circulation we prove that the $k$-th magnetic eigenvalue converges to the $k$-th eigenvalue of the Laplacian as the pole approaches the boundary. We show that the magnetic eigenvalues depend in a smooth way on the position of the pole, as long as they remain simple. In case of half-integer circulation, we show  that the rate of convergence depends on the number of nodal lines of the corresponding magnetic eigenfunction. In addition, we provide several numerical simulations both on the circular sector and on the square, which find a perfect theoretical justification within our main results, together with the ones in \cite{BH}. 
\end{abstract}

\section{Introduction}
Let $\Omega\subset\R^2$ be an open, simply connected, bounded set. For $a=(a_1,a_2)$ varying in $\Omega$, we consider the magnetic Schr\"odinger operator
\[
(i \nabla + A_{a})^{2}u=-\Delta u +i\nabla\cdot A_{a} u +2i A_{a}\cdot \nabla u+|A_{a}|^2u
\]
acting on functions with zero boundary conditions on $\partial \Omega$, where $A_a$ is a magnetic potential of Aharonov-Bohm type, singular at the point $a$. More specifically, the magnetic potential has the form
\begin{equation}\label{eq:magnetic_potential_definition}
A_{a}(x) = \alpha \left( - \frac{x_{2} - a_{2}}{(x_{1} - a_{1})^{2}+(x_{2} - a_{2})^{2}}, \frac{x_{1} - a_{1}}{(x_{1} - a_{1})^{2}+(x_{2} - a_{2})^{2}} \right) + \nabla \chi
\end{equation}
where $x=(x_{1},x_{2})\in\Omega\setminus\{a\}$, $\alpha\in (0,1)$ is a fixed constant and $\chi \in C^\infty(\bar{\Omega})$. Since the regular part $\chi$ does not play a significant role, throughout the paper we will suppose without loss of generality that $\chi\equiv0$.

The magnetic field associated to this potential is a $2\pi\alpha$-multiple of the Dirac delta at $a$, orthogonal to the plane. A quantum particle moving in $\Omega \setminus\{a\}$ will be affected by the magnetic potential, although it remains in a region where the magnetic field is zero (Aharonov-Bohm effect \cite{AB}). We can think at the particle as being affected by the non-trivial topology of the set $\Omega\setminus\{a\}$.

We are interested in studying the behavior of the spectrum of the operator $(i \nabla + A_{a})^{2}$ as $a$ moves in the domain and when it approaches its boundary. By standard spectral theory, the spectrum of such operator consists of a diverging sequence of real positive eigenvalues (see Section \ref{sec:preliminaries}). We will denote by $\lambda_j^a$, $j\in \N\setminus\{0\}$, the eigenvalues counted with their multiplicity (see \eqref{eigenvalueofmagneticoperator}) and by $\varphi_j^a$ the corresponding eigenfunctions, normalized in the $L^2(\Omega)$-norm. 
We shall focus our attention on the extremal and critical points of the maps $a\mapsto \lambda_j^a$.


One motivation for our study is that, in the case of half integer circulation, critical positions of the moving pole can be related to optimal partition problems. The link between spectral minimal partitions and nodal domains of eigenfunctions has been investigated in full detail in (\cite{HH2012,HHT2009,HHT2010dim3,HHT2010}). By the results in \cite{HHT2009}, in two dimensions, the boundary of a minimal partition is the union  of finitely many regular arcs, meeting at some multiple intersection points dividing the angle in an equal fashion.  If  the multiplicity of the clustering domains is even, then the partition is nodal, i.e. it is the nodal set of an eigenfunction. On the other hand,  the results in \cite{BH,BHH,BHV,TN1} suggest that the minimal partitions featuring a clustering point of odd multiplicity  should be related to the nodal domains of eigenfunctions of Aharonov-Bohm Hamiltonians which corresponds to a critical value of the eigenfunction with respect to the moving pole.

%
%
Our first result states the continuity of the magnetic eigenvalues with respect to the position of the singularity, up to the boundary.
\begin{theorem}\label{theorem:continuity}
For every $j\in \N\setminus\{0\}$, the function $a\in\Omega \mapsto \lambda_j^a \in \R$ admits a continuous extension on $\overline\Omega$. More precisely, as $a\to\partial\Omega$, we have that $\lambda_{j}^{a}$ converges to $\lambda_{j}$, the $j$-th eigenvalue of $-\Delta$ in $H^1_0(\Omega)$.
\end{theorem}
We remark that this holds for every $\alpha\in (0,1)$. As an immediate consequence of this result, we have that this map, being constant on $\partial\Omega$, always admits an interior extremal point.
\begin{corollary}\label{cor:extremal_point_interior_existence}
For every $j\in \N\setminus\{0\}$, the function $a\in\Omega \mapsto \lambda_j^a \in \R$ has an extremal point in $\Omega$.
\end{corollary}
%
Heuristically, we can interpret the previous theorem thinking at a magnetic potential $A_b$, singular at $b\in\partial\Omega$. The domain $\Omega\setminus\{b\}$ coincides with $\Omega$, so that it has a trivial topology. For this reason, the magnetic potential is not experienced by a particle moving in $\Omega$ and the operator acting on the particle is simply the Laplacian.

This result was first conjectured in the case $k=1$ in \cite{TN1}, where it was applied to show that the function $a\mapsto \lambda_1^a$ has a global interior maximum, where it is not differentiable, corresponding to an eigenfunction of multiplicity exactly two. 
Numerical simulations in \cite{BH} supported the conjecture for every $k$.
During the completion of this work, we became aware that the continuity of the eigenvalues with respect to multiple moving poles has been obtained independently in \cite{Lena}.

We remark that the continuous extension up to the boundary is a non-trivial issue because the nature of the the operator changes as $a$ approaches $\partial\Omega$. 
This fact can be seen in the more specific case when $\alpha=1/2$, which is equivalent to the standard Laplacian on the double covering (see \cite{HHHO1999,HHOHOO,TN1}). We go then from a problem on a fixed domain with a varying operator (which depends on the singularity $a$) to a problem with a fixed operator (the laplacian) and a varying domain (for the convergence of the eigenvalues of elliptic operators on varying domains, we refer to \cite{AD,D}).
In this second case, the singularity is transferred from the operator into the domain. Indeed, when $a$ approaches the boundary, the double covering develops a corner at the origin. In particular, Theorem 7.1 in \cite{HHT2010dim3} cannot be applied in our case since there is no convergence in capacity of the domains.

In the light of the previous corollary it is natural to study additional properties of the extremal points.
Our aim is to establish a relation between the nodal properties of $\varphi_j^b$ and the vanishing order of $|\lambda_j^a-\lambda_j^b|$ as $a\to b$. First of all we will need some additional regularity, which is guaranteed by the following theorem in case of simple eigenvalues and regular domain.
\begin{theorem}\label{theorem:differentiability_simple_eigenvalue}
Let $b\in\Omega$. If $\lambda^b_j$ is simple, then, for every $j\in \N\setminus\{0\}$, the map $a \in \Omega \mapsto \lambda^a_j$ is locally of class $C^\infty$ in a neighborhood of $b$.
\end{theorem}

In order to examine the link with the nodal set of eigenfunctions, we shall focus on the case $\alpha=1/2$. In this case, it was proved in \cite{HHHO1999,HHOHOO,TN1} (see also Proposition \ref{proposition:asymptotic_expansion_eigenfunction} below) that the eigenfunctions have an odd number of nodal lines ending at the pole $a$ and an even number of nodal lines meeting at zeros different from $a$. 
We say that an eigenfunction has a zero of order $k/2$ at a point if it has $k$ nodal lines meeting at such point. More precisely, we give the following definition.

\begin{definition}[Zero of order $k/2$]
Let $f:\Omega\to\C$, $b\in\Omega$ and $k\in\N\setminus\{0\}$.
\begin{itemize}
\item[(i)] If $k$ is even, we say that $f$ has a zero of order $k/2$ at $b$ if it is of class at least $C^{k/2}$ in a neighborhood of $b$ and $f(b)=\ldots=D^{k/2-1}f(b)=0$, while $D^{k/2}f(b)\neq0$.
\item[(ii)] If $k$ is odd, we say that $f$ has a zero of order $k/2$ at $b$ if $f(x^2)$ has a zero of order $k$ at $b$ (here $x^2$ is the complex square).
\end{itemize}
\end{definition}

The following result is proved in \cite{TN1}.

\begin{theorem}[{\cite[Theorem 1.1]{TN1}}]
Suppose that $\alpha=1/2$. Fix any $j\in\N\setminus\{0\}$. If $\varphi_j^b$ has a zero of order $1/2$ at $b\in\Omega$ then either $\lambda_j^b$ is not simple, or $b$ is not an extremal point of the map $a\mapsto \lambda_j^a$.
\end{theorem}

Under the assumption that $\lambda_j^b$ is simple, we prove here that the converse also holds. In addition, we show that the number of nodal lines of $\varphi_j^b$ at $b$ determines the 
order of vanishing of $|\lambda_j^b - \lambda_j^a|$ as $a \to b$.

\begin{theorem} \label{theorem:comportmentofeigenvelue}
Suppose that $\alpha=1/2$. Fix any $j\in\N\setminus\{0\}$.
If $\lambda_j^{b}$ is simple and $\varphi_j^{b}$ has a zero of order $k/2$ at $b\in\Omega$, with $k\geq3$ odd, then
\begin{align}
|\lambda_j^{a} - \lambda_j^{b}|  \leq C |a-b|^{(k+1)/2} \quad \text{ as }a\to b,
\end{align}
for a constant $C>0$ independent of $a$.
\end{theorem}

In conclusion, in case of half-integer circulation we have the following picture, which completes Corollary \ref{cor:extremal_point_interior_existence}.

\begin{corollary}
Suppose that $\alpha=1/2$. Fix any $j\in\N\setminus\{0\}$. If $b\in\Omega$ is an extremal point of $a\mapsto\lambda_j^a$ then either $\lambda_j^b$ is not simple, or $\varphi_j^b$ has a zero of order $k/2$ at $b$, $k\geq3$ odd. In this second case, the first $(k-1)/2$ terms of the Taylor expansion of $\lambda_j^a$ at $b$ cancel.
\end{corollary}


In the forthcoming paper \cite{NNT2} we intend to extend Theorem \ref{theorem:comportmentofeigenvelue} to the case $b\in\partial\Omega$. In this case we know from Theorem \ref{theorem:continuity} that $\lambda_j^a$ converges to $\lambda_j$ as $a\to b\in\partial\Omega$ and we aim to estimate the rate of convergence depending on the number of nodal lines of $\varphi_j$ at $b$, motivated by the numerical simulations in \cite{BH}.

We would like to mention that the relation between the presence of a magnetic field and the number of nodal lines of the eigenfunctions, as well as the consequences on the behavior of the eigenvalues, have been recently studied in different contexts, giving rise to surprising conclusions. In \cite{berkolaiko,verdiere} the authors consider a magnetic Schr\"odinger operator on graphs and study the behavior of its eigenvalues as the circulation of the magnetic field varies. In particular, they consider an arbitrary number of singular poles, having circulation close to 0. They prove that the simple eigenvalues of the Laplacian (zero circulation) are critical values of the function $\alpha\mapsto \lambda_j(\alpha)$, which associates to the circulation $\alpha$ the corresponding eigenvalue. In addition, they show that the number of nodal lines of the Laplacian eigenfunctions depends on the Morse index of $\lambda_j(0)$.

The paper is organized as follows. In Section \ref{sec:preliminaries}, we define the functional space $H^{1}_{A_{a}}(\Omega)$, which is the more suitable space to consider our problem. We also recall an Hardy-type inequality and a theorem about the regularity of the eigenfunctions $\varphi_{j}^{a}$. Finally, in the case of an half-integer circulation, we recall the equivalence between the problem we consider and the standard laplacian equation on the double covering. The first part of Theorem \ref{theorem:continuity}, concerning the interior continuity of the eigenvalues $\lambda_{j}^{a}$ is proved in Section \ref{sec:continuity1} and the second part concerning the extension to the boundary is studied in Section \ref{sec:continuity2}. In Section \ref{sec:differentiability}, we prove Theorem \ref{theorem:differentiability_simple_eigenvalue}.  Section \ref{sec:comportementeigenvalues} contains the proof of Theorem \ref{theorem:comportmentofeigenvelue}.
Finally, Section~\ref{sec.numeric} illustrates these results in the case of the angular sector of aperture $\pi/4$ and the square.

\section{Preliminaries}\label{sec:preliminaries}

We will work in the functional space $H^{1}_{A_{a}}(\Omega)$, which is defined as the completion of $C^{\infty}_{0}(\Omega \backslash \{a\})$ with respect to the norm
\[
\| u \|_{H^{1}_{A_{a}}(\Omega)} := \| (i \nabla + A_{a}) u \|_{L^{2}(\Omega)}.
\]
As proved for example in \cite[Lemma 2.1]{TN1}, we have an equivalent characterization
\[
H^{1}_{A_{a}}(\Omega) = \left\{ u \in H^{1}_0(\Omega) \, : \, \frac{u}{|x-a|} \in L^{2}(\Omega) \right\},
\]
and moreover we have that $H^{1}_{A_{a}}(\Omega)$ is continuously embedded in $H^1_0(\Omega)$: there exists a constant $C>0$ such that for every $u\in H^1_{A_a}(\Omega)$ we have
\begin{equation}\label{eq:H_1_A_a_characterization}
\|u\|_{H_0^1(\Omega)}\leq C \|u\|_{H^{1}_{A_{a}}(\Omega) }.
\end{equation}
This is proved by making use of a Hardy-type inequality by Laptev and Weidl \cite{LaWe}. Such inequality also holds for functions with non-zero boundary trace, as shown in \cite[Lemma 7.4]{MOR} (see also \cite{MORerratum}). More precisely, given $D\subset \Omega$ simply connected and with smooth boundary, there exists a constant $C>0$ such that for every $u\in H^1_{A_a}(\Omega)$
\begin{equation}\label{eq:hardy}
\left\|\frac{u}{|x-a|}\right\|_{L^2(D)} \leq C 
\| (i \nabla + A_{a}) u \|_{L^{2}(D)}.
\end{equation}
As a reference on Aharonov-Bohm operators we cite \cite{MR}.
As a consequence of the continuous embedding, we have the following.
\begin{lemma} \label{lemma:compactnessinversemagneticlaplacian}
Let $Im$ be the compact immersion of $H^{1}_{A_{a}}(\Omega)$ into $(H^{1}_{A_{a}}(\Omega))^{\prime}$. Then, the operator $((i \nabla + A_{a})^{2})^{-1} \circ Im : H^{1}_{A_{a}}(\Omega) \rightarrow H^{1}_{A_{a}}(\Omega)$ is compact.
\end{lemma}
As $((i \nabla + A_{a})^{2})^{-1}$ is also symmetric and positive, we deduce that the spectrum of $(i \nabla + A_{a})^{2}$ consists of a diverging sequences of real positive eigenvalues, having finite multiplicity. They also admit the following variational characterization
\begin{align}\label{eigenvalueofmagneticoperator}
\lambda_{j}^{a} = \inf_{ \substack{W_j\subset H^{1}_{A_{a}}(\Omega)\\ \dim W_{j} = j}} \sup_{\Phi \in W_{j}} \frac{\|\Phi\|^2_{H^{1}_{A_{a}}(\Omega)}}{\|\Phi\|^2_{L^2(\Omega)}}.
\end{align}


Recall that $A_a$ has the form \eqref{eq:magnetic_potential_definition} if and only if it satisfies
\begin{equation}\label{eq:ab_potential}
\nabla\times A_a=0 \quad\text{in } \Omega\setminus\{a\} \qquad\text{and} \qquad
\frac{1}{2\pi}\oint_\sigma A_a\cdot dx=\alpha
\end{equation}
for every closed path $\sigma$ which winds once around $a$. The value of the circulation strongly affects the behavior of the eigenfunctions, starting from their regularity, as the following lemma shows.
\begin{lemma}[{\cite[Section 7]{FFT2011}}]\label{lemma:regularity_eigenfunctions}
If $A_a$ has the form \eqref{eq:magnetic_potential_definition} then $\varphi_{j}^{a} \in C^{0,\alpha}(\Omega)$, where $\alpha$ is precisely the circulation of $A_a$. 
\end{lemma}
If the circulations of two magnetic potentials differ by an integer number, the corresponding operators are equivalent under a gauge transformation, so that they have the same spectrum (see \cite[Theorem 1.1]{HHHO1999} and \cite[Lemma 3.2]{TN1}). For this reason, we can set $\chi=0$ in \eqref{eq:ab_potential} and we can consider $\alpha$ in the interval $(0,1)$ without loosing generality. Moreover, in the same papers, it is shown that, when the circulations differ by a value $1/2$, one operator is equivalent to the other one composed with the complex square root. In particular, in case of half-integer circulation the operator is equivalent to the standard Laplacian in the double covering.
\begin{lemma}[{\cite[Lemma 3.3]{HHHO1999}}]\label{eq:gauge_invariance}
Suppose that $A_a$ has the form \eqref{eq:ab_potential} with $\alpha=1/2$ (and $\chi=0$). Then the function
\[
e^{-i\theta(y)}\varphi_j^a(y^2+a) \quad \text{defined in } \{y\in\C:\ y^2+a\in\Omega\},
\]
(here $\theta$ is the angle of the polar coordinates) is real valued and solves the following equation on its domain
\[
-\Delta(e^{-i\theta(y)}\varphi_j^a(y^2+a))=4\lambda_j^a |y|^2 e^{-i\theta(y)}\varphi_j^a(y^2+a).
\] 
\end{lemma}
As a consequence, we have that, in case of half integer circulation, $\varphi_j^a$ behaves, up to a complex phase, as an elliptic eigenfunction far from the singular point $a$. The behavior near $a$  is, up to a complex phase, that of the square root of an elliptic eigenfunction. We summarize the known properties that we will need in the following proposition. The proofs can be found in \cite[Theorem 1.3]{FFT2011}, \cite[Theorem 2.1]{HHHO1999} and \cite[Theorem 1.5]{TN1} (see also \cite{HW}).

\begin{proposition}\label{proposition:asymptotic_expansion_eigenfunction}
Let $\alpha=1/2$. There exists an odd integer $k\geq1$ such that $\varphi_j^a$ has a zero of order $k/2$ at $a$. Moreover, the following asymptotic expansion holds near $a$
\begin{equation*}
\varphi_j^a(|x-a|,\theta_a)=e^{i\alpha\theta_a}\frac{|x-a|^{k/2}}{k} \left[ c_k\cos(k\alpha\theta_a) + d_k \sin(k\alpha\theta_a) \right] + g(|x-a|,\theta_a)
\end{equation*}
where $x-a=|x-a|e^{i\theta_a}$, $c_k^2+d_k^2\neq0$ and the remainder $g$ satisfies
\begin{equation*}
\lim_{r\to0} \frac{\|g(r,\cdot)\|_{C^1(\partial D_{r}(a))}}{r^{k/2}}=0,
\end{equation*}
where $D_r(a)$ is the disk centered at $a$ of radius $r$.
In addition, there is a positive radius $R$ such that $(\varphi_j^a)^{-1}(\{0\})\cap D_R(a)$ consists of $k$ arcs of class $C^\infty$. If $k\geq 3$ then the tangent lines to the arcs at the point $a$ divide the disk into $k$ equal sectors.
\end{proposition}

\section{Continuity of the eigenvalues with respect to the pole in the interior of the domain} \label{sec:continuity1}

In this section we prove the first part of Theorem \ref{theorem:continuity}, that is the continuity of the function $a\mapsto\lambda_j^a$ when the pole $a$ belongs to the interior of the domain.

\begin{lemma} \label{lemma:cut_off_function}
Given $a,b\in\Omega$ there exists a radial cut-off function $\eta_a:\R^2\to \R$ such that $\eta_a(x)=0$ for $|x-a|<2|b-a|$ and moreover
\[
\int_{\R^{2}} \left( |\nabla \eta_a |^{2} + (1 - \eta_a^{2}) \right) \, dx \to 0 \quad \text{ as } a\to b.
\]
\end{lemma}
\begin{proof}
Given any $0<\varepsilon<1$ we set
\begin{align}\label{cutofffunction}
\eta(x) = \left\{ \begin{aligned} & 0 &  0 \leq |x| \leq \varepsilon \\
                                    & \frac{\log \varepsilon - \log |x|}{\log \varepsilon - \log \sqrt{\varepsilon}} & \varepsilon \leq |x| \leq \sqrt{\varepsilon} \\
                                    & 1 & x \geq \sqrt{\varepsilon}. \end{aligned} \right.
\end{align}
Choosing $\varepsilon=2|b-a|$ and $\eta_a(x)=\eta(x-a)$, an explicit calculation shows that the properties are satisfied.
\end{proof}

\begin{lemma}\label{lemma:interior_theta_a_theta_b}
Given $a,b\in \Omega$ there exist $\theta_a$ and $\theta_b$ such that  $\theta_a-\theta_b \in C^\infty(\Omega\setminus\{ta+(1-t)b,\, t\in[0,1]\})$ and moreover in this set we have
\[
\alpha\nabla(\theta_a-\theta_b)=A_a-A_b.
\]
\end{lemma}
\begin{proof}
Let $a=(a_1,a_2)$ and $b=(b_1,b_2)$. Suppose that $a_1<b_1$, the other cases can be treated in a similar way. We shall provide a suitable branch of the polar angle centered at $a$, which is discontinuous on the half-line starting at $a$ and passing through $b$. To this aim we consider the following branch of the arctangent
\[
\arctan:\R\to \left(-\frac{\pi}{2},\frac{\pi}{2}\right).
\]
We set
\[
\theta_{a} = \left\{ \begin{aligned} & \arctan \frac{x_{2} - a_{2}}{x_{1} - a_{1}} \quad & x_{1} > a_{1}, \, x_{2} \geq \frac{b_{2}-a_{2}}{b_{1}-a_{1}}x_{1} + \frac{a_{2}b_{1} - b_{2}a_{1}}{b_{1}-a_{1}} \\
                                     & \pi/2 \quad & x_{1} = a_{1}, \, x_{2} > a_{2} \\
                                     & \pi + \arctan \frac{x_{2} - a_{2}}{x_{1} - a_{1}} \quad & x_{1} < a_{1} \\
                                     & 3\pi/2 \quad & x_{1} = a_{1}, \, x_{2} < a_{2} \\
                                     & 2 \pi + \arctan \frac{x_{2} - a_{2}}{x_{1} - a_{1}} \quad & x_{1} > a_{1}, \, x_{2} < \frac{b_{2}-a_{2}}{b_{1}-a_{1}}x_{1} + \frac{a_{2}b_{1} - b_{2}a_{1}}{b_{1}-a_{1}}. \end{aligned} \right.
\]
With this definition $\theta_{a}$ is regular except on the half-line
\[
x_{2} = \frac{b_{2}-a_{2}}{b_{1}-a_{1}}x_{1} + \frac{a_{2}b_{1} - b_{2}a_{1}}{b_{1}-a_{1}}, \quad
x_1>a_1
\]
and an explicit calculation shows that $\alpha\nabla\theta_a=A_a$ in the set where it is regular. The definition of $\theta_b$ is analogous: we keep the same half-line, whereas we replace $(a_1,a_2)$ with $(b_1,b_2)$ in the definition of the function. One can verify that $\theta_a-\theta_b$ is regular except for the segment from $a$ to $b$.
\end{proof}

Recall that in the following $\varphi_j^a$ is an eigenfunction associated to $\lambda_j^a$, normalized in the $L^2$-norm. Moreover, we can assume that the eigenfunctions are orthogonal.

\begin{lemma}\label{lemma:interior_tilde_varphi_j_definition}
Given $a,b\in \Omega$, let $\eta_a$ be defined as in Lemma \ref{lemma:cut_off_function} and let $\theta_a,\theta_b$ be defined as in Lemma \ref{lemma:interior_theta_a_theta_b}. Fix an integer $k\geq1$ and set, for $j=1,\ldots,k$,
\[
\tilde{\varphi}_{j} = e^{i\alpha(\theta_{a}-\theta_{b})} \eta_{a} \varphi_{j}^{b}.
\]
Then $\tilde{\varphi}_{j} \in H^{1}_{A_{a}}(\Omega)$ and moreover for every $(\alpha_1,\ldots,\alpha_k)\in\R^k$ it holds
\[
(1 - \varepsilon_{a}) \|\sum_{j = 1}^{k} \alpha_{j} \varphi_{j}^b\|^{2}_{L^{2}(\Omega)} \leq 
\|\sum_{j = 1}^{k} \alpha_{j} \tilde{\varphi}_{j} \|_{L^{2}(\Omega)}^{2} \leq 
k \|\sum_{j = 1}^{k} \alpha_{j} \varphi_{j}^b\|^{2}_{L^{2}(\Omega)},
\]
where $\varepsilon_a\to 0$ as $a\to b$.
\end{lemma}
\begin{proof}
Let us prove first that $\tilde{\varphi}_{j} \in H^{1}_{A_{a}}(\Omega)$. By Lemmas \ref{lemma:cut_off_function} and \ref{lemma:interior_theta_a_theta_b} we have that $\theta_a-\theta_b \in C^\infty(\text{supp}\{\eta_a\})$, so that $\tilde{\varphi}_{j} \in H^1_0(\Omega)$. Moreover $\tilde{\varphi}_{j}(x)=0$ if $|x-a|<2|b-a|$, hence $\tilde{\varphi}_{j}/|x-a| \in L^2(\Omega)$. Concerning the inequalities, we compute on one hand
\[
\|\sum_{j = 1}^{k} \alpha_{j} \tilde{\varphi}_{j} \|_{L^{2}(\Omega)}^{2} \leq 
k \sum_{j=1}^k \alpha_j^2 \|\eta_a\varphi_j^b\|_{L^2(\Omega)}^2 
\leq k \sum_{j=1}^k \alpha_j^2 = k \|\sum_{j = 1}^{k} \alpha_{j} \varphi_{j}^b\|^{2}_{L^{2}(\Omega)},
\]
where we used the inequality $\sum_{i,j=1}^{k}\alpha_{i}\alpha_{j} \leq k \sum_{j=1}^{k}\alpha^{2}_{j}$ and the fact that the eigenfunctions are orthogonal and normalized in the $L^2(\Omega)$-norm. On the other hand we compute
\[
\|\sum_{j = 1}^{k} \alpha_{j} \varphi_{j}^b\|^{2}_{L^{2}(\Omega)} - 
\|\sum_{j = 1}^{k} \alpha_{j} \tilde{\varphi}_{j} \|_{L^{2}(\Omega)}^{2} =
\sum_{i,j=1}^k \alpha_i\alpha_j \int_\Omega (1-\eta_a^2)\varphi_i^b \bar{\varphi}_{j}^{b} \,dx.
\]
Thanks to the regularity result proved by Felli, Ferrero and Terracini (see Lemma \ref{lemma:regularity_eigenfunctions}), we have that $\varphi_i^b$ are bounded in $L^\infty(\Omega)$. Therefore the last quantity is bounded by
\[
C k \sum_{j=1}^k \alpha_j^2 \int_\Omega (1-\eta_a^2)\, dx =
C k \|\sum_{j = 1}^{k} \alpha_{j} \varphi_{j}^b\|^{2}_{L^{2}(\Omega)} \int_\Omega (1-\eta_a^2)\, dx
\]
and the conclusion follows from Lemma \ref{lemma:cut_off_function}.
\end{proof}

We have all the tools to prove the first part of Theorem \ref{theorem:continuity}. We will use some ideas from \cite[Theorem 7.1]{HHT2010dim3}.

\begin{theorem}\label{theorem:continuityinterior}
For every $k\in \N\setminus\{0\}$ the function $a\in\Omega \mapsto \lambda_k^a \in \R$ is continuous.
\end{theorem}
\begin{proof}
We divide the proof in two steps.

{\bf Step 1}
First we prove that
\[
\limsup_{a \to b} \lambda_{k}^{a} \leq \lambda_{k}^{b}.
\]
To this aim it will be sufficient to exhibit a $k$-dimensional space $E_{k} \subset H^{1}_{A_{a}}(\Omega)$ with the property that
\begin{align} \label{eq:interior_step1_E_k}
\| \Phi \|_{H^{1}_{A_{a}}(\Omega)}^2 \leq \left(\lambda_{k}^{b} + \varepsilon'_{a} \right) \|\Phi\|^{2}_{L^{2}(\Omega)}
\quad \text{ for every } \Phi\in E_k,
\end{align}
with $\varepsilon'_a\to0$ as $a\to b$. Let span$\{\varphi_1^b,\ldots,\varphi_k^b\}$ be any spectral space attached to $\lambda_1^b,\ldots,\lambda_k^b$. Then we define
\[
E_k:=\text{span}\{\tilde{\varphi}_1,\ldots, \tilde{\varphi}_k\} \quad\text{ with }\quad
\tilde{\varphi}_{j} = e^{i\alpha(\theta_{a}-\theta_{b})} \eta_{a} \varphi_{j}^{b}.
\]
We know from Lemma \ref{lemma:interior_tilde_varphi_j_definition} that $E_k \subset H^{1}_{A_{a}}(\Omega)$. Moreover, it is immediate to see that dim$E_k=k$.
Let us now verify (\ref{eq:interior_step1_E_k}) with $\Phi = \sum_{j=1}^{k} \alpha_{j} \tilde{\varphi}_{j}$, $\alpha_j\in\R$.
We compute
\begin{align}\label{eq:interior_step1_estimate1}
\| \Phi \|_{H^{1}_{A_{a}}(\Omega)}^2  & = \int_{\Omega}| \sum_{j=1}^{k} \alpha_{j} (i \nabla + A_{b}) (\eta_{a}\varphi_{j}^{b}) |^{2} \, \, dx \notag \\
& = \int_{\Omega} \sum_{i,j=1}^{k} \alpha_{i} \alpha_{j} (i \nabla + A_{b})^{2} (\eta_{a}\varphi^{b}_{i} ) (\eta_{a}\bar{\varphi}_{j}^{b}) \, dx,
\end{align}
where we have used the equality
\[
\left(i \nabla + A_{a} \right) \tilde{\varphi}_{j} = e^{i\alpha(\theta_{a}- \theta_{b})} \left(i \nabla + A_{b}\right)(\eta_{a} \varphi_{j}^{b})
\]
and integration by parts. Next notice that
\[
(i \nabla + A_{b}) (\eta_{a}\varphi_{i}^{b}) = (i \nabla + A_{b})\varphi_{i}^{b} \eta_{a} + i \varphi_{i}^{b} \nabla \eta_{a},
\]
so that
\[
(i \nabla + A_{b})^{2} (\eta_{a}\varphi_{i}^{b}) = (i \nabla + A_{b})^{2}\varphi_{i}^{b} \eta_{a} + 2i (i \nabla + A_{b}) \varphi_{i}^{b} \cdot \nabla \eta_{a} - \varphi^{b}_{i} \Delta \eta_{a}.
\]
By replacing in \eqref{eq:interior_step1_estimate1}, we obtain
\begin{align}\label{eq:interior_step1_estimate2}
\| \Phi \|_{H^{1}_{A_{a}}(\Omega)}^2  & = \int_{\Omega} \sum_{i,j=1}^{k} \alpha_i\alpha_j \left( \lambda_i^b \varphi_{i}^{b} \eta_{a} + 2i (i \nabla + A_{b}) \varphi_{i}^{b} \cdot \nabla \eta_{a} - \varphi^{b}_{i} \Delta \eta_{a} \right) \bar{\varphi}^{b}_{j} \eta_{a} \, dx \notag \\
& \leq \lambda_{k}^{b} \|\sum_{j=1}^{k} \alpha_{j} \varphi_{j}^{b}\|^2_{L^{2}(\Omega)} + \beta_{a}
\end{align}
where 
\begin{align}\label{eq:beta_a_definition} 
\beta_{a}  = \int_{\Omega} \sum_{i,j=1}^{k} \alpha_i\alpha_j & \left\{ \lambda_{i}^{b} (\eta_{a}^{2} -1) \varphi_{i}^{b} \bar{\varphi}_{j}^{b} + 2i (i \nabla + A_{b})\varphi_{i}^{b} \cdot \nabla \eta_{a} \bar{\varphi}_{j}^{b} \eta_{a} \right.\notag\\
&- \left.\varphi_{i}^{b} \bar{\varphi}_{j}^{b} \Delta \eta_{a} \eta_{a} \right\} \, dx.
\end{align}
We need to estimate $\beta_a$. From Lemma \ref{lemma:regularity_eigenfunctions} we deduce the existence of a constant $C>0$ such that $\|\varphi_{j}^{b}\|_{L^\infty(\Omega)} \leq C$ for every $j=1,\ldots,k$. Hence
\[
\left|\int_{\Omega} \sum_{i,j=1}^{k} \alpha_i\alpha_j \lambda_{i}^{b}(\eta_{a}^{2} -1) \varphi_{i}^{b} \bar{\varphi}_{j}^{b} \, dx\right| \leq C \sum_{j=1}^k \alpha_j^2 \int_\Omega (1 - \eta_a^2)\, dx.
\]
Using the fact that $\|\varphi_{j}^{b}\|^2_{H^1_0(\Omega)}\leq C\|\varphi_{j}^{b}\|^2_{H^1_{A_a}(\Omega)}=C\lambda_j^b$ (see equation \eqref{eq:H_1_A_a_characterization}), we have
\[
\left|\int_{\Omega} \sum_{i,j=1}^{k} \alpha_i\alpha_j \nabla\varphi_{i}^{b}\cdot\nabla\eta_a \bar{\varphi}_{j}^{b} \eta_a \, dx\right| \leq C \sum_{j=1}^k \alpha_j^2 \left(\int_\Omega |\nabla\eta_a|^2 \, dx\right)^{1/2}.
\]
Next we apply the Hardy inequality \eqref{eq:hardy} to obtain
\[
\begin{split}
& \left|\int_{\Omega} \sum_{i,j=1}^{k} \alpha_i\alpha_j \varphi_{i}^{b} \bar{\varphi}_{j}^{b} A_b\cdot\nabla\eta_a  \eta_a \, dx \right| 
\leq C \sum_{j=1}^k \alpha_j^2 \int_\Omega |\varphi_j^bA_b\cdot\nabla\eta_a|\,dx \\
& \leq C \sum_{j=1}^k \alpha_j^2 \left\|\frac{\varphi_j^b}{x-b}\right\|_{L^2(\Omega)} 
\|(x-b)A_b\|_{L^\infty(\Omega)} \|\nabla\eta_a\|_{L^2(\Omega)} \\
&\leq C \sum_{j=1}^k \alpha_j^2 \|\nabla\eta_a\|_{L^2(\Omega)}.
\end{split}
\]
Concerning the last term in \eqref{eq:beta_a_definition}, similar estimates give
\begin{align*}
&\left|\int_{\Omega} \sum_{i,j=1}^{k} \alpha_i\alpha_j \varphi_{i}^{b} \bar{\varphi}_{j}^{b} \Delta \eta_a  \eta_a \, dx\right| \\
&= \left|\int_{\Omega} \sum_{i,j=1}^{k} \alpha_i\alpha_j \left(|\nabla\eta_a|^2 \varphi_{i}^{b} \bar{\varphi}_{j}^{b} +\eta_a  \nabla \eta_a\cdot \nabla(\varphi_{i}^{b} \bar{\varphi}_{j}^{b})\right) \, dx\right|\\
&\leq C \sum_{j=1}^k \alpha_j^2 \left(\int_\Omega |\nabla\eta_a|^2 \, dx\right)^{1/2}.
\end{align*}
In conclusion we have obtained
\begin{align*}
|\beta_a| & \leq C \|\sum_{j=1}^{k} \alpha_{j} \varphi_{j}^{b}\|^2_{L^{2}(\Omega)} \left\{ \int_{\Omega} (1 - \eta_{a}^{2}) \, dx  + \left(\int_{\Omega} |\nabla \eta_{a}|^{2}  \, dx \right)^{1/2}\right\} \\
& = \|\sum_{j=1}^{k} \alpha_{j} \varphi_{j}^{b}\|^2_{L^{2}(\Omega)} \varepsilon''_a,
\end{align*}
with $\varepsilon''_a\to0$ as $a\to b$ by Lemma \ref{lemma:cut_off_function}.
By inserting the last estimate into \eqref{eq:interior_step1_estimate2} and then using Lemma \ref{lemma:interior_tilde_varphi_j_definition} we obtain \eqref{eq:interior_step1_E_k} with $\varepsilon_a'=(\varepsilon_a''+\lambda_k^b\varepsilon_a)/(1-\varepsilon_a)$.

{\bf Step 2}
We want now to prove the second inequality
\[
\liminf_{a \to b} \lambda_{k}^{a} \geq \lambda_{k}^{b}.
\]
From relation \eqref{eq:H_1_A_a_characterization} and Step 1 we deduce
\[
\|\varphi_j^a\|^2_{H^1_0(\Omega)}\leq C \| \varphi_j^a \|^2_{H^{1}_{A_{a}}(\Omega)} \leq C \lambda_j^b.
\]
Hence there exists $\tilde{\varphi}_{j}\in H^1_0(\Omega)$ such that (up to subsequences) $\varphi_{j}^{a}\rightharpoonup \tilde{\varphi}_{j}$ weakly in $H^{1}_{0}(\Omega)$ and $\varphi_{j}^{a}\to \tilde{\varphi}_{j}$ strongly in $L^{2}(\Omega)$, as $a\to b$. In particular we have
\begin{equation}\label{eq:interior_orthogonality_L2}
\int_\Omega |\tilde{\varphi}_{j}|^2 \, dx=1 \quad\text{ and }\quad
\int_\Omega \tilde{\varphi}_{i}\tilde{\varphi}_{j}\, dx=0
\text{ if } i\neq j.
\end{equation}
Moreover, Fatou's lemma, relation \eqref{eq:hardy} and Step 1 provide
\[ 
\begin{split}
\|\tilde{\varphi}_{j}/|x-b|\|_{L^2(\Omega)} \leq \liminf_{a\to b} \|\varphi^{a}_{j}/|x-a|\|_{L^2(\Omega)} \leq C\liminf_{a\to b} \|\varphi_j^a\|_{H^1_{A_a}(\Omega)} \\
= C\liminf_{a\to b} \sqrt{\lambda_j^a} \leq C\sqrt{\lambda_j^b},
\end{split}
\]
so we deduce that $\tilde{\varphi}_{j} \in H^{1}_{A_{b}}(\Omega)$.

Given a test function $\phi \in C^{\infty}_{0}(\Omega \backslash \{b\})$, consider $a$ sufficiently close to $b$ so that $a\not\in \text{supp}\{\phi\}$. We have that
\[ 
\begin{split}
& \int_{\Omega} \lambda_{j}^{a} \varphi_{j}^{a} \bar{\phi} \, dx   
=\int_{\Omega} \varphi_{j}^{a} \overline{(i \nabla + A_{a})^{2}\phi} \, dx \\
&= \int_{\Omega} \left\{ -\Delta\varphi_j^a\bar\phi+\varphi_j^a[\overline{i\nabla\cdot A_a\phi+2iA_a\cdot\nabla\phi+|A_a|^2\phi}] \right\}\,dx \\
&= \int_{\Omega} \left\{(i\nabla+A_b)^2\varphi_j^a\bar{\phi}-i\nabla\cdot(A_a+A_b)\varphi_j^a\bar{\phi}
-2i(A_a\cdot\nabla\bar{\phi}\varphi_j^a+A_b\cdot\nabla\varphi_j^a\bar{\phi}) \right.\\ & \left.
 +(|A_a|^2-|A_b|^2)\varphi_j^a\bar{\phi}\right\}\,dx 
=\int_{\Omega} \left\{ (i \nabla + A_{b})^{2}\varphi_{j}^{a}\bar{\phi} - i\nabla\cdot(A_a-A_b) \varphi_{j}^{a} \bar{\phi}  \right.\\ & \left.
-2i  \varphi_{j}^{a} (A_a-A_b)\cdot\nabla\bar{\phi} + (|A_a|^2-|A_b|^2) \varphi_{j}^{a} \bar{\phi}  \right\}\, dx,
\end{split}
\]
where in the last step we used the identity
\[
-2i\int_\Omega A_b\cdot\nabla\varphi_j^a\bar{\phi}\,dx=
2i\int_\Omega(\nabla\cdot A_b \varphi_j^a\bar{\phi}+A_b\varphi_j^a\nabla\bar{\phi})\,dx.
\]
Since $a,b\not\in \text{supp}\{\phi\}$ then $A_a\to A_b$ in $C^\infty(\text{supp}\{\phi\})$. Hence for a suitable subsequence we can pass to the limit in the previous expression obtaining
\[
\int_{\Omega} (i \nabla + A_{b})^{2} \tilde{\varphi}_{j} \bar{\phi} = \int_{\Omega} \lambda_{j}^{\infty} \tilde{\varphi}_{j} \bar{\phi} \quad \text{ for every } \phi \in C^{\infty}_{0}(\Omega \backslash \{b\})
\]
where $\lambda_{j}^{\infty}:=\liminf_{a \to b} \lambda_{j}^{a}$. By density, the same is valid for $\phi\in H^1_{A_b}(\Omega)$. As a consequence of the last equation and of \eqref{eq:interior_orthogonality_L2}, the functions $\tilde{\varphi}_{j}$ are orthogonal in $H^{1}_{A_{b}}(\Omega)$ and hence
\begin{align*}
\lambda_{k}^{b} & = \inf_{\substack{W_k\subset H^{1}_{A_{b}}(\Omega)\\ \dim W_{k} = k}} \sup_{\Phi \in W_{k}} \frac{\int_{\Omega} |(i \nabla + A_{b}) \Phi|^{2}}{\int_{\Omega} |\Phi|^{2}} \\
                & \leq \sup_{<\tilde{\varphi}_{j}>} \frac{\int_{\Omega} |(i \nabla + A_{b}) \sum_{j=1}^{k} \alpha_{j}\tilde{\varphi}_{j} |^{2}}{|\sum_{j=1}^{k} \alpha_{j}\tilde{\varphi}_{j}|^{2}} \\
                & = \sup _{<\tilde{\varphi}_{j}>} \frac{\int_{\Omega} \sum_{j=1}^{k} \alpha_{j}^{2} \lambda_j^\infty}{\sum_{j=1}^{k} \alpha_{j}^2} \\
                &\leq \lambda_k^\infty= \liminf_{a \to b} \lambda_{k}^{a} .
\end{align*}
This concludes Step 2 and the proof of the theorem.
\end{proof}

\section{Continuity of the eigenvalues with respect to the pole up to the boundary of the domain} \label{sec:continuity2}

In this section we prove the second part of Theorem \ref{theorem:continuity}, that is the continuous extension up to the boundary of the domain.
We will denote by $\varphi_j$ an eigenfunction associated to $\lambda_j$, the $j$-th eigenvalue of the Laplacian in $H^1_0(\Omega)$. As usual, we suppose that the eigenfunctions are normalized in $L^2$ and orthogonal.
The following two lemmas can be proved exactly as the corresponding ones in the previous section.

\begin{lemma}\label{lemma:boundary_theta_a_theta_b}
Given $a\in \Omega$ and $b\in\partial\Omega$ there exist $\theta_a$ and $\theta_b$ such that  $\theta_a \in C^\infty(\Omega\setminus\{ta+(1-t)b,\, t\in[0,1]\})$, $\theta_b \in C^\infty(\Omega)$ and moreover in the respective sets of regularity the following holds
\[
\alpha\nabla\theta_a=A_a \qquad \alpha\nabla\theta_b=A_b.
\]
\end{lemma}

\begin{lemma}\label{lemma:boundary_tilde_varphi_j_definition}
Given $a\in \Omega$ and $b\in\partial\Omega$, let $\eta_a$ be defined in Lemma \ref{lemma:cut_off_function} and let $\theta_a$ be defined in Lemma \ref{lemma:interior_theta_a_theta_b}. Set, for $j=1,\ldots,k$,
\[
\tilde{\varphi}_{j} = e^{i\alpha\theta_{a}} \eta_{a} \varphi_{j}.
\]
Then for every $(\alpha_1,\ldots,\alpha_k)\in\R^k$ it holds
\[
(1 - \varepsilon_{a}) \|\sum_{j = 1}^{k} \alpha_{j} \varphi_{j}\|^{2}_{L^{2}(\Omega)} \leq 
\|\sum_{j = 1}^{k} \alpha_{j} \tilde{\varphi}_{j} \|_{L^{2}(\Omega)}^{2} \leq 
k \|\sum_{j = 1}^{k} \alpha_{j} \varphi_{j}\|^{2}_{L^{2}(\Omega)},
\]
where $\varepsilon_a\to 0$ as $a\to b$.
\end{lemma}

\begin{theorem}\label{theorem:continuityboundary}
Suppose that $a \in \Omega$ converges to $b\in \partial \Omega$. Then for every $k\in \N\setminus\{0\}$ we have that $\lambda_{k}^{a}$ converges to $\lambda_{k}$.
\end{theorem}
\begin{proof}
Following the scheme of the proof of Theorem \ref{theorem:continuityinterior} we proceed in two steps.

{\bf Step 1} First we show that
\begin{equation}\label{eq:boundary_step1_limsup}
\limsup_{a \to b} \lambda_{k}^{a} \leq \lambda_{k}.
\end{equation}
Since the proof is very similar to the one of Step 1 in Theorem \ref{theorem:continuityinterior} we will only point out the main differences. We define
\[
E_k:=\left\{\Phi=\sum_{j=1}^k\alpha_j\tilde{\varphi}_{j},\ \alpha_j\in\R\right\} \quad\text{ with }\quad
\tilde{\varphi}_{j} = e^{i\alpha\theta_{a}} \eta_{a} \varphi_{j}.
\]
We can verify the equality
\[
\left(i \nabla + A_{a} \right) ( e^{i\alpha\theta_{a}} \eta_{a} \varphi_{j} ) = 
i e^{i\alpha\theta_{a}} \nabla ( \eta_{a} \varphi_{j}),
\]
so that we have
\begin{align*}
\| \Phi \|_{H^{1}_{A_{a}}(\Omega)}^2  = \int_{\Omega} |\sum_{j=1}^{k} \alpha_{j} \nabla (\eta_{a}\varphi_{j}) |^{2} \, \, dx 
\leq \lambda_{k} \| \sum_{j = 1}^{k} \alpha_{j} \varphi_{j} \|^{2}_{L^{2}(\Omega)} + \beta_{a},
\end{align*}
with
\[
\beta_{a}  = \sum_{i,j = 1}^{k} \alpha_i\alpha_j\left( \int_{\Omega} |\nabla \eta_{a}|^{2} \varphi_{i} \varphi_{j} + 2 \eta_{a} \nabla \eta_{a}\cdot \nabla \varphi_{j} \varphi_{i} + ( \eta_{a}^{2}-1) \nabla \varphi_{i}\cdot \nabla \varphi_{j} \right) \, dx.
\]
Proceeding similarly to the proof of Theorem \ref{theorem:continuityinterior} we can estimate
\[
|\beta_a| \leq \varepsilon''_{a} \| \sum_{j = 1}^{k} \alpha_{j} \varphi_{j} \|^{2}_{L^{2}(\Omega)},
\]
with $\varepsilon''_a\to 0$ as $a\to b$. In conclusion, using Lemma \ref{lemma:boundary_tilde_varphi_j_definition}, we have obtained
\[
\| \Phi \|_{H^{1}_{A_{a}}(\Omega)}^2\leq \left(\lambda_k+\frac{\varepsilon''_a+\lambda_k\varepsilon_a}{1-\varepsilon_a}\right) \| \Phi \|_{L^2(\Omega)}^2 \quad \text{ for every } \Phi\in E_k,
\]
with $\varepsilon_a,\varepsilon''_a\to 0$ as $a\to b$. Therefore \eqref{eq:boundary_step1_limsup} is proved.

{\bf Step 2} We will now prove the second inequality
\begin{align*}
\liminf_{a \to b} \lambda_{k}^{a} \geq \lambda_{k}.
\end{align*}
Given a test function $\phi \in C^{\infty}_{0}(\Omega)$, for $a$ sufficiently close to $b$ we have that
\[
\{ta+(1-t)b, \ t\in[0,1]\} \subset \Omega\setminus\{\text{supp}\phi\}.
\]
Then $\phi\in H^1_{A_a}(\Omega)$ and Lemma \ref{lemma:boundary_theta_a_theta_b} implies that $e^{i\alpha\theta_a}\phi \in C_0^\infty(\Omega)$.
For this reason we can compute the following
\begin{equation}\label{eq:boundary_step2_testing}
\int_{\Omega} \nabla (e^{-i \alpha\theta_{b}} \varphi_{j}^{a})\cdot\nabla \bar{\phi} \, dx = \int_{\Omega} e^{-i \alpha\theta_{b}} \varphi_{j}^{a} (\overline{-\Delta( e^{-i \alpha\theta_{a}} \phi e^{i \alpha\theta_{a}})} ) \, dx.
\end{equation}
Since
\[
-\Delta( e^{-i \alpha\theta_{a}} \phi e^{i \alpha\theta_{a}})=(i\nabla+A_a)^2\phi-2iA_a\cdot\nabla\phi-i\nabla\cdot A_a\phi-|A_a|^2\phi,
\]
the right hand side in \eqref{eq:boundary_step2_testing} can be rewritten as
\[
\int_{\Omega} \left( (i\nabla+A_a)^2(e^{-i \alpha\theta_{b}} \varphi_{j}^{a})\bar{\phi} +
e^{-i \alpha\theta_{b}} \varphi_{j}^{a} (2iA_a\cdot \nabla\bar{\phi}+i\nabla\cdot A_a\bar{\phi}-|A_a|^2\bar{\phi}) \right) \, dx.
\]
At this point notice that
\begin{align*}
(i\nabla+A_a)^2(e^{-i \alpha\theta_{b}} \varphi_{j}^{a})= e^{-i \alpha\theta_{b}} & \left(
(i\nabla+A_a)^2\varphi_{j}^{a}+i\nabla\cdot A_b\varphi_{j}^{a}+2iA_b\cdot\nabla\varphi_{j}^{a} \right.\\
&\left.+|A_b|^2\varphi_{j}^{a}+2A_a\cdot A_b\varphi_{j}^{a} \right).
\end{align*}
By inserting these information in \eqref{eq:boundary_step2_testing} we obtain
\begin{equation}\label{eq:boudary_step2_testing_beta_a}
\int_{\Omega} \nabla (e^{-i \alpha\theta_{b}} \varphi_{j}^{a})\cdot\nabla \bar{\phi} \, dx = \lambda_j^a\int_\Omega e^{-i \alpha\theta_{b}} \varphi_{j}^{a}\bar{\phi} \, dx +\beta_a,
\end{equation}
with
\begin{align*}
\beta_a & =\int_\Omega e^{-i \alpha\theta_{b}} \bar{\phi}  \left(
i\nabla\cdot A_b\varphi_{j}^{a} + 2iA_b\cdot\nabla\varphi_{j}^{a} +|A_b|^2\varphi_{j}^{a}+2A_a\cdot A_b\varphi_{j}^{a} \right)\, dx \\
& + \int_\Omega e^{-i \alpha\theta_{b}} \varphi_{j}^{a}  \left( 2iA_a\cdot \nabla\bar{\phi}+i\nabla\cdot A_a\bar{\phi}-|A_a|^2\bar{\phi} \right) \, dx.
\end{align*}
Integration by parts leads
\[
\beta_a=\int_\Omega e^{-i \alpha\theta_{b}} \varphi_{j}^{a} \left(
-\bar{\phi}|A_a-A_b|^2+2i\nabla\bar{\phi}\cdot(A_a-A_b)+i\bar{\phi}\nabla\cdot(A_a-A_b)
\right)\, dx,
\]
so that $|\beta_a|\to0$ as $a\to b$, since $A_a\to A_b$ in $C^\infty(\text{supp}\{\phi\})$.
%
%
Therefore we can pass to the limit in \eqref{eq:boudary_step2_testing_beta_a} to obtain
\[
\int_{\Omega} \nabla  \tilde{\varphi}_{j}\cdot\nabla \bar{\phi} \, dx = \lambda_j^\infty \int_\Omega \tilde{\varphi}_{j}\bar{\phi} \, dx
\quad \text{ for every } \phi \in C^{\infty}_{0}(\Omega),
\]
where $\tilde{\varphi}_j$ is the weak limit of a suitable subsequence of $e^{-i \alpha\theta_{b}} \varphi_{j}^{a}$ (which exists by Step 1) and $\lambda_j^\infty:=\liminf_{a \to b} \lambda_{j}^{a}$.
The conclusion of the proof is as in Theorem \ref{theorem:continuityinterior}.
\end{proof}

\section{Differentiability of the simple eigenvalues with respect to the pole} \label{sec:differentiability}

In this section we prove Theorem \ref{theorem:differentiability_simple_eigenvalue}. We omit the subscript in the notation of the eigenvalues and eigenfunctions; with this notation, $\lambda^a$ is any eigenvalue of $(i\nabla+A_a)^2$ and $\varphi^a$ is an associated eigenfunction.

\begin{proof}[Proof of Theorem \ref{theorem:differentiability_simple_eigenvalue}]
Let $b\in\Omega$ be such that $\lambda^b$ is simple, as in the assumptions of the theorem.
For $R$ such that $B_{2R}(b)\subset\Omega$, let $\xi$ be a cut-off function satisfying $\xi\in C^\infty(\Omega)$, $0\leq\xi\leq1$, $\,\xi(x)=1$ for $x\in B_R(b)$ and $\xi(x)=0$ for $x\in \Omega\setminus B_{2R}(b)$. For every $a\in B_R(b)$ we define the transformation
\[
\Phi_a:\Omega\to\Omega, \qquad \Phi_a(x)=\xi(x)(x-b+a)+(1-\xi(x))x.
\]
Then $\varphi^a\circ\Phi_a \in H^1_{A_b}(\Omega)$ and satisfies, for every $a\in B_R(b)$,
\begin{equation}\label{eq:differentiability_v_b}
(i\nabla+A_b)^2(\varphi^a\circ\Phi_a) + \mathcal{L} (\varphi^a\circ\Phi_a)=\lambda^a \varphi^a\circ\Phi_a
\end{equation}
and
\begin{equation}\label{eq:differentiability_normalization}
\int_{\Omega} |\Phi'_a|^{2} |\varphi^{a} \circ \Phi_a|^2\,dx=1,
\end{equation}
where $\mathcal{L}$ is a second-order operator of the form
\[
\mathcal{L} v=-\sum_{i,j=1}^2 a^{ij}(x) \frac{\partial^2 v}{\partial x_i\partial x_j} 
+\sum_{i=1}^2 b^i(x) \frac{\partial v}{\partial x_i} + c(x) v,
\]
with $a^{ij},b^i,c \in C^\infty(\Omega, \C)$ vanishing in $B_R(b)$ and outside of $B_{2R}(b)$.
Notice that
\[
\Phi_a'(x)=I+\nabla \xi(x)\otimes(a-b)
\]
is a small perturbation of the identity whenever $\vert b-a\vert$ is sufficiently small, so that the operator in the left hand side of \eqref{eq:differentiability_v_b} is elliptic (see for example \cite[Lemma 9.8]{brezis2010functional}). 

To prove the differentiability, we will use the implicit function theorem in Banach spaces. To this aim, we define the operator
\begin{align}\label{eq:F_definition_implicit_function}
\begin{split}
& F : B_R(b) \times H^{1}_{A_{b}}(\Omega) \times \mathbb{R} \rightarrow (H^{1}_{A_{b}}(\Omega))^{\prime} \times \mathbb{R} \\
& (a, v, \lambda) \mapsto ((i \nabla + A_b)^{2} v +\mathcal{L}v - \lambda v, \int_{\Omega} |\Phi_a'|^{2} |v|^{2}\,dx - 1).
\end{split}
\end{align}
Notice that $F$ is of class $C^\infty$ by the ellipticity of the operator, provided that $R$ is suficiently small, and that $F(a,\varphi^a\circ\Phi_a,\lambda^a)=0$ for every $a\in B_R(b)$, as we saw in \eqref{eq:differentiability_v_b}, \eqref{eq:differentiability_normalization}. In particular we have $ F(b, \varphi^{b}, \lambda^{b}) = 0$,
since $\Phi_b$ is the identity.
We now have to verify that the differential of $F$ with respect to the variables $(v, \lambda)$, evaluated at the point $(b,\varphi^b,\lambda^{b})$, that we denote by $\mathrm{d}_{(v,\lambda)} F (b,\varphi^b,\lambda^{b})$, belongs to $Inv(H^{1}_{A_{b}}(\Omega)\times \mathbb{R}, (H^{1}_{A_{b}}(\Omega))^{\prime} \times \mathbb{R})$.
The differential is given by
\begin{align*}
\mathrm{d}_{(v,\lambda)} F (b,\varphi^b,\lambda^{b}) =
\begin{pmatrix}
& (i \nabla +  A_{b})^{2} - \lambda^{b} Im  & -  \varphi^b \\
& 2 \int_{\Omega} \overline{\varphi}^b \,dx  & 0 \\
\end{pmatrix},
\end{align*}
where $Im$ is the compact immersion of $H^{1}_{A_{b}}(\Omega)$ in $(H^{1}_{A_{b}}(\Omega))^{\prime}$, which was introduced in Lemma \ref{lemma:compactnessinversemagneticlaplacian}.

Let us first prove that it is injective. To this aim we have to show that, if $(w,s) \in H^{1}_{A_{b}}(\Omega)\times \mathbb{R}$ is such that
\begin{align}
(i \nabla + A_{b})^{2} w - \lambda^{b} w = s \varphi^b \label{eq:differentiability_injective1}\\
2 \int_{\Omega} \overline{\varphi}^b w \,dx = 0, \label{eq:differentiability_injective2}
\end{align}
then $(w,s)=(0,0)$. Relations \eqref{eq:differentiability_injective2} and \eqref{eq:differentiability_normalization} (with $a=b$ and $\Phi_b$ the identity) imply that
\begin{equation}\label{eq:differentiability_injective3}
w \neq k \varphi^b \quad\text{ for all } k\neq 0.
\end{equation}
By testing \eqref{eq:differentiability_injective1} by $\varphi^b$ we obtain
\[
s=\int_{\Omega} ((i \nabla + A_{b}) w \cdot \overline{(i \nabla + A_{b}) \varphi^b} -\lambda^{b} w \overline{\varphi^b} )\,dx.
\]
On the other hand, testing by $w$ the equation satisfied by $\varphi^b$, we see that
$s=0$, so that \eqref{eq:differentiability_injective1} becomes
\[
(i \nabla + A_{b})^{2} w = \lambda^{b} w.
\]
The assumption $\lambda^{b}$ simple, together with \eqref{eq:differentiability_injective3}, implies $w=0$. This concludes the proof of the injectivity.

For the surjectivity, we have to show that for all $(f,r) \in (H^{1}_{A_{b}}(\Omega))^{\prime} \times  \mathbb{R}$ there exist $(w,s)\in H^{1}_{A_{b}}(\Omega)\times \mathbb{R}$ which verifies the following equalities
\begin{align}
(i \nabla + A_{b})^{2} w - \lambda^{b}  w & = f + s \varphi^b \label{eq:differentiability_surjective1}\\
2 \int_{\Omega} \overline{\varphi}^b w \,dx & = r . \label{eq:differentiability_surjective2}
\end{align}
We recall that the operator $(i \nabla + A_{b})^{2} -\lambda^{b} Im : H^{1}_{A_{b}}(\Omega)\to (H^{1}_{A_{b}}(\Omega))^{\prime}$ is Fredholm of index 0. This a standard fact, which can be proved for example noticing that this operator is isomorphic to $Id-\lambda^{b}((i \nabla + A_{b})^{2})^{-1}(Im)$ through the Riesz isomorphism and because the operator $(i \nabla + A_{b})^{2}$ is invertible. This is Fredholm of index 0 because it has the form identity minus compact, the compactness coming from Lemma \ref{lemma:compactnessinversemagneticlaplacian}. Therefore we have (through Riesz isomorphism)
\begin{equation} \label{eq:differentiability_fredholm}
\text{Rank}((i \nabla + A_{b})^{2}-\lambda^{b} Im)=(\text{Ker}((i \nabla + A_{b})^{2} -\lambda^{b} Im))^{\perp}
=(\text{span}\{ \varphi^b \})^{\perp},
\end{equation}
where we used the assumption $\lambda^{b}$ simple in the last equality. As a consequence, we obtain from \eqref{eq:differentiability_surjective1} an expression for $s$
\[
s=-\int_{\Omega} f \overline{\varphi^b} \,dx.
\]
Next we can decompose $w$ in $w_{0} + w_{1}$ such that $w_{0} \in \text{Ker}((i \nabla + A_{b})^{2} -\lambda^{b} Im)$ and $w_{1}$ is in the orthogonal space. Condition \eqref{eq:differentiability_surjective1} becomes
\begin{align}
(i \nabla + A_{b})^{2} w_{1} - \lambda^{b} w_{1} = f - \varphi^b \int_{\Omega} f \overline{\varphi^b} \,dx
\end{align}
and \eqref{eq:differentiability_fredholm} ensures the existence of a solution $w_1$. Given such $w_1$, condition \eqref{eq:differentiability_surjective2} determines $w_0$ as follows
\[
w_0=\left(-\int_{\Omega}  \overline{\varphi}^b  w_1 \,dx+\frac{r}{2}\right) \varphi^b,
\]
so that the surjectivity is also proved.

We conclude that the implicit function theorem applies, so that the maps $a\in \Omega \mapsto \lambda^{a}\in\R$ and $a\in\Omega \mapsto \varphi^a \circ \Phi_a \in H^1_{A_b}(\Omega)$ are of class $C^\infty$ locally in a neighbourhood of $b$.
\end{proof}

By combining the previous result with a standard lemma of local inversion we deduce the following fact, which we will need in the next section.

\begin{corollary}\label{corollary:local_inversion}
Let $b\in\Omega$. If $\lambda^b$ is simple then the map $\Psi :  \Omega \times H^{1}_{A_{b}}(\Omega) \times \mathbb{R} \rightarrow  \mathbb{R} \times (H^{1}_{A_{b}}(\Omega))^{\prime} \times \mathbb{R}$ given by
\[
\Psi(a,v,\lambda)= (a,F(a,v,\lambda)),
\]
with $F$ defined in \eqref{eq:F_definition_implicit_function}, is locally invertible in a neighbourhood of $(b,\varphi^b,\lambda^b)$, with inverse $\Psi^{-1}$ of class $C^\infty$.
\end{corollary}
\begin{proof}
We saw in the proof of Theorem \ref{theorem:differentiability_simple_eigenvalue} that, if $\lambda^b$ is simple, then $\mathrm{d}_{(v,\lambda)} F (b,\varphi^b,\lambda^b)$ is invertible. It is sufficient to apply Lemma 2.1 in Chapter 2 of the book of Ambrosetti and Prodi \cite{AP1}.
\end{proof}


\section{Vanishing of the derivative at a multiple zero} \label{sec:comportementeigenvalues}

In this section we prove Theorem \ref{theorem:comportmentofeigenvelue}. Recall that here $\alpha=1/2$. We will need the following preliminary results.

\begin{lemma}\label{lemma:estimate_on_normal_derivative}
Let $\lambda>0$ and let $D_r=D_r(0)\subset \R^2$. Consider the following set of equations for $r>0$ small
\begin{equation}\label{eq:estimate_normal_derivative}
\left\{\begin{array}{ll}
        -\Delta u=\lambda u \quad &\text{ in } D_r \\
		u=r^{k/2}f+g(r,\cdot) \quad &\text{ on } \partial D_r,
       \end{array}\right.
\end{equation}
where $f,g(r,\cdot) \in H^{1}(\partial D_r)$ and $g$ satisfies
\begin{equation}\label{eq:estimate_normal_derivative_assumption}
\lim_{r\to0} \frac{\|g(r,\cdot)\|_{H^{1}(\partial D_r)}}{r^{k/2}}=0 
\end{equation}
for some integer $k\geq3$. Then for $r$ sufficiently small there exists a unique solution to \eqref{eq:estimate_normal_derivative}, which moreover satisfies
\[
\|u\|_{L^2(D_r)}\leq C r^{(k+2)/2} \quad\text{ and }\quad
\left\|\frac{\partial u}{\partial \nu}\right\|_{L^2(\partial D_r)}\leq C r^{(k-1)/2},
\]
where $C>0$ is independent of $r$.
\end{lemma}
\begin{proof}
Let $z_1$ solve
\[
\left\{\begin{array}{ll}
        -\Delta z_1=0 \quad &\text{ in } D_1 \\
		z_1=f+r^{-k/2}g(r,\cdot) \quad &\text{ on } \partial D_1.
       \end{array}\right.
\]
Since the quadratic form
\begin{equation}\label{eq:quadratic_form_coercive}
\int_{D_1}\left(|\nabla v|^2-\lambda r^2 v^2\right)\,dx
\end{equation}
is coercive for $v \in H^1_0(D_1)$ for $r$ sufficiently small, there exists a unique solution $z_2$ to the equation
\begin{equation}\label{eq:z_2_equation}
\left\{\begin{array}{ll}
        -\Delta z_2-\lambda r^2 z_2=\lambda r^2 z_1 \quad &\text{ in } D_1 \\
		z_2=0 \quad &\text{ on } \partial D_1.
       \end{array}\right.
\end{equation}
Then $u(x)=r^{k/2}(z_1(x/r)+z_2(x/r))$ is the unique solution to \eqref{eq:estimate_normal_derivative}. In order to obtain the desired bounds on $u$ we will estimate separately $z_1$ and $z_2$. Assumption \eqref{eq:estimate_normal_derivative_assumption} implies
\begin{align}\label{eq:estimatez1}
\|z_1\|_{H^1(D_1)} = \|f+r^{-k/2}g(r,\cdot)\|_{H^{1/2}(\partial D_1)}\leq C \|f\|_{H^{1}(\partial D_1)},
\end{align}
for $r$ sufficiently small.
We compare the function $z_1$ to its limit function when $r\to0$, which is the harmonic extension of $f$ in $D_1$, which we will denote $w$. Then we have
\[
\left\{\begin{array}{ll}
        -\Delta (z_1-w)=0 \quad &\text{ in } D_1 \\
		z_1-w=r^{-k/2}g(r,\cdot) \quad &\text{ on } \partial D_1,
       \end{array}\right.
\]
and hence \eqref{eq:estimate_normal_derivative_assumption} implies
\[
\left\|\frac{\partial }{\partial \nu}(z_1-w)\right\|_{L^2(\partial D_1)}\leq C\|z_1-w\|_{H^1(\partial D_1)}=
C\frac{\|g(r,\cdot)\|_{H^1(\partial D_1)}}{r^{k/2}}\to0.
\]
Then we estimate $z_2$ as follows
\begin{align*}
\|z_2 \|_{L^2(D_1)}^2 & \leq C \int_{D_1} |\nabla z_2 |^2 \, dx 
\leq C \int_{D_1} (|\nabla z_2 |^2 - \lambda r^2 z_2^2 )\, dx \\
& \leq C \|\lambda r^2 z_1\|_{L^2(D_1)} \|z_2\|_{L^2(D_1)},
\end{align*}
where we used Poincar\'e inequality, the coercivity of the quadratic form \eqref{eq:quadratic_form_coercive} and the defintion of $z_2$ \eqref{eq:z_2_equation}. Hence estimate \eqref{eq:estimatez1} implies
\[
\|z_2 \|_{L^2(D_1)} \leq C r^2 \|f\|_{H^{1}(\partial D_1)} \to 0 \quad \text{as } r\to0.
\]
This and \eqref{eq:estimatez1} provide, by a change of variables in the integral, the desired estimate on $\|u\|_{L^2(D_r)}$. Now, the standard bootstrap argument for elliptic equations applied to \eqref{eq:z_2_equation} provides
\begin{align*}
\| z_2 \|_{H^2(D_1)} \leq C ( \| \lambda r^2 z_1 \|_{L^2(D_1)} + \| z_2 \|_{L^2(D_1)} ) \to 0,
\end{align*}
and hence the trace embedding implies
\begin{align*}
 \left\| \frac{\partial z_2}{\partial \nu} \right\|_{L^2(\partial D_1)} \leq C \| \nabla z_2 \|_{H^1(D_1)} \leq C \| z_2 \|_{H^2(D_1)} \to 0.
\end{align*}
So, we have obtained that there exists $C > 0$ independent of $r$ such that
\begin{align*}
\left\| \frac{\partial}{\partial \nu} (z_1 + z_2) \right\|_{L^2(\partial D_1)} \leq C.
\end{align*}
Finally, going back to the function $u$, we have
\begin{align*}
\left\| \frac{\partial u}{\partial \nu} \right\|_{L^2(\partial D_r)}
=r^{(k-1)/2} \left\| \frac{\partial}{\partial \nu} (z_1 + z_2) \right\|_{L^2(\partial D_1)} \leq Cr^{(k-1)/2}
\end{align*} 
where we used the change of variable $x = ry$.
\end{proof}

\begin{lemma} \label{lemma:uniformbound}
Let $\phi\in H^1_{A_a}(\Omega)$ ($a\in\Omega$). Then
\begin{align}\label{eq:L_2_norm_boundary}
\frac{1}{|a|^{1/2}} \|\phi\|_{L^2(\partial D_{|a|})} \leq C \| \phi \|_{H^{1}_{A_a}(\Omega)}
\end{align}
where $C$ only depends on $\Omega$.
\end{lemma}
\begin{proof}
Set $\tilde{\phi}(y)=\phi(|a|y)$ defined for $y\in\tilde{\Omega}=\{ x/|a|:\ x\in\Omega \}$. We apply this change of variables to the left hand side in \eqref{eq:L_2_norm_boundary} and then use the trace embedding to obtain
\[
\frac{1}{|a|^{1/2}} \|\phi\|_{L^2(\partial D_{|a|})}=\|\tilde{\phi}\|_{L^2(\partial D_1)}\leq 
C \|\tilde{\phi}\|_{H^1(D_1)} \leq C \|\tilde{\phi}\|_{H^1(D_2)} .
\]
We have that $\tilde{\phi}\in H^1_{A_{e}}(\tilde{\Omega})$, where $e=a/|a|$. Therefore we can apply relation \eqref{eq:hardy} as follows
\[
\|\tilde{\phi}\|_{L^2(D_2)}\leq 
\|y-e\|_{L^\infty(D_2)}\left\|\frac{\tilde{\phi}}{|y-e|} \right\|_{L^2(D_2)}
\leq C\|(i\nabla+A_{e})\tilde{\phi} \|_{L^2(D_2)},
\]
\[
\begin{split}
\|\nabla \tilde\phi\|_{L^2(D_2)} &\leq \|(i\nabla+A_e)\tilde\phi\|_{L^2(D_2)} + \|A_a\tilde\phi\|_{L^2(D_2)} \\
&\leq \|(i\nabla+A_e)\tilde\phi\|_{L^2(D_2)} + \|(y-e)A_e\|_{L^\infty(D_2)} \left\|\frac{\tilde{\phi}}{|y-e|} \right\|_{L^2(D_2)} \\
& \leq C\|(i\nabla+A_{e})\tilde{\phi} \|_{L^2(D_2)}.
\end{split}
\]
We combine the previous inequalities obtaining
\[
\begin{split}
\frac{1}{|a|^{1/2}} \|\phi\|_{L^2(\partial D_{|a|})} 
\leq C \|(i\nabla+A_{e})\tilde{\phi} \|_{L^2(D_2)} 
\leq C \| \phi \|_{H^{1}_{A_a}(\Omega)},
\end{split}
\]
where in the last step we used the fact that the quadratic form is invariant under dilations.
\end{proof}

To simplify the notations we suppose without loss of generality that $0\in\Omega$ and we take $b=0$. Moreover, we omit the subscript in the notation of the eigenvalues as we did in the previous section.
As a first step in the proof of Theorem \ref{theorem:comportmentofeigenvelue}, we shall estimate $|\lambda^a-\lambda^0|$ in the case when the pole $a$ belongs to a nodal line of $\varphi^{0}$ ending at $0$. We make this restriction because all the constructions in the following proposition require that $\varphi^0$ vanishes at $a$.

\begin{proposition} \label{proposition:comportmenteigenvalues}
Suppose that $\lambda^{0}$ is simple and that $\varphi^{0}$ has a zero of order $k/2$ at the origin, with $k \geq 3$ odd. Denote by $\Gamma$ a nodal line of $\varphi^0$ with endpoint at $0$ (which exists by Proposition \ref{proposition:asymptotic_expansion_eigenfunction}) and take $a\in\Gamma$. Then there exists a constant $C>0$ independent of $|a|$ such that
\begin{align*}
|\lambda^{a} - \lambda^{0}|  \leq C |a|^{k/2} \quad \text{ as }|a| \rightarrow 0, \ a\in\Gamma.
\end{align*}
\end{proposition}
\begin{proof}
The idea of the proof is to construct a function $u_a\in H^1_{A_a}(\Omega)$ satisfying
\begin{equation}\label{eq:ua_definition}
(i \nabla + A_{a})^{2} u_{a} - \lambda^{0} u_{a} = g_{a}, \qquad \|u_a\|_{L^2(\Omega)}=1 - \epsilon_a
\end{equation}
with
\begin{equation}\label{eq:ga_epsa}
\|g_a\|_{(H^1_{A_a}(\Omega))'}\simeq |a|^{k/2} \qquad\text{ and }\qquad |\epsilon_a| \simeq |a|^{(k+2)/2}
\end{equation}
and then to apply the Corollary \ref{corollary:local_inversion}.
For the construction of the function $u_a$ we will heavily rely on the assumption $a\in\Gamma$.

{\bf Step 1:} construction of $u_a$. We define it separately in $D_{|a|}=D_{|a|}(0)$ and in its complement $\Omega \setminus D_{|a|}$, using the following notation
\begin{align}
u_{a} = \left\{ \begin{aligned} & u_{a}^{ext} \qquad  \Omega \backslash D_{|a|} & \\
                                & u_{a}^{int} \qquad  D_{|a|} .&
                                \end{aligned} \right.
\end{align}
Concerning the exterior function we set
\begin{align}\label{eq:u_a_ext_definition}
u_{a}^{ext} = e^{i \alpha (\theta_{a} - \theta_{0})} \varphi^{0},
\end{align}
where $\theta_a,\theta_0$ are defined as in Lemma \ref{lemma:interior_theta_a_theta_b} in such a way that $\theta_a-\theta_0$ is regular in $\Omega\setminus D_{|a|}$ (here $\theta_0=\theta$ is the the angle in the usual polar coordinates, but we emphasize the position of the singularity in the notation).  
Therefore $u_{a}^{ext}$ solves the following magnetic equation
\begin{align}\label{eq:u_a_ext_equation}
\left\{ \begin{aligned}  
&(i \nabla + A_{a})^{2} u_{a}^{ext}  = \lambda^{0} u_{a}^{ext} \qquad &  \Omega \backslash D_{|a|}  \\
& u_{a}^{ext}  = e^{i \alpha (\theta_{a} - \theta_{0})} \varphi^{0} \qquad  & \partial D_{|a|} \\
& u_{a}^{ext}  =0 \qquad  & \partial\Omega. 
         \end{aligned} \right.
\end{align}

For the definition of $u_a^{int}$ we will first consider a related elliptic problem. Notice that, by our choice $a\in\Gamma$, we have that $e^{-i\alpha\theta_0}\varphi^0$ is continuous on $\partial D_{|a|}$. Indeed, $e^{-i\alpha\theta_0}$ restricted to $\partial D_{|a|}$ is discontinuous only at the point $a$, where $\varphi^0$ vanishes. Moreover, note that this boundary trace is at least $H^1(\partial D_{|a|})$. Indeed, the eigenfunction $\varphi^{0}$ is $C^{\infty}$ far from the singularity and $e^{i \alpha \theta_{0}}$ is also regular except on the point $a$. Then, the boundary trace is differentiable almost everywhere.

This allows to apply Lemma \ref{lemma:estimate_on_normal_derivative}, thus providing the existence of a unique function $\psi_a^{int}$, solution of the following equation
\begin{align} \label{eq:psi_a_int}
\left\{ \begin{aligned} & - \Delta \psi_{a}^{int} = \lambda^{0} \psi_{a}^{int} \qquad & D_{|a|} \\
                        & \psi_{a}^{int} = e^{-i\alpha\theta_0}\varphi^0 \qquad & \partial D_{|a|}.
                        \end{aligned} \right.
\end{align}
Then we complete our construction of $u_a$ by setting
\begin{align}\label{eq:u_a_int_definition}
u_{a}^{int} = e^{i \alpha \theta_{a}} \psi_{a}^{int},
\end{align} 
which is well defined since $\theta_a$ is regular in $D_{|a|}$. Note that $u_a^{int}$ solves the following elliptic equation
\begin{align}\label{eq:u_a_int_equation}
\left\{ \begin{aligned}  
& (i \nabla + A_{a})^{2} u_{a}^{int}  = \lambda^{0} u_{a}^{int} \qquad  & D_{|a|} &\\                         
& u_{a}^{int}  = u_{a}^{ext} \qquad  & \partial D_{|a|}.
\end{aligned} \right.
\end{align}

{\bf Step 2:} estimate of the normal derivative of $u_a^{int}$ along $\partial D_{|a|}$. By assumption, $\varphi^0$ has a zero of order $k/2$ at the origin, with $k\geq3$ odd. Hence by Proposition \ref{proposition:asymptotic_expansion_eigenfunction} the following asymptotic expansion holds on $\partial D_{|a|}$ as $|a|\to0$
\begin{equation}\label{eq:asymptic_expansion_varphi_0}
e^{-i\alpha\theta_0}\varphi^0(|a|,\theta_0)=\frac{|a|^{k/2}}{k}[c_k \cos(k\alpha\theta_0)+ d_k \sin(k\alpha\theta_0)]+g(|a|,\theta_0),
\end{equation}
with
\begin{equation}\label{eq:asymptic_expansion_varphi_0_remainder}
\lim_{|a|\to0}\frac{\|g(|a|,\cdot)\|_{C^1(\partial D_{|a|})}}{|a|^{k/2}}=0.
\end{equation}
Hence Lemma \ref{lemma:estimate_on_normal_derivative} applies to $\psi_a^{int}$ given in \eqref{eq:psi_a_int}, providing the existence of a constant $C$ independent of $|a|$ such that
\begin{equation}\label{eq:asymptic_expansion_varphi_0_L2}
\|\psi_a^{int}\|_{L^2(D_{|a|})}\leq C |a|^{(k+2)/2} \quad\text{and}\quad
\left\|\frac{\partial \psi_a^{int}}{\partial\nu}\right\|_{L^2(\partial D_{|a|})}\leq C |a|^{(k-1)/2}.
\end{equation}
Finally, differentiating \eqref{eq:u_a_int_definition} we see that
\[
( i\nabla +A_a ) u_a^{int}=ie^{i\alpha\theta_a}\nabla\psi_a^{int},
\]
so that, integrating, we obtain the following $L^2$-estimate for the magnetic normal derivative of $u_a^{int}$ along $\partial D_{|a|}$
\begin{equation}\label{eq:estimate_normal_derivative_u_a_int}
\|(i\nabla+A_a)u_a^{int}\cdot \nu\|_{L^2(\partial D_{|a|})}\leq C|a|^{(k-1)/2}.
\end{equation}

{\bf Step 3:} estimate of the normal derivative of $u_a^{ext}$ along $\partial D_{|a|}$. We differentiate \eqref{eq:u_a_ext_definition} to obtain
\begin{align}\label{eq:differential_u_a_ext}
( i\nabla + A_a ) u_a^{ext} =A_0 u_a^{ext}+ie^{i\alpha(\theta_a-\theta_0)}\nabla\varphi^0. 
\end{align}
On the other hand, the following holds a.e. 
\[
\nabla\varphi^0=iA_0\varphi^0+e^{i\alpha\theta_0}\nabla(e^{-i\alpha\theta_0}\varphi^0),
\]
so that
\[
ie^{i\alpha(\theta_a-\theta_0)}\nabla\varphi^0=-A_0 u_a^{ext} +i e^{i\alpha\theta_a} \nabla(e^{-i\alpha\theta_0}\varphi^0).
\]
Combining the last equality with \eqref{eq:differential_u_a_ext} we obtain a.e.
\[
( i\nabla + A_a ) u_a^{ext}=ie^{i\alpha\theta_a}\nabla(e^{-i\alpha\theta_0}\varphi^0)
\]
and hence $|(i\nabla +A_a) u_a^{ext}|\leq C |a|^{k/2-1}$ on $\partial D_{|a|}$ a.e., for some $C$ not depending on $|a|$, by \eqref{eq:asymptic_expansion_varphi_0} and \eqref{eq:asymptic_expansion_varphi_0_remainder}. Integrating on $\partial D_{|a|}$ we arrive at the same estimate as for $u_a^{int}$, that is
\begin{equation}\label{eq:estimate_normal_derivative_u_a_ext}
\|(i\nabla+A_a)u_a^{ext}\cdot\nu\|_{L^2(\partial D_{|a|})}\leq C|a|^\frac{k-1}{2}.
\end{equation}

{\bf Step 4:} proof of \eqref{eq:ga_epsa}. 
We test equation \eqref{eq:u_a_ext_equation} with a test function $\phi\in H^1_{A_a}(\Omega)$ and apply the formula of integration by parts to obtain
\[
\begin{split}
\int_{\Omega\setminus D_{|a|}} \left\{(i\nabla+A_a)u_a^{ext}\overline{(i\nabla+A_a)\phi} -\lambda^0 u_a^{ext}\bar{\phi} \right\}\,dx \\
= i\int_{\partial D_{|a|}} (i\nabla+A_a)u_a^{ext}\cdot\nu \bar{\phi}\,d\sigma.
\end{split}
\]
Similarly, equation \eqref{eq:u_a_int_equation} provides
\[
\begin{split}
\int_{D_{|a|}} \left\{(i\nabla+A_a)u_a^{int}\overline{(i\nabla+A_a)\phi} -\lambda^0 u_a^{int}\bar{\phi} \right\}\,dx = \\
-i\int_{\partial D_{|a|}} (i\nabla+A_a)u_a^{int}\cdot\nu \bar{\phi} \,d\sigma.
\end{split}
\]
Then, we test the equation in \eqref{eq:ua_definition} with $\phi$, we integrate by parts and we replace the previous equalities to get
\[
\int_\Omega g_a\bar\phi \,dx= i\int_{\partial D_{|a|}}(i\nabla+A_a)(u_a^{ext}-u_a^{int})\cdot\nu \bar\phi d\sigma.
\]
To the previous expression we apply first the H\"older inequality and then the estimates obtained in the previous steps \eqref{eq:estimate_normal_derivative_u_a_int} and \eqref{eq:estimate_normal_derivative_u_a_ext} to obtain
\begin{align*}
\left| \int_\Omega g_a\bar\phi \,dx\right|  \leq  
\|(i\nabla+A_a)u_a^{int}\cdot\nu\|_{L^2(\partial D_{|a|})} \|\phi\|_{L^2(\partial D_{|a|})} \\
+ \|(i\nabla+A_a)u_a^{ext}\cdot\nu\|_{L^2(\partial D_{|a|})} \|\phi\|_{L^2(\partial D_{|a|})}
\leq C |a|^{(k-1)/2} \|\phi\|_{L^2(\partial D_{|a|})}.
\end{align*}
Finally, Lemma \ref{lemma:uniformbound} provides the desired estimate on $g_a$.
Then we estimate $\epsilon_a$ as follows. Since $\|u_a^{ext}\|_{L^2(\Omega\setminus D_{|a|})}= \|\varphi^0\|_{L^2(\Omega\setminus D_{|a|})}$ we have
\begin{align}
\left| \|u_a\|_{L^2(\Omega)} -1\right| = \left| \|u_a^{int}\|_{L^2(D_{|a|})}^2- \|\varphi^0\|_{L^2(D_{|a|})}^2 \right| \leq C|a|^{k+2},
\end{align}
where in the last inequality we used the fact that $\|\varphi^0\|_{L^2(D_{|a|})}^2 \leq C|a|^{k+2}$ by \eqref{eq:asymptic_expansion_varphi_0} and \eqref{eq:asymptic_expansion_varphi_0_remainder}, and that $\|u_a^{int}\|_{L^2(D_{|a|})}^2=\|\psi_a^{int}\|_{L^2(D_{|a|})}^2 \leq C|a|^{k+2}$, by \eqref{eq:asymptic_expansion_varphi_0_L2}.

{\bf Step 5:} local inversion theorem. To conclude the proof we apply the Corollary \ref{corollary:local_inversion}. Let $\Psi$ be the function defined therein (recall that here $b=0$). The construction that we did in the previous steps ensures that
\begin{align*}
\Psi(a, \varphi^{a}\circ\Phi_a, \lambda^{a}) & = (a, 0, 0) \\
\Psi(a, u_{a}\circ\Phi_a, \lambda^{0}) & = (a, g_{a}\circ\Phi_a, \epsilon_{a}),
\end{align*}
with $g_{a}$, $\epsilon_{a}$ satisfying \eqref{eq:ga_epsa}. We proved in Theorem \ref{theorem:continuityinterior} that 
\[
|\lambda^a-\lambda^0|+\|\varphi^a\circ\Phi_a-\varphi^0\|_{H^1_{A_0}(\Omega)}\to0
\]
as $|a|\to0$. Moreover, it is not difficult to see that
\[
\|u_a\circ\Phi_a-\varphi^0\|_{H^1_{A_0}(\Omega)}\to0
\]
as $|a|\to0$. Hence the points $(a,\varphi^a\circ\Phi_a,\lambda^a)$ and $(a,u_a\circ\Phi_a,\lambda^0)$ are approaching $(0,\varphi^0,\lambda^0)$ in the space $\Omega\times H^1_{A_0}(\Omega)\times\R$ as $|a|\to0$. Since $\Psi$ admits an inverse of class $C^\infty$ in a neighbourhood of $(0,\varphi^0,\lambda^0)$ (recall that $\lambda^0$ is simple), we deduce that
\begin{align*}
\| (\varphi^{a}- u_{a})\circ\Phi_a \|_{H^{1}_{A_0}(\Omega)} + | \lambda^{a} - \lambda^{0} | \leq C( \|g_{a}\|_{(H^1_{A_a}(\Omega))'}+|\epsilon_a|)
\leq C |a|^{k/2},
\end{align*}
for some constant $C$ independent of $a$, which concludes the proof of the proposition.
\end{proof}

At this point we have proved the desired property only for pole $a$ belonging to the nodal lines of $\varphi^{0}$. We would like to extend this result to all $a$ sufficiently closed to $0$. We will proceed in the following way.
Thanks to Theorem \ref{theorem:differentiability_simple_eigenvalue}, we can consider the Taylor expansion of the function $a\mapsto \lambda^a$ in a neighbourhood of 0. Then Proposition \ref{proposition:comportmenteigenvalues} provides $k$ vanishing conditions, corresponding to the $k$ nodal lines of $\varphi^{0}$. Finally, we will use these conditions to show that in fact the first terms of the polynome are identically zero. Let us begin with a lemma on the existence and the form of the Taylor expansion.

\begin{lemma}\label{lemma:taylordevelopment}
If $\lambda^{0}$ is simple then the following expansion is valid for $a\in\Omega$ sufficiently close to 0 and for all $H \in \mathbb{N}$
\begin{align}\label{eq:taylordevelopment}
\lambda^{a} - \lambda^{0} = \sum_{h=1}^H |a|^h P_{h}(\vartheta(a)) + o(|a|^{H}),
\end{align}
where $a=|a|(\cos\vartheta(a),\sin\vartheta(a))$ and
\begin{align}\label{eq:memberoftaylordevelopment}
P_{h}(\vartheta) = \sum_{j=0}^{h} \beta_{j,h} \cos^{j}\vartheta \sin^{h-j}\vartheta
\end{align}
for some $\beta_{j,h}\in\R$ not depending on $|a|$.
\end{lemma}
\begin{proof}
Since $\lambda^0$ is simple, $\lambda^a$ is also simple for $a$ sufficiently closed to 0.
Then we proved in Theorem \ref{theorem:differentiability_simple_eigenvalue} that $\lambda_{j}^a$ is $C^\infty$ in the variable $a$.
As a consequence, we can consider the first terms of the Taylor expansion, with Peano rest, of $\lambda_{j}^a$
\[
\lambda^a-\lambda^0=\sum_{h=1}^H \sum_{j=0}^h \frac{1}{j!(h-j)!}
\frac{\partial^h \lambda^a}{\partial^ja_1\partial^{h-j}a_2}\Big|_{a=0} a_1^j a_2^{h-j} + o(|a|^{H}),
\]
where $a=(a_1,a_2)$. Setting
\[
\beta_{j,h}=\frac{1}{j!(h-j)!}
\frac{\partial^h \lambda^a}{\partial^ja_1\partial^{h-j}a_2}\Big|_{a=0} 
\]
and $a_1=|a|\cos\vartheta(a)$, $a_2=|a|\sin\vartheta(a)$ the thesis follows.
\end{proof}

The following lemma tells us that on the $k$ nodal lines of $\varphi^{0}$, the first low-order polynomes cancel.

\begin{lemma} \label{lemma:annulationpolynomeonnodallines}
Suppose that $\lambda^0$ is simple and that $\varphi^0$ has a zero of order $k/2$ at $0$, with $k\geq3$ odd. Then there exist an angle $\tilde\vartheta \in [0,2\pi)$ and non-negative quantities $\varepsilon_0,\ldots,\varepsilon_{k-1}$ arbitrarily small such that
\[
P_h\left(\tilde{\vartheta}+\frac{2\pi l}{k}+\varepsilon_l\right)=0 \quad \text{for every integers } l\in [0,k-1], h\in \left[1,\frac{k-1}{2}\right]
\]
where $P_h$ is defined in \eqref{eq:memberoftaylordevelopment}.
\end{lemma}
\begin{proof}
We know from Proposition \ref{proposition:asymptotic_expansion_eigenfunction} that $\varphi^0$ has $k$ nodal lines with endpoint at 0, which we denote $\Gamma_l$, $l=0,\ldots, k-1$. Take points $a_l \in \Gamma_l$,  $l=0,\ldots, k-1$, satisfying $|a_0|=\ldots=|a_{k-1}|$ and denote
\[
a_l=|a_l|(\cos\vartheta(a_l),\sin\vartheta(a_l)).
\]
First we claim that
\[
P_h(\vartheta(a_l))=0 \quad\text{for every integers } l\in [0,k-1], h\in \left[1,\frac{k-1}{2}\right].
\]
Indeed, suppose by contradiction that this is not the case for some $l,h$ belonging to the intervals defined above. Then for such $l,h$ the following holds by Lemma \ref{lemma:taylordevelopment}
\[
\lambda^{a_l}-\lambda^0 = C |a_l|^h + o(|a_l|^h) \quad\text{for some } C\neq0.
\]
On the other hand we proved in Proposition \ref{proposition:comportmenteigenvalues} that there exists $C>0$ independent of $a$ such that, for every $l=0,\ldots,k-1$, we have
\[
|\lambda^{a_l}-\lambda^0|\leq C |a_l|^{k/2} \quad\text{as } |a_l|\to0.
\]
This contradicts the last estimate because $h\leq (k-1)/2$, so that the claim is proved.

Finally setting $\tilde{\vartheta}:=\vartheta(a_0)$, Proposition \ref{proposition:asymptotic_expansion_eigenfunction} implies
\[
\vartheta(a_l)=\tilde{\vartheta}+\frac{2\pi l}{k}+\varepsilon_l \qquad l=1,\ldots,k-1
\]
with $\varepsilon_l\to0$ as $|a_l|\to0$.
\end{proof}

The next lemma extends this the previous property to all $a$ close to $0$.

\begin{lemma} \label{lemma:annulationpolynome}
Fix $k\geq3$ odd. For any integer $h\in [1,(k-1)/2]$ consider any polynom of the form
\begin{align}\label{eq:polynom_P_h}
P_{h}(\vartheta) = \sum_{j=0}^{h} \beta_{j,h} \cos^{j}\vartheta \sin^{h-j}\vartheta,
\end{align}
with $\beta_{j,h}\in\R$. Suppose that there exist $\tilde{\vartheta}\in [0,2\pi)$ and $\varepsilon_0,\ldots,\varepsilon_{k-1}$ satisfying $0\leq\varepsilon_l\leq\pi/(4k)$ such that
\[
P_h\left(\tilde{\vartheta}+\frac{2\pi l}{k}+\varepsilon_l\right)=0 \quad \text{for every integer }
l\in[0,k-1].
\]
Then $P_h\equiv 0$.
\end{lemma}
\begin{proof}
We prove the result by induction on $h$. 

{\bf Step 1} Let $h=1$, then
\[
P_1(\vartheta)=\beta_0\sin\vartheta+\beta_1\cos\vartheta
\]
and the following conditions hold for $l=0,\ldots,k-1$:
\begin{equation}\label{eq:system_polynome_P_h}
\beta_0\sin\left(\tilde{\vartheta}+\frac{2\pi l}{k}+\varepsilon_l\right)+ \beta_1\cos\left(\tilde{\vartheta}+\frac{2\pi l}{k}+\varepsilon_l\right)=0.
\end{equation}
In case for every $l=0,\ldots,k-1$ we have
\[
\sin\left(\tilde{\vartheta}+\frac{2\pi l}{k}+\varepsilon_l\right)\neq 0 \quad \text{ and } \quad
\cos\left(\tilde{\vartheta}+\frac{2\pi l}{k}+\varepsilon_l\right) \neq 0,
\]
system \eqref{eq:system_polynome_P_h} has two unknowns $\beta_0,\beta_1$ and $k\geq3$ linearly independent equations. Hence in this case $\beta_0=\beta_1=0$ and $P_1\equiv0$. In case there exists $l$ such that
\[
\sin\left(\tilde{\vartheta}+\frac{2\pi l}{k}+\varepsilon_l\right) = 0
\]
then of course $\cos\left(\tilde{\vartheta}+2\pi l/k+\varepsilon_l\right) \neq 0$, which implies $\beta_1=0$. We claim that in this case
\begin{equation}\label{eq:sin_non_zero}
\sin\left(\tilde{\vartheta}+\frac{2\pi l'}{k}+\varepsilon_{l'}\right) \neq 0
\end{equation}
for every integer $l'\in [0,k-1]$ different from $l$. To prove the claim we proceed by contradiction. We can suppose without loss of generality that
\[
\tilde{\vartheta}+\frac{2\pi l}{k}+\varepsilon_l=0 \quad\text{and}\quad \tilde{\vartheta}+\frac{2\pi l'}{k}+\varepsilon_{l'}=\pi.
\]
Then
\[
l=-\frac{k}{2\pi}(\tilde{\vartheta}+\varepsilon_l) \quad\text{and}\quad
l'=\frac{k}{2\pi}(\pi-\tilde{\vartheta} - \varepsilon_{l'})
\]
so that
\[
l'-l=\frac{k}{2}+k\frac{\varepsilon_l-\varepsilon_{l'}}{2\pi}.
\]
The assumption $0\leq\varepsilon_l\leq\pi/(4k)$ implies
\[
\frac{k}{2}-\frac{1}{4}\leq l'-l \leq \frac{k}{2}+\frac{1}{4}.
\]
Being $k\geq3$ an odd integer, the last estimate provides $l'-l\not\in\N$, which is a contradiction. Therefore we have proved \eqref{eq:sin_non_zero}. Now consider any of the equations in \eqref{eq:system_polynome_P_h} for $l'\neq l$. Inserting the information $\beta_1=0$ and \eqref{eq:sin_non_zero} we get $\beta_0=0$ and hence $P_1\equiv0$. In case one of the cosine vanishes one can proceed in the same way, so we have proved the basis of the induction. 

{\bf Step 2} Suppose that the statement is true for some $h\leq (k-3)/2$ and let us prove it for $h+1$. The following conditions hold for $l=0,\ldots,k-1$
\begin{equation}
\sum_{j=0}^{h+1} \beta_j \cos^{j}\left(\tilde{\vartheta}+\frac{2\pi l}{k}+\varepsilon_l\right) \sin^{h+1-j}\left(\tilde{\vartheta}+\frac{2\pi l}{k}+\varepsilon_l\right)=0.
\end{equation}
We can proceed similarly to Step 1. If none of the sinus, cosinus vanishes then we have a system with $h+2\leq(k+1)/2$ unknowns and $k$ linearly independent equations, hence $P_{h+1}\equiv0$. Otherwise suppose that there exists $l$ such that
\[
\sin\left(\tilde{\vartheta}+\frac{2\pi l}{k}+\varepsilon_l\right) = 0.
\]
Then we saw in Step 1 that
\[
\cos\left(\tilde{\vartheta}+\frac{2\pi l}{k}+\varepsilon_l\right) \neq 0 \quad \text{ and }\quad \sin\left(\tilde{\vartheta}+\frac{2\pi l'}{k}+\varepsilon_{l'}\right) \neq 0
\]
for every integer $l'\in [0,k-1]$ different from $l$. By rewriting $P_{h+1}$ in the following form
\[
P_{h+1}(\vartheta)=\sin\vartheta P_h(\vartheta)+\beta_{h+1}\cos^{h+1}\vartheta, 
\]
with $P_h$ as in \eqref{eq:polynom_P_h}, we deduce both that $\beta_{h+1}=0$ and that
\[
P_h\left(\tilde{\vartheta}+\frac{2\pi l'}{k}+\varepsilon_{l'}\right)=0
\]
for every $l'\in [0,k-1]$ different from $l$. These are $k-1$ conditions for a polynome of order $h\leq(k-3)/2$, so the induction hypothesis implies $P_h\equiv0$ and in turn $P_{h+1}\equiv0$.
\end{proof}

\begin{proof}[End of the proof of Theorem \ref{theorem:comportmentofeigenvelue}]
Take any $a\in\Omega$ sufficiently close to 0, then by Lemma \ref{lemma:taylordevelopment}
\[
\lambda^{a} - \lambda^{0} = \sum_{h=1}^H|a|^h P_{h}(\vartheta(a)) + o(|a|^{H}).
\]
By combining Lemmas \ref{lemma:annulationpolynomeonnodallines} and \ref{lemma:annulationpolynome} we obtain that $P_h\equiv0$ for every $h\in [1,(k-1)/2]$, therefore $|\lambda^a-\lambda^0|\leq C |a|^{(k+1)/2}$ for some constant $C$ independent of $a$.
\end{proof}

\section{Numerical illustration\label{sec.numeric}}
Let us now illustrate  some results of this paper using  the Finite Element Library \cite{Melina} with isoparametric $\mathbb P_{6}$ lagrangian elements. We will restrict our attention to the case of half-integer circulation $\alpha=1/2$.

The numerical method we used here was presented in details in \cite{BH}. Given a domain $\Omega$ and a point $a\in\Omega$, to compute the eigenvalues $\lambda_{j}^a$ of the Aharonov-Bohm operator $(i\nabla+A_{a})^2$ on $\Omega$, we compute those of the Dirichlet Laplacian on the double covering $\Omega^{\mathcal R}_{a}$ of $\Omega\setminus\{a\}$, denoted by $\mu^{\mathcal R}_{j}$. This spectrum of the Laplacian on $\Omega^{\mathcal R}_{a}$ is decomposed in two disjoint parts:
\begin{itemize}
\item the spectrum of the Dirichlet Laplacian on $\Omega$, $\lambda_{j}$,
\item the spectrum of the magnetic Schr\"odinger operator $(i\nabla+A_{a})^2$, $\lambda_{j}^a$.
\end{itemize}
Thus we have
\[
\{\mu_{j}^{\mathcal R})_{j\geq 1}=\{\lambda^a_{j}\}_{j\geq 1} \bigsqcup \{\lambda_{j}\}_{j\geq 1}.
\]
Therefore by computing the spectrum of the Dirichlet Laplacian on $\Omega$ and, for every $a\in\Omega$, that on the double covering $\Omega^{\mathcal R}_{a}$, we deduce the spectrum of the Aharonov-Bohm operator $(i\nabla+A_{a})^2$ on $\Omega$.
This method avoids to deal with the singularity of the magnetic potential and furthermore allows to work with real valued functions. We have just to compute the spectrum of the Dirichlet Laplacian, which is quite standard. The only effort to be done is to mesh a double covering domain.

Let us now present the computations for the angular sector of aperture $\pi/4$:
\[
\Sigma_{\pi/4}=\left\{(x_1,x_2)\in \mathbb R^2,\ x_1>0,\ |x_2|< x_1\tan\frac\pi8,\ x_1^2+x_2^2< 1\right\}.
\]
By symmetry, it is enough to compute the spectrum for $a$ in the half-domain. 
We take a discretization grid of step $1/N$ with $N=100$ or $N=1000$:
\[
a\in \Pi_{N}:=\left\{\left(\frac mN,\frac nN\right), 0<m<N, 0< \frac{|n|}m<\tan\frac\pi8,  \frac{m^2+n^2}{N^2}<1\right\}.
\]
Figure~\ref{fig.AB1-9} gives the first nine eigenvalues $\lambda_{j}^a$ for $a\in \Pi_{100}$. 
In these figures, the angular sector is represented by a dark thick line. Outside the angular sector are represented the eigenvalues $\lambda_{j}$ of the Dirichlet Laplacian on $\Sigma_{\pi/4}$ (which do not depend on $a$). We observe the convergence proved in Theorem~\ref{theorem:continuity}:
\[
\forall j\geq1,\quad \lambda_{j}^a \to \lambda_{j}\quad\mbox{ as }a\to \partial{\Sigma}_{\pi/4}.
\]
In Figure~\ref{fig.VPpi4ABbis}, we provide the 3-D representation of Figures~\ref{fig.AB1-9a} and \ref{fig.AB1-9b}.

\begin{figure}[h!t]
\begin{center}
\subfigure[$a\mapsto\lambda_{1}^a$\label{fig.AB1-9a}]{\includegraphics[height=3cm]{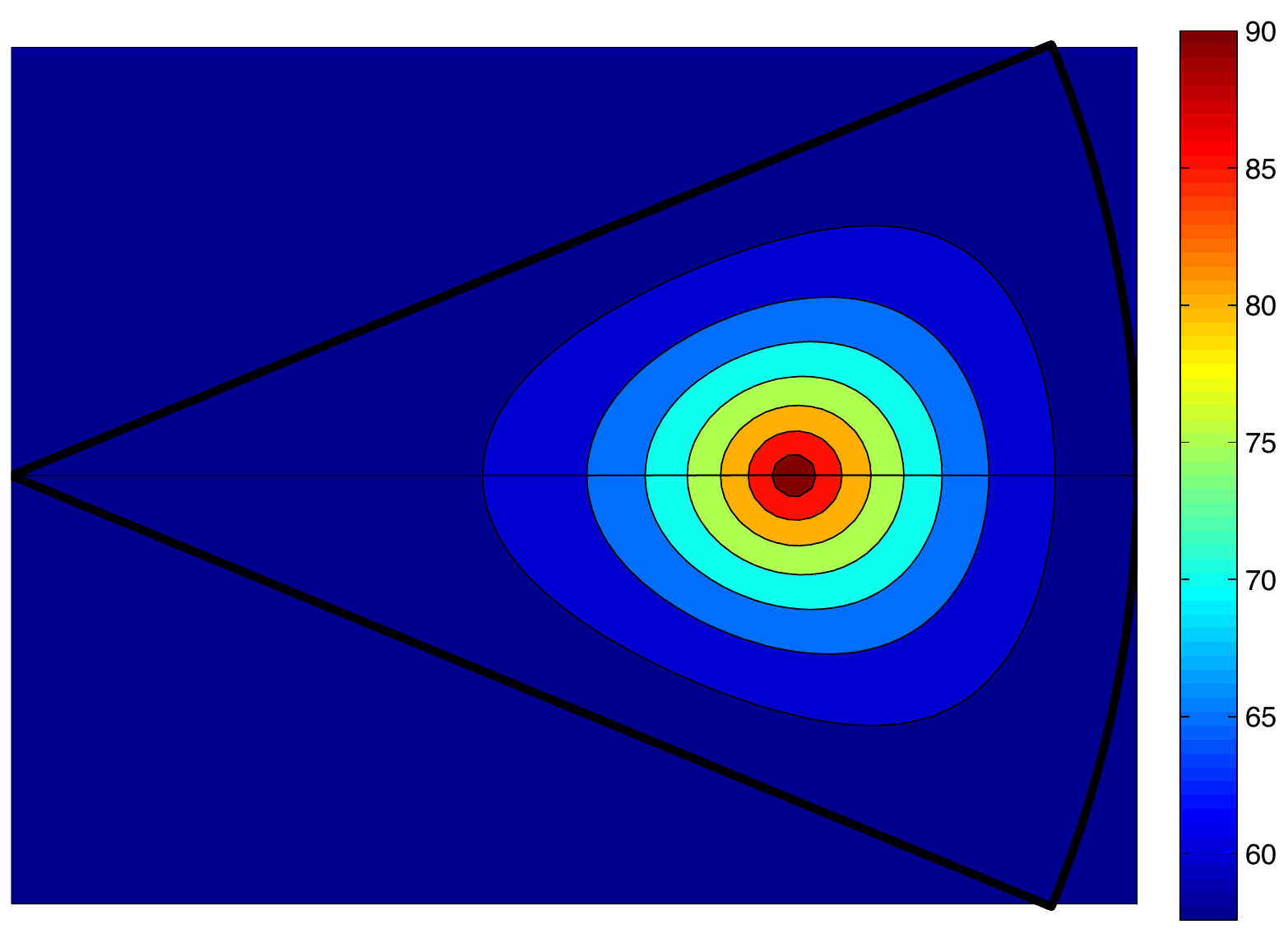}}
\subfigure[$a\mapsto\lambda_{2}^a$\label{fig.AB1-9b}]{\includegraphics[height=3cm]{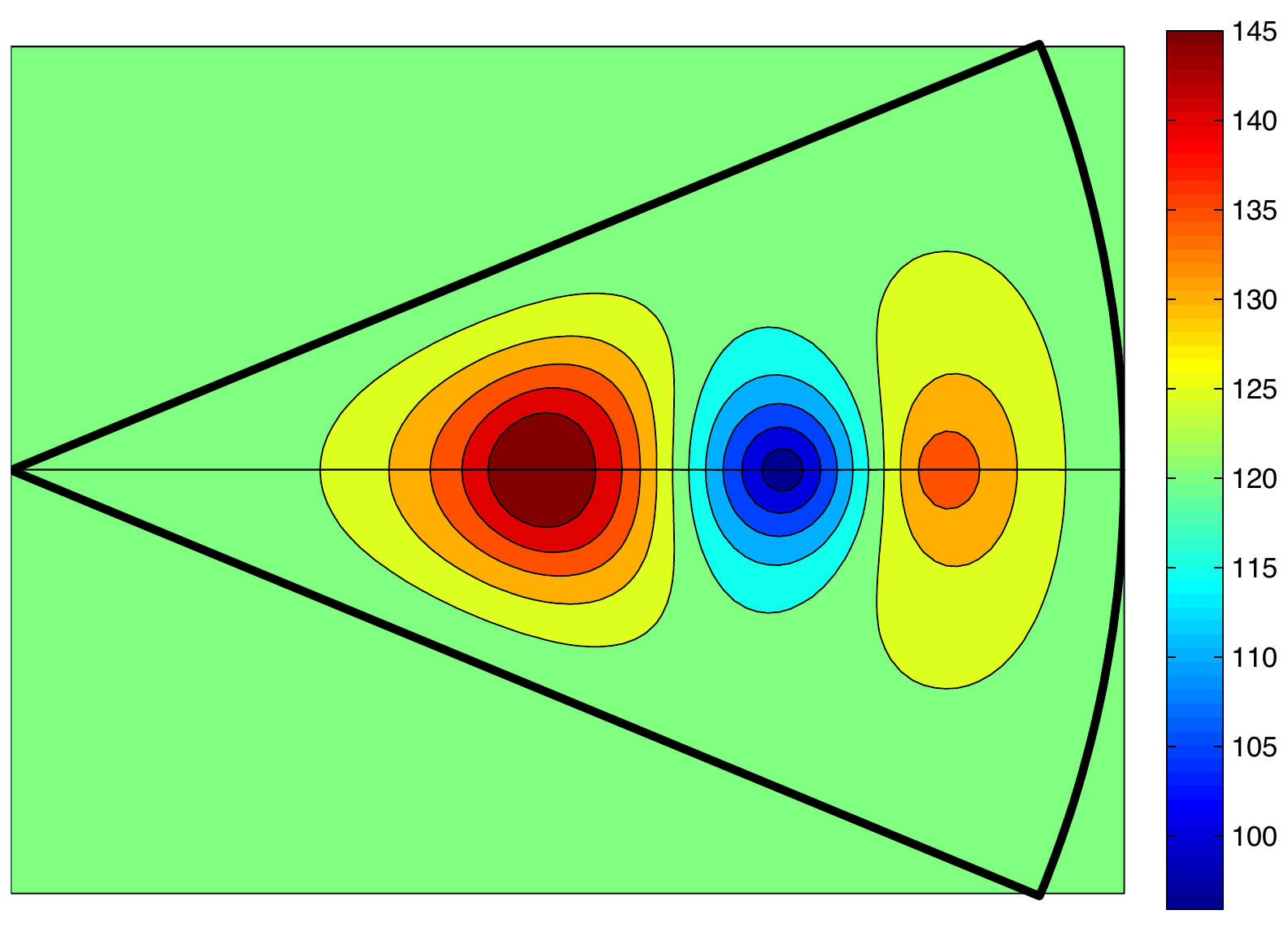}}
\subfigure[$a\mapsto\lambda_{3}^a$]{\includegraphics[height=3cm]{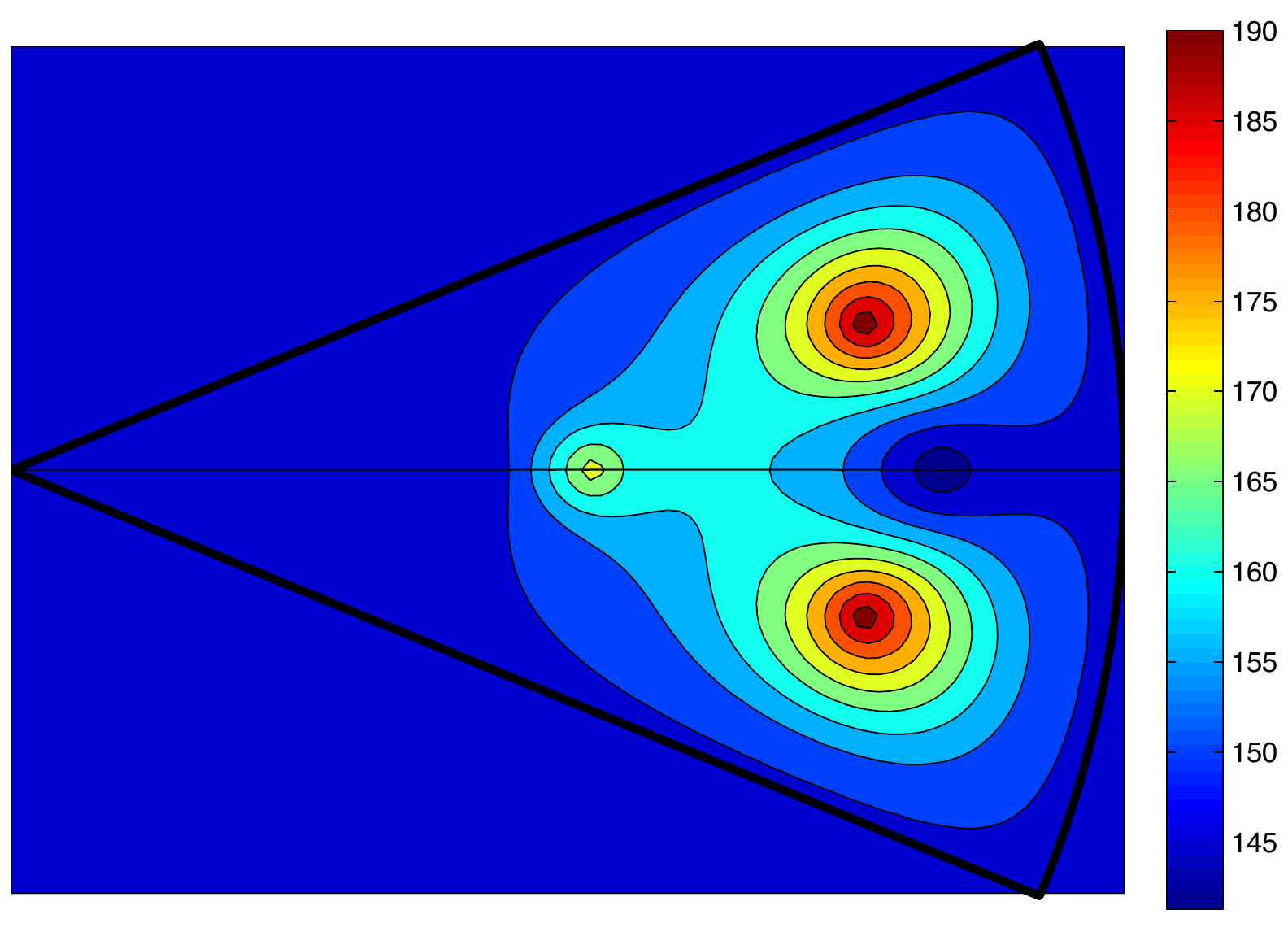}}\\
\subfigure[$a\mapsto\lambda_{4}^a$]{\includegraphics[height=3cm]{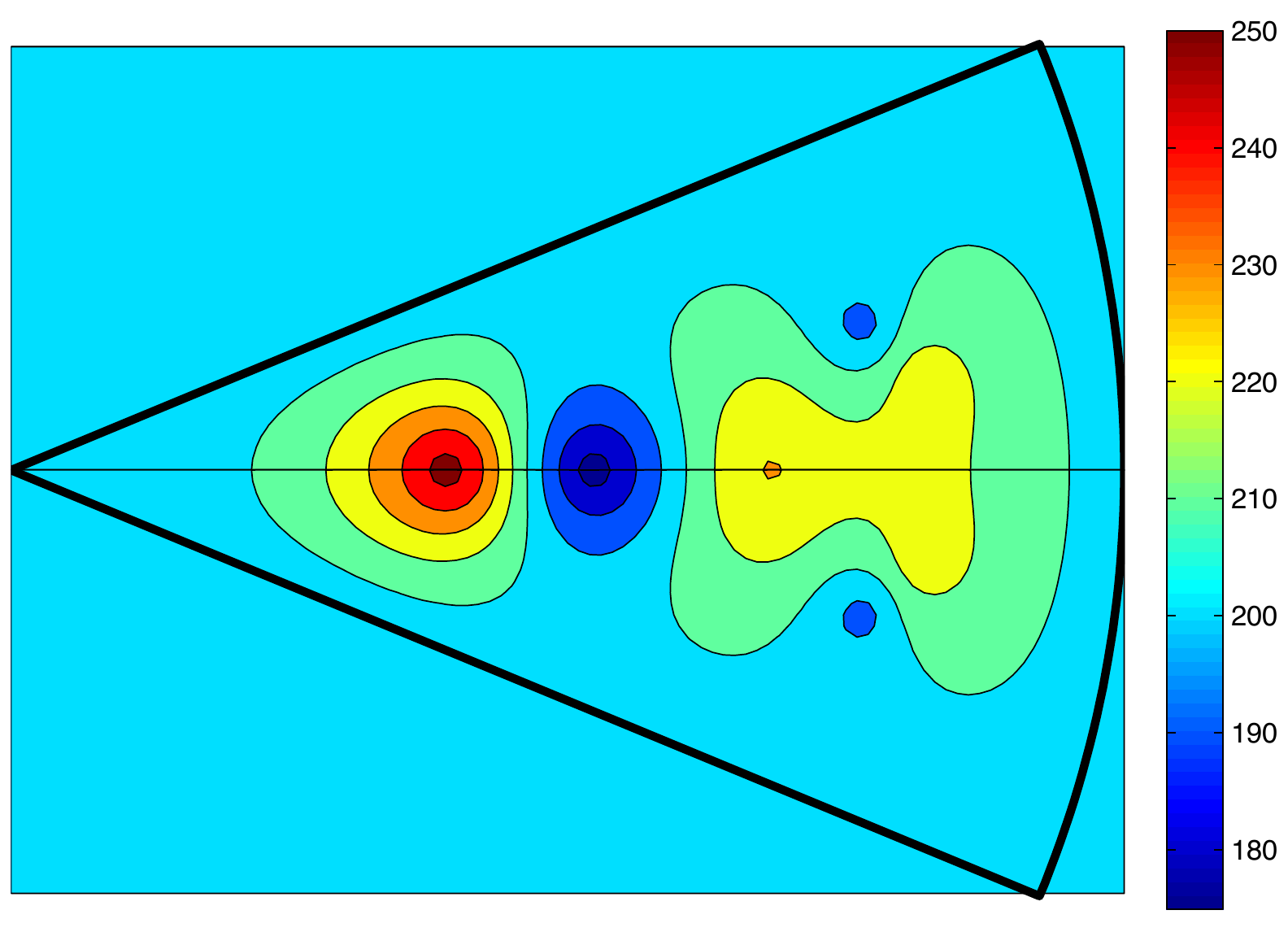}}
\subfigure[$a\mapsto\lambda_{5}^a$]{\includegraphics[height=3cm]{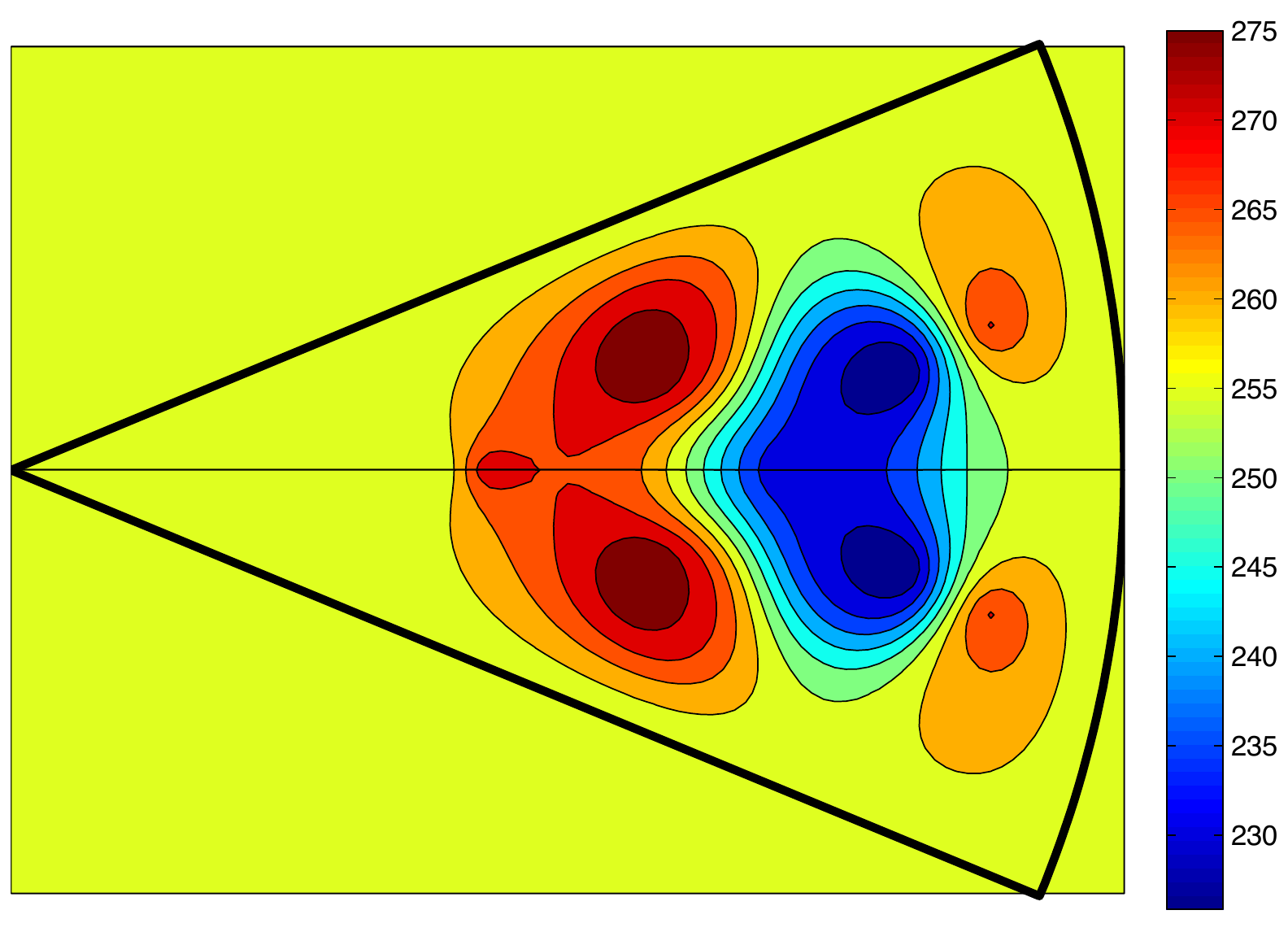}}
\subfigure[$a\mapsto\lambda_{6}^a$]{\includegraphics[height=2.99cm]{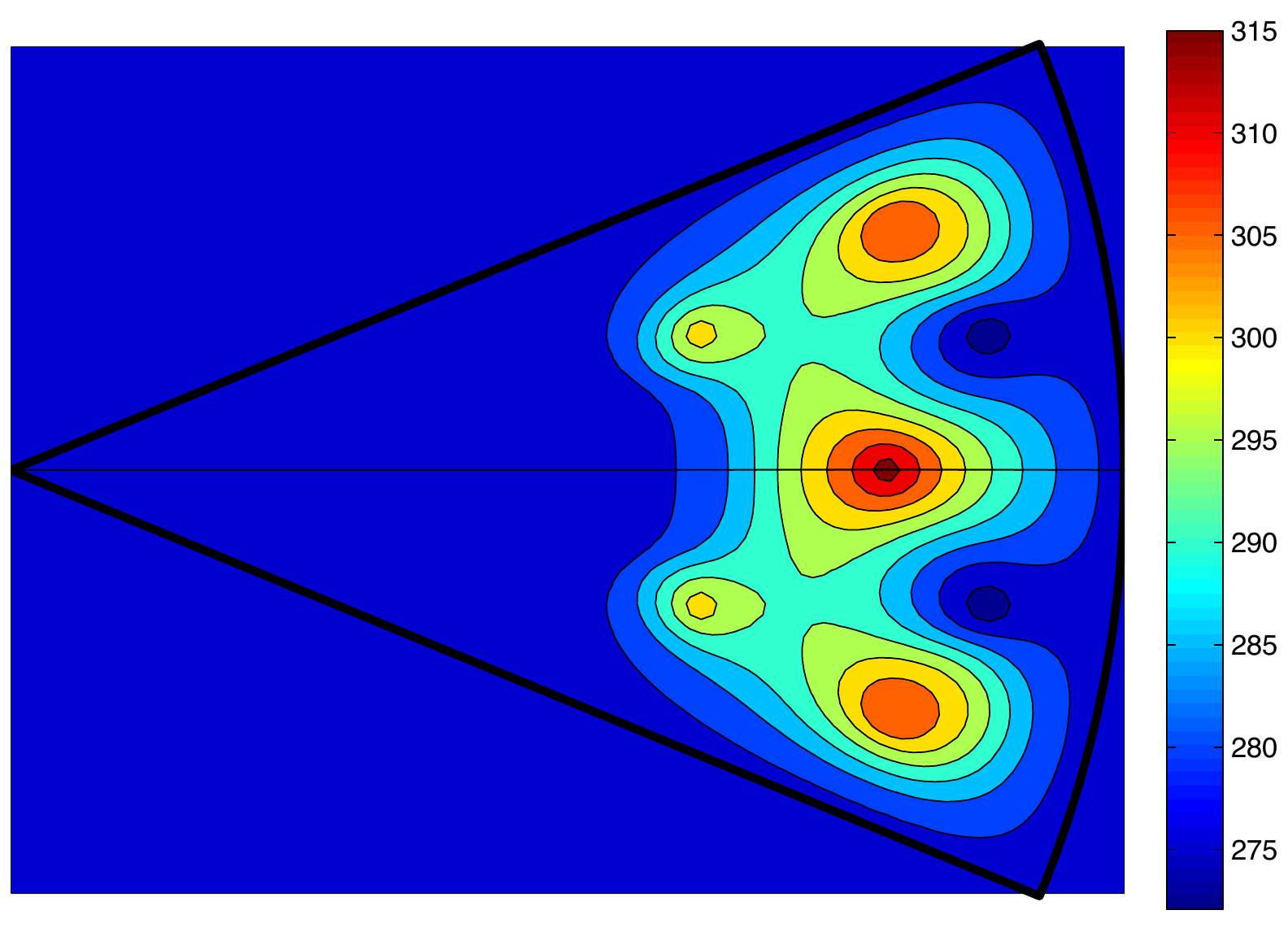}}\\
\subfigure[$a\mapsto\lambda_{7}^a$]{\includegraphics[height=3cm]{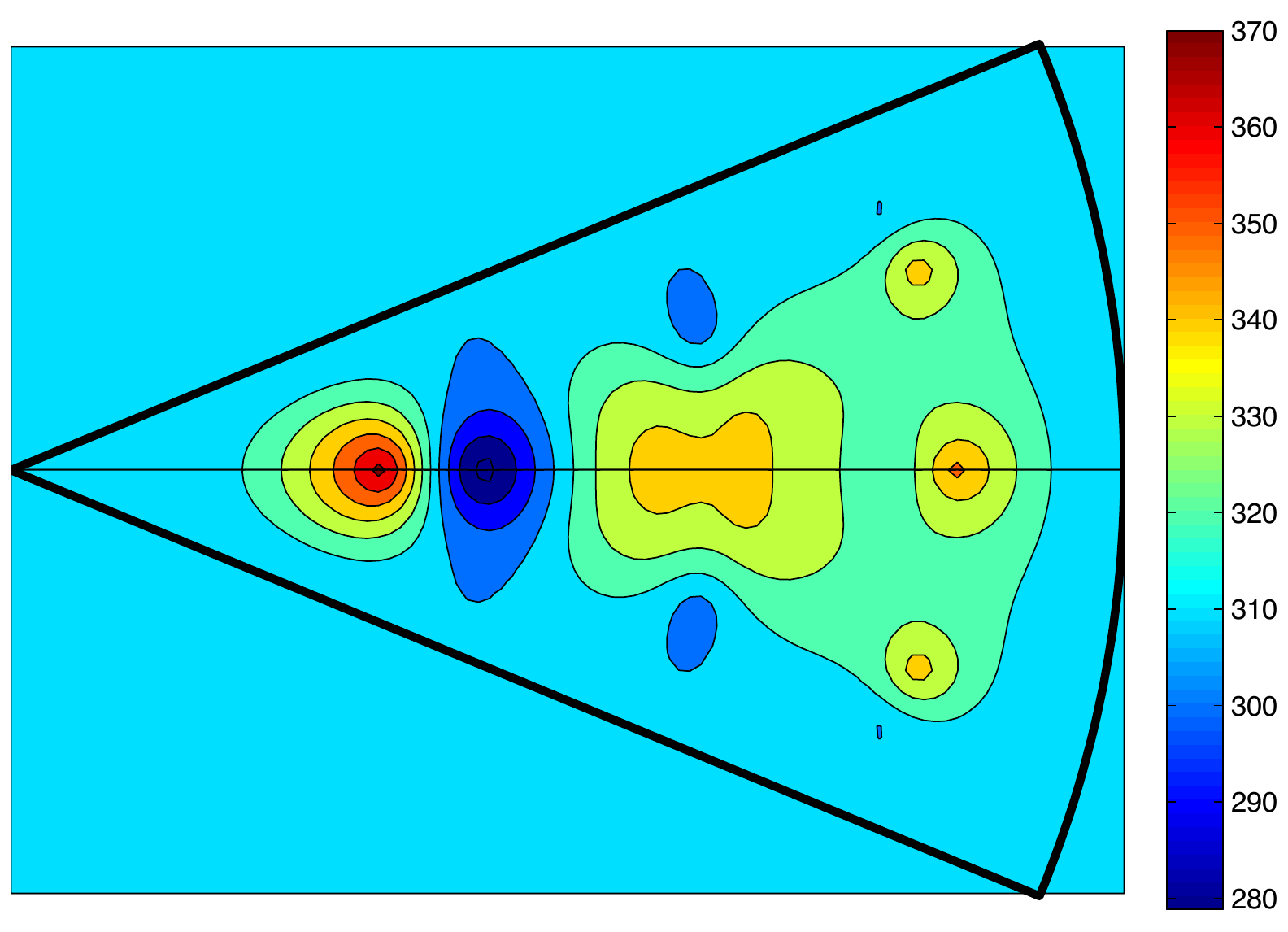}}
\subfigure[$a\mapsto\lambda_{8}^a$]{\includegraphics[height=3cm]{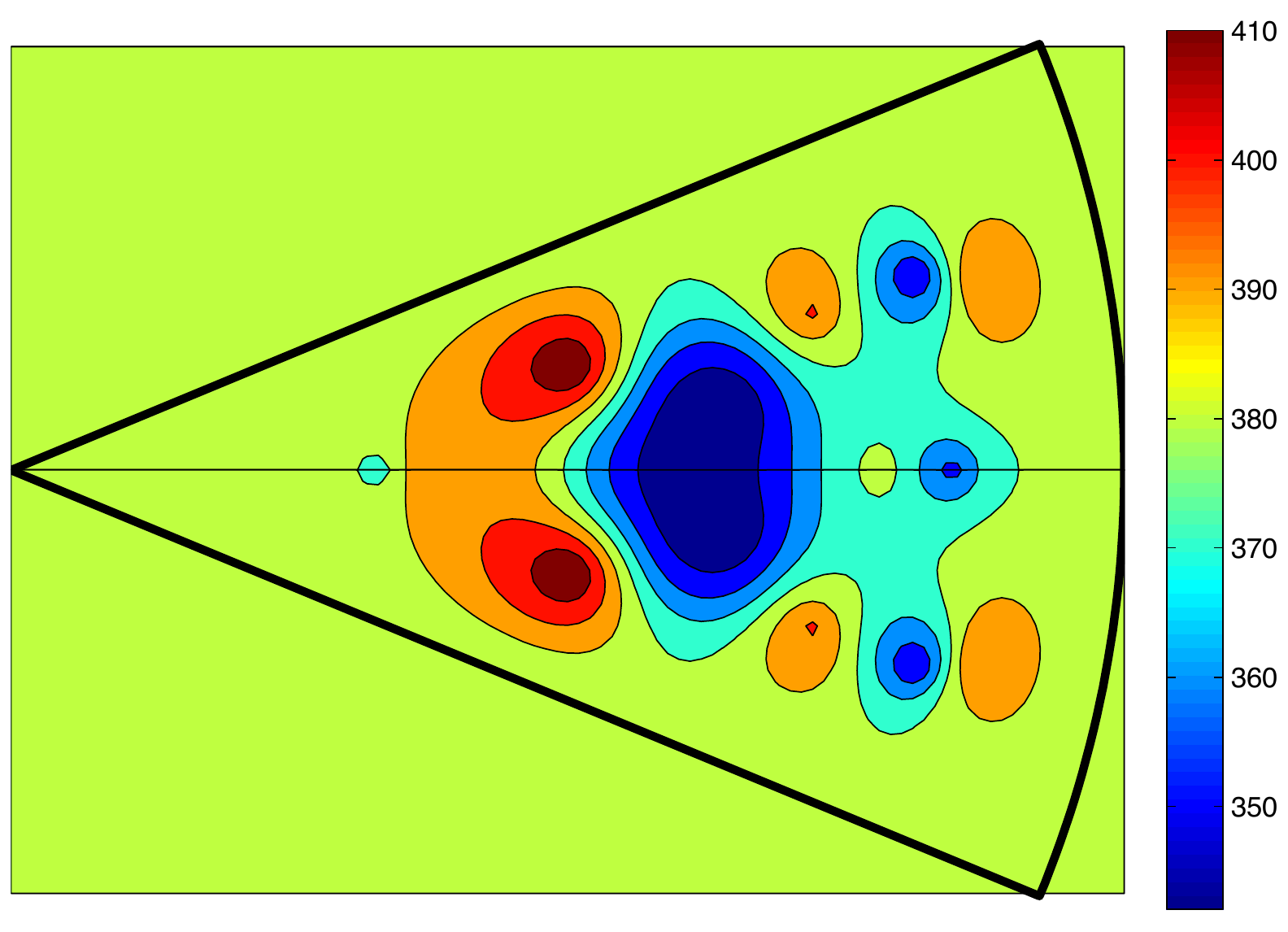}}
\subfigure[$a\mapsto\lambda_{9}^a$]{\includegraphics[height=2.99cm]{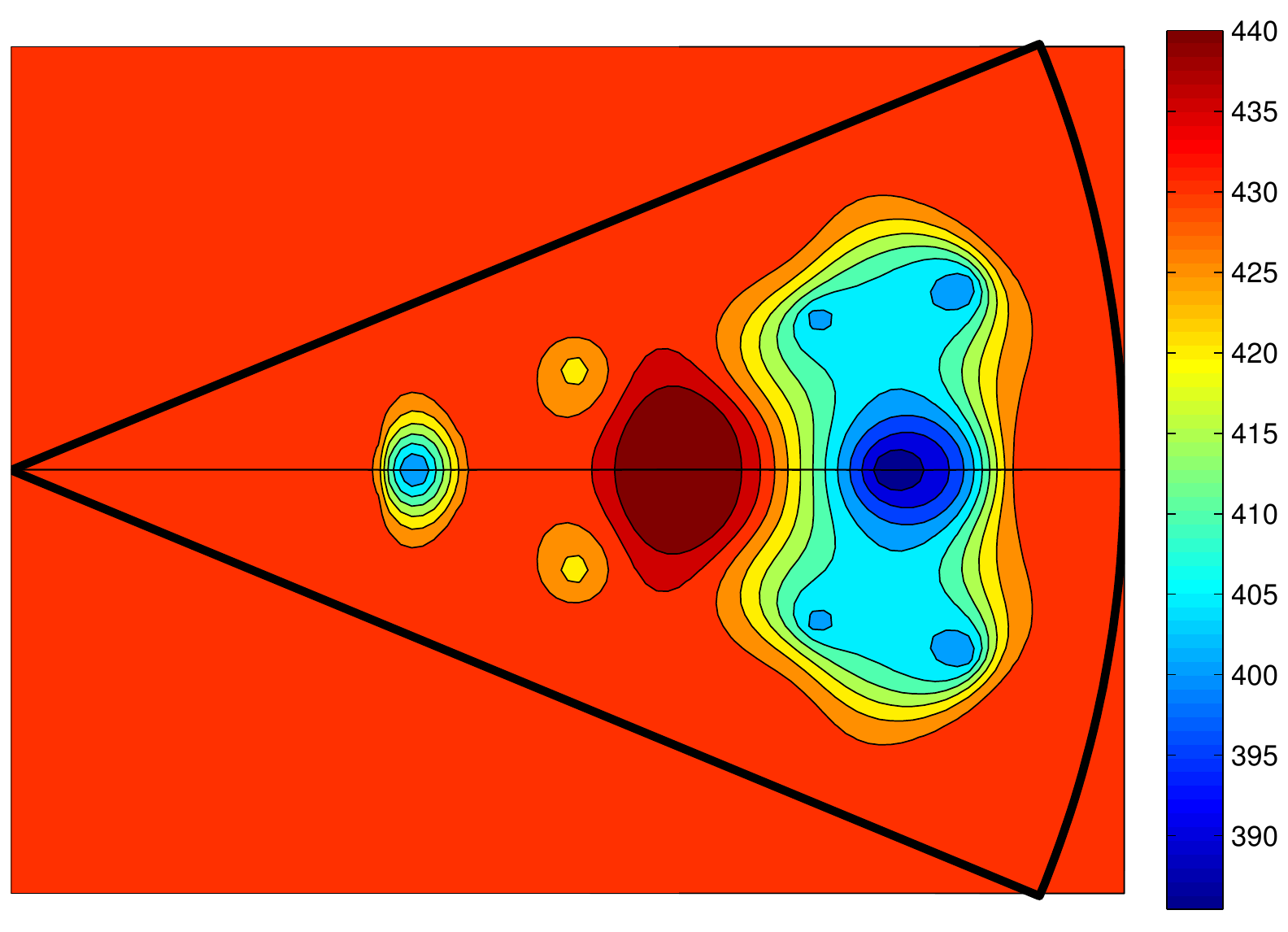}}
\caption{First nine eigenvalues of $(i\nabla+A_{a})^2$ in $\Sigma_{\pi/4}$, $a\in \Pi_{100}$.\label{fig.AB1-9}}
\end{center}
\end{figure}

Let us now deal more accurately with the singular points on the symmetry axis. Numerically, we take a discretization step equal to $1/1000$ and consider $a\in \{(m/1000,0), 1\leq m\leq 1000\}$. Figure~\ref{fig.VPABy0} gives the first nine eigenvalues of the Aharonov-Bohm operator $(i\nabla+A_{a})^2$ in $\Sigma_{\pi/4}$. Here we can identify the points $a$ belonging to the symmetry axis such that $\lambda_j^a$ is not simple. If we look for example at the first and second eigenvalues, we see that they are not simple respectively for one and three values of $a$ on the symmetry axis. At such values, the function $a\mapsto \lambda_j^a$, $j=1,2$, is not differentiable, as can be seen in Figure~\ref{fig.VPpi4ABbis}. Figure~\ref{fig.VPpi4ABbis} illustrates Theorem~\ref{theorem:differentiability_simple_eigenvalue} for a domain with a piecewise $C^\infty$ boundary: we see that the function $a\mapsto \lambda_{j}^a$, $j=1,2$, is regular except at the points where the eigenvalue $\lambda_{j}^a$ is not simple. 

\begin{figure}[h!t]
\begin{center}
\subfigure[$a\mapsto\lambda_{1}^a$, $a\in \Pi_{100}$]{\includegraphics[height=4.5cm]{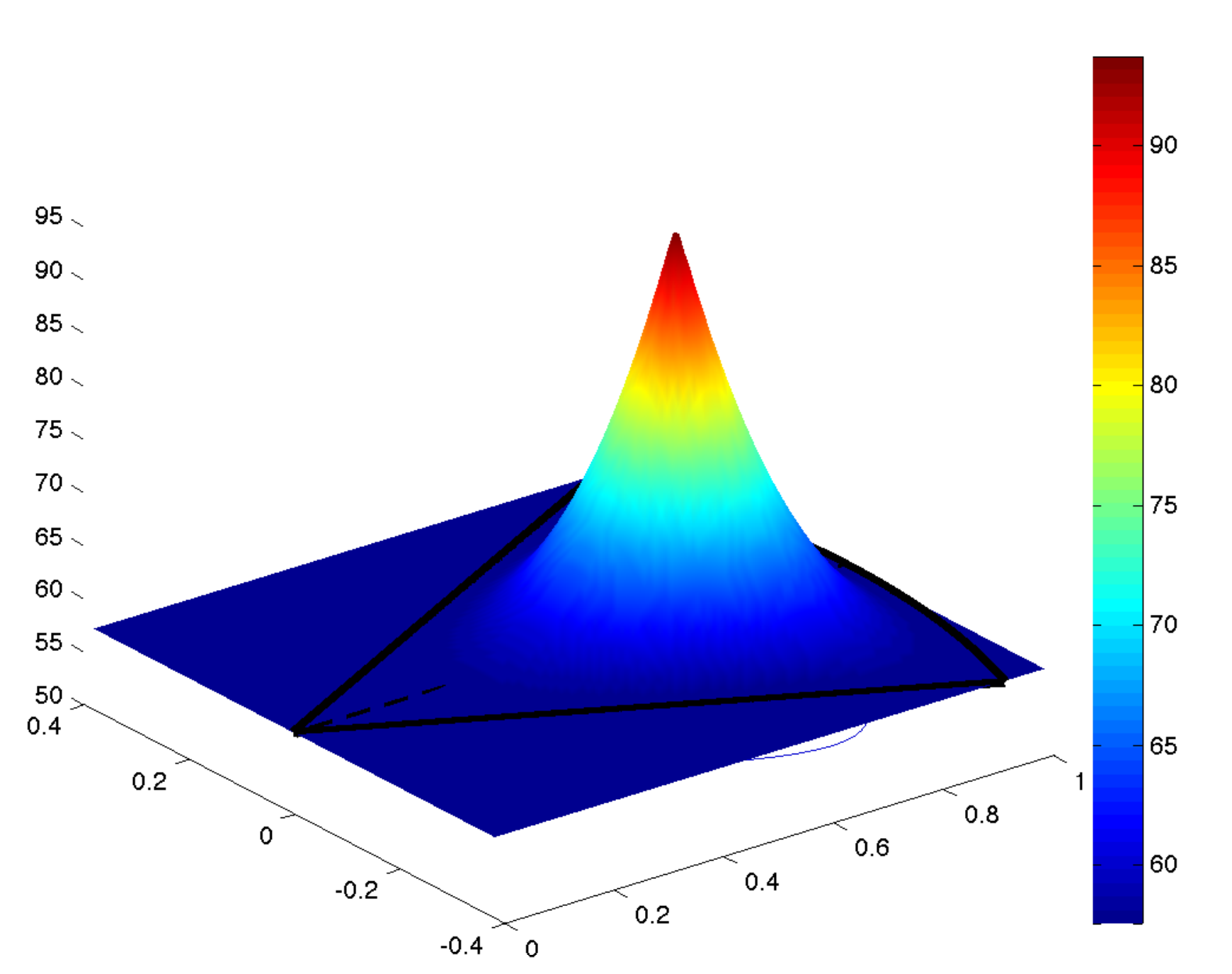}}
\subfigure[$a\mapsto\lambda_{2}^a$, $a\in \Pi_{100}$]{\includegraphics[height=4.5cm]{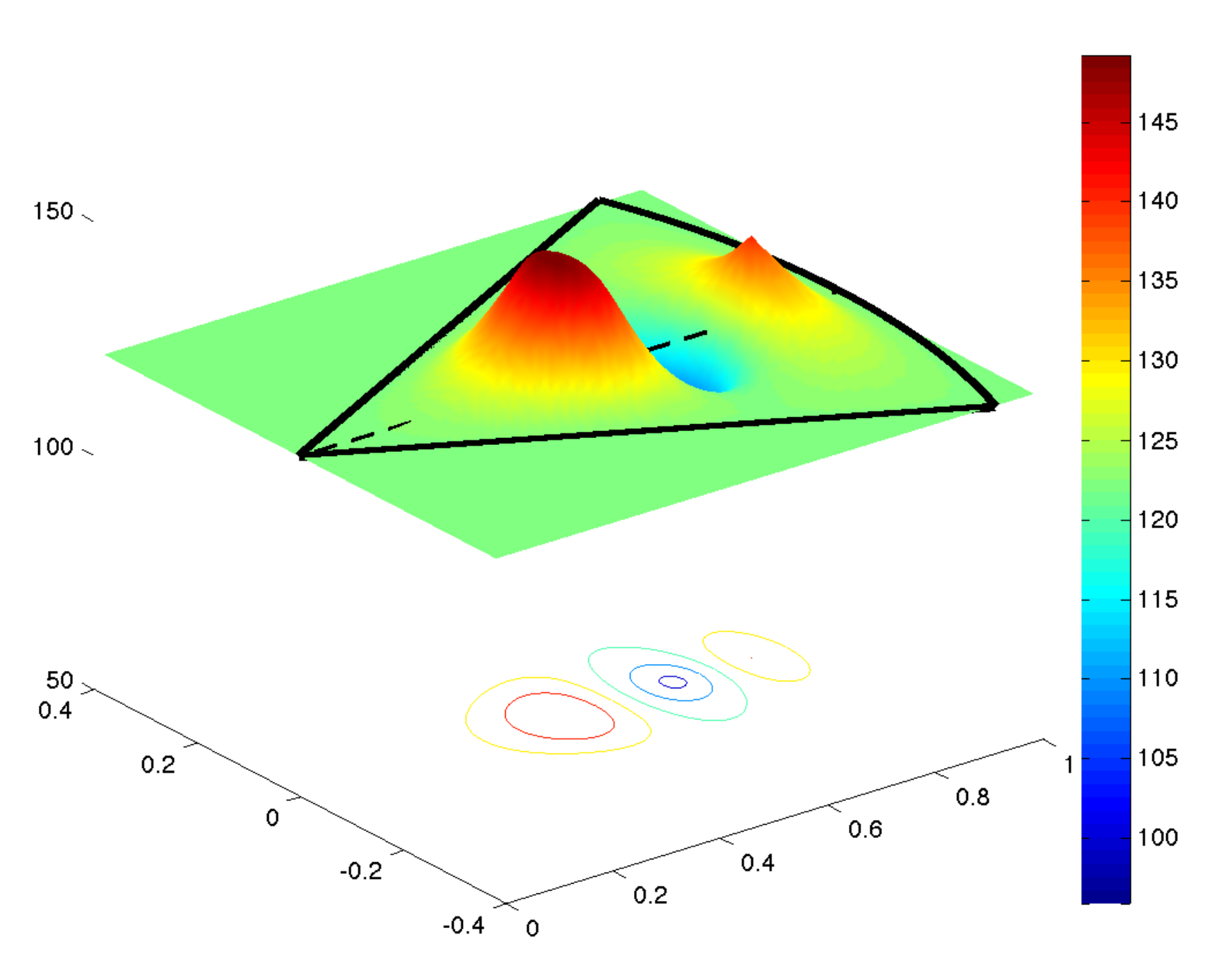}}
\caption{3-D representation of Figures~\ref{fig.AB1-9a} and \ref{fig.AB1-9b}.\label{fig.VPpi4ABbis}}
\end{center}
\end{figure}

\begin{figure}[h!t]
\begin{center}
\includegraphics[height=7cm]{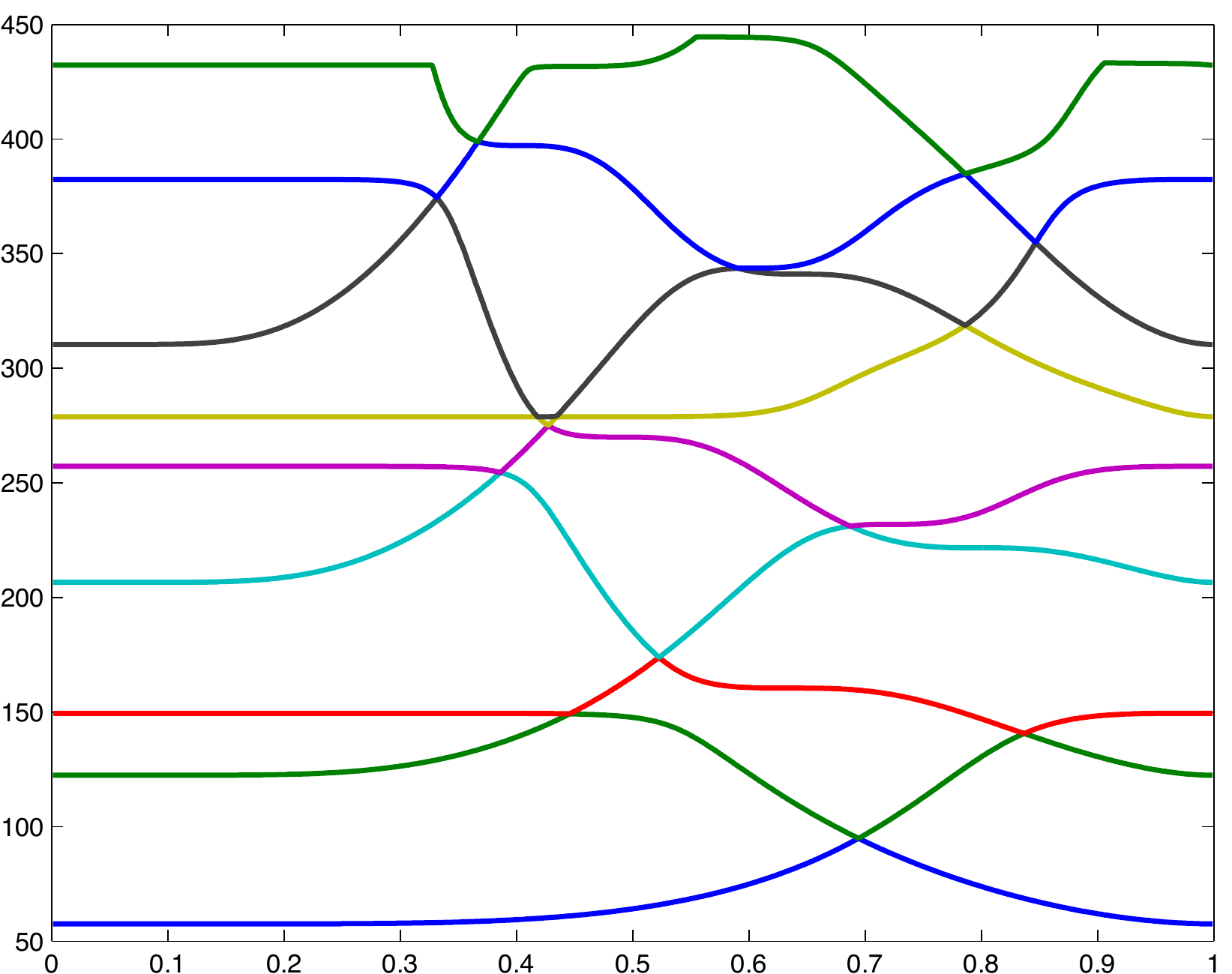}
\caption{$a\mapsto \lambda_{j}^a$, $a\in\left\{(\frac m{1000},0), 0< m< 1000\right\}$, $1\leq j\leq 9$.\label{fig.VPABy0}}
\end{center}
\end{figure}

Going back to Figure~\ref{fig.VPABy0}, we see that the only critical points of $\lambda_j^a$ which correspond to simple eigenvalues are inflexion points. As an example, we have analyzed the inflexion points for $\lambda_{3}^a$, $\lambda_{4}^a$, $\lambda_{5}^a$ when $a=(a_{1},0)$ with $a_{1}\in(0.6,0.7)$, $a_{1}\in(0.75,0.85)$ and $a_{1}\in(0.45,0.55)$ respectively. We will denote these points by $a_{(j)}$, $j=3,4,5$.
In Figure~\ref{fig.vecpAB3-5}, we have plotted the nodal lines of the eigenfunctions $\varphi_j^{a_{(j)}}$ associated with $\lambda_j^{a_{(j)}}$, $j=3,4,5$. We observe that each $\varphi_j^{a_{(j)}}$ has a zero of order $3/2$ at $a_{(j)}$. Correspondingly, the derivative of $\lambda_j^a$ at $a_{(j)}$ vanishes in Figure~\ref{fig.VPABy0}, thus illustrating Theorem \ref{theorem:comportmentofeigenvelue}. In the three examples proposed here, also the second derivative of $\lambda_j^a$ vanishes at $a_{(j)}$.
\begin{figure}[h!t]
\begin{center}
\subfigure[$\lambda_{3}^a$, $a=(0.63,0)\simeq a_{(3)}$]{\includegraphics[height=3cm]{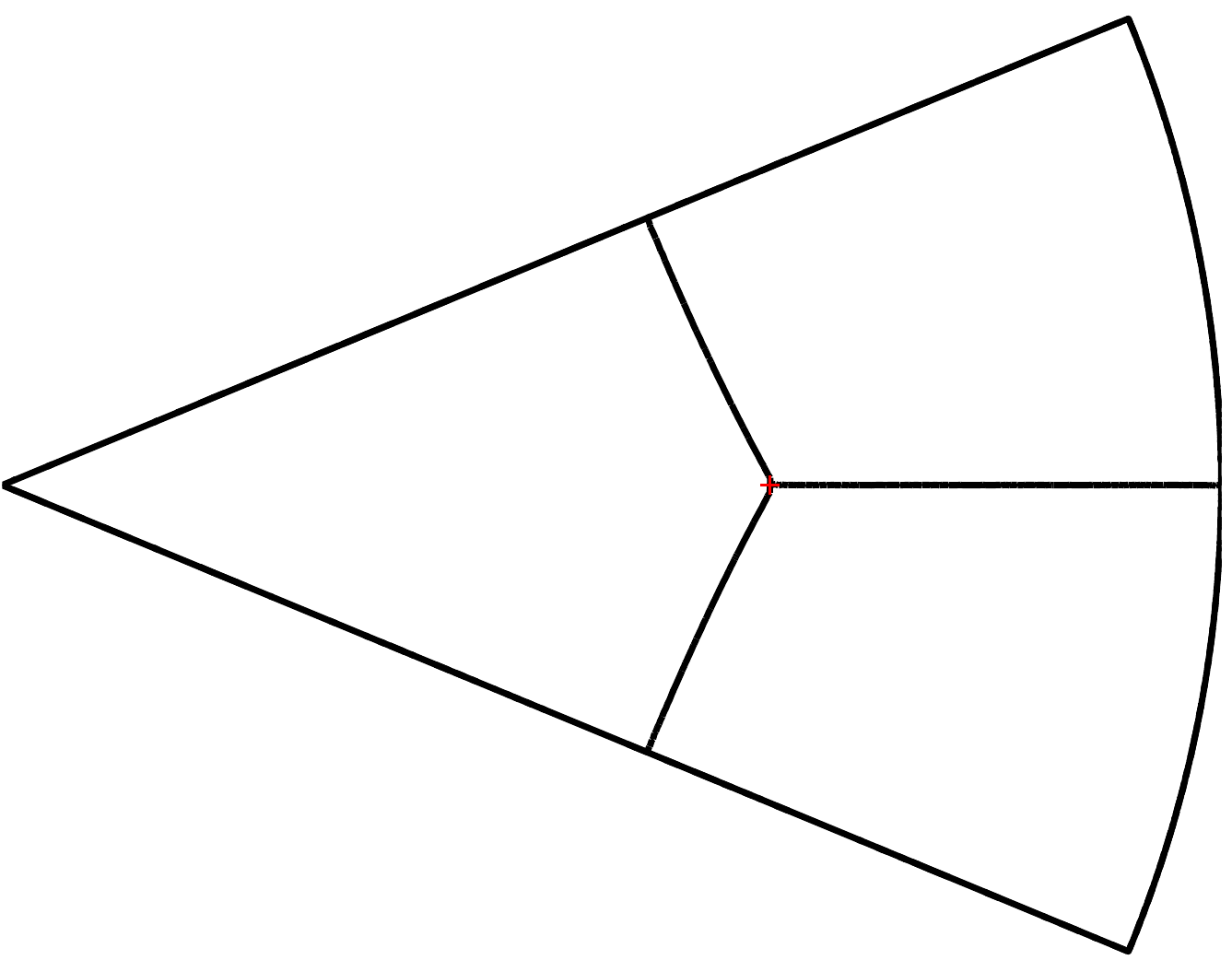}}
\subfigure[$\lambda_{4}^a$, $a=(0.79,0)\simeq a_{(4)}$]{\includegraphics[height=3cm]{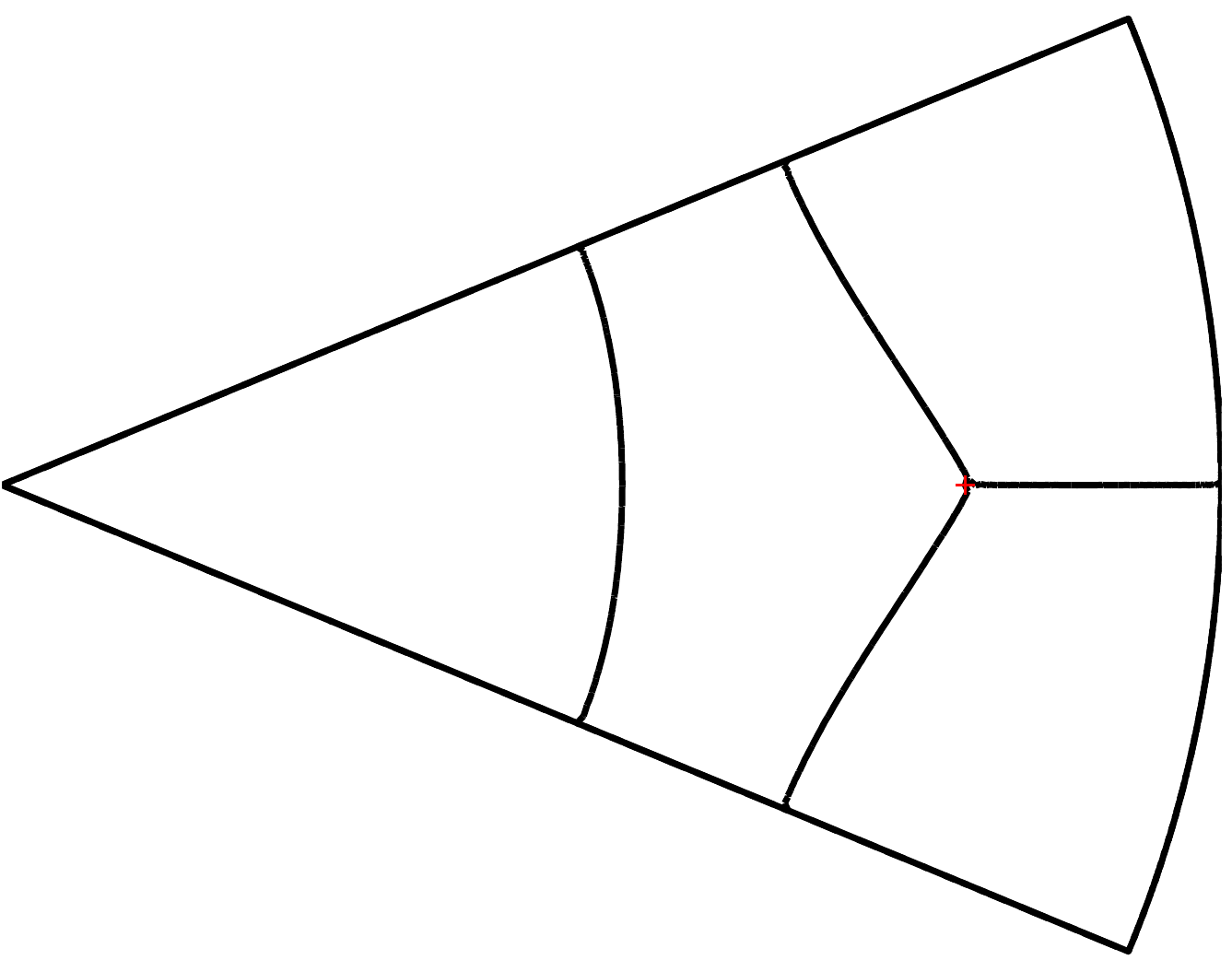}}
\subfigure[$\lambda_{5}^a$, $a=(0.49,0)\simeq a_{(5)}$]{\includegraphics[height=3cm]{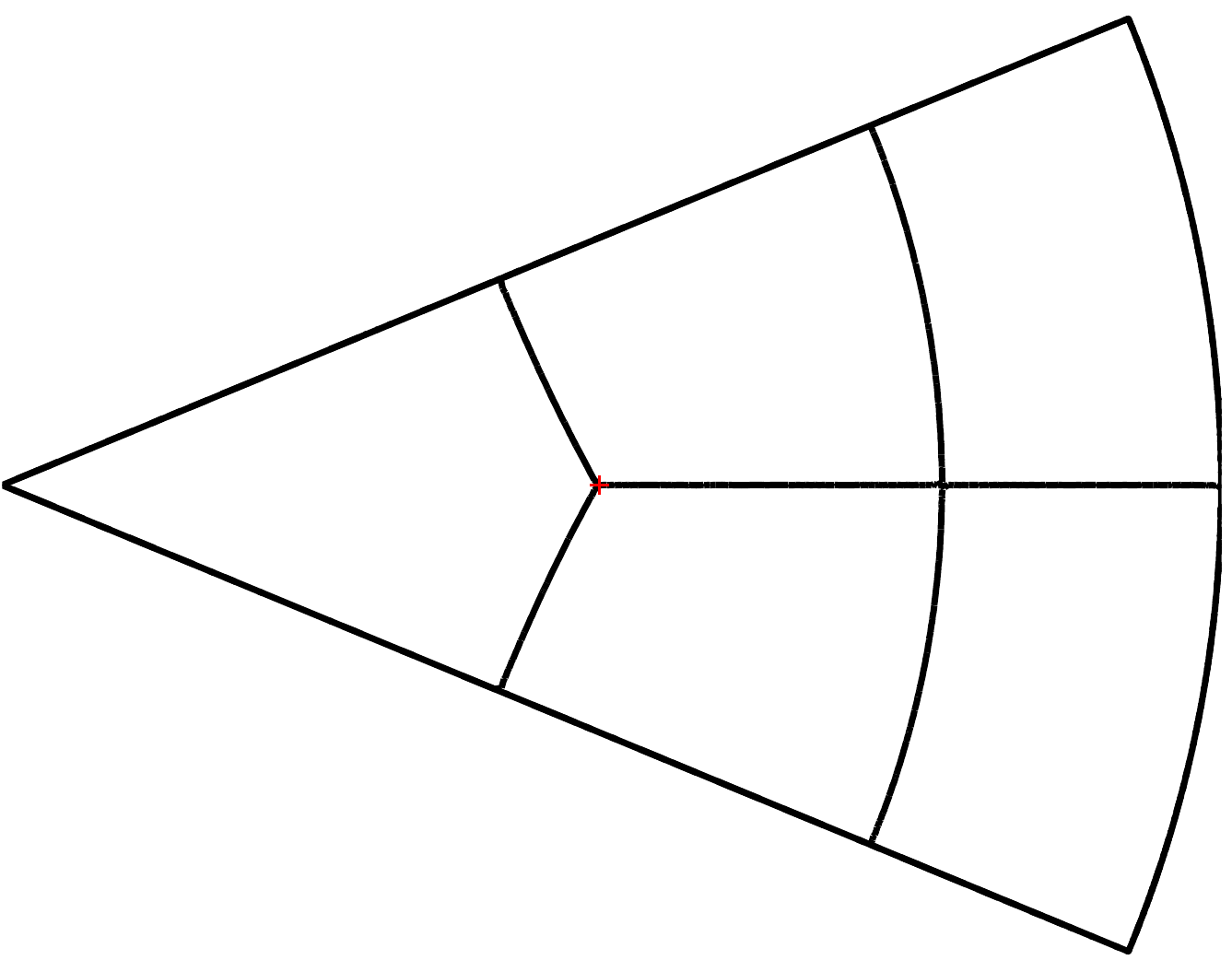}}
\caption{Nodal lines of an eigenfunction associated with $\lambda_{j}^{a_{(j)}}$, $j=3,4,5$.\label{fig.vecpAB3-5}}
\end{center}
\end{figure}

Let us now move a little the singular point around $a_{(j)}$. We use a discretization step of $1/1000$. Figure~\ref{fig.VPAB3-5zoom} represents the behavior of $\lambda_{j}^a$ for $a$ close to $a_{(j)}$. 
It indicates that these points are degenerated saddle points. The behavior of the function $a\mapsto \lambda_{j}^a$, $j=3,4,5$, around $a_{(j)}$ is quite similar to that of the function 
$(t,x)\mapsto t(t^2-x^2)$ around the origin $(0,0)$. 
\begin{figure}[h!t]
\begin{center}
\subfigure[$a\mapsto \lambda_{3}^a$, $a\in\left\{(\frac m{1000},\frac n{1000}), 600\leq m\leq 680, 0\leq n\leq 30\right\}$.\label{fig.VPAB3zoom}]{\includegraphics[width=7cm]{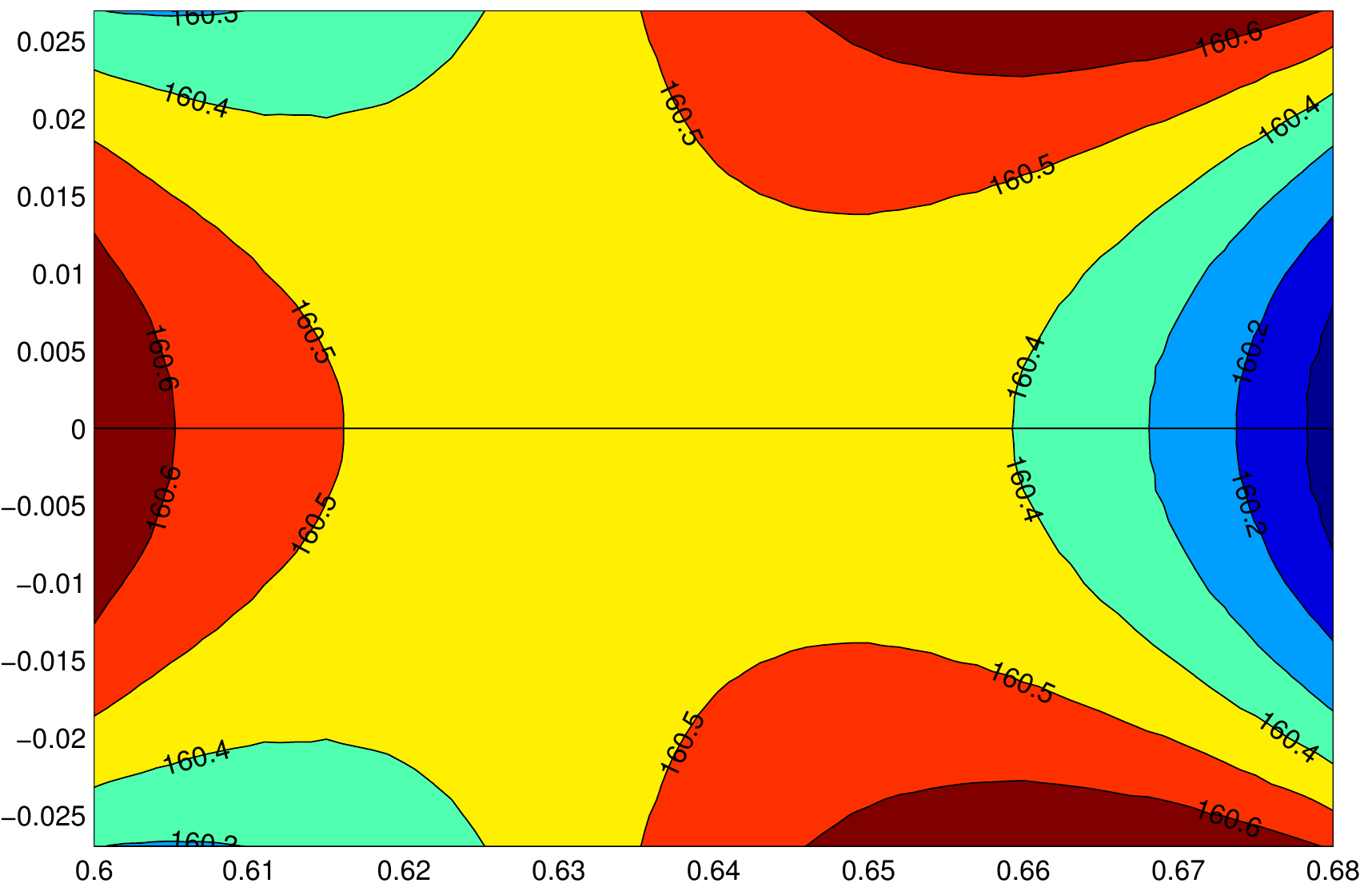}
\includegraphics[width=6.3cm]{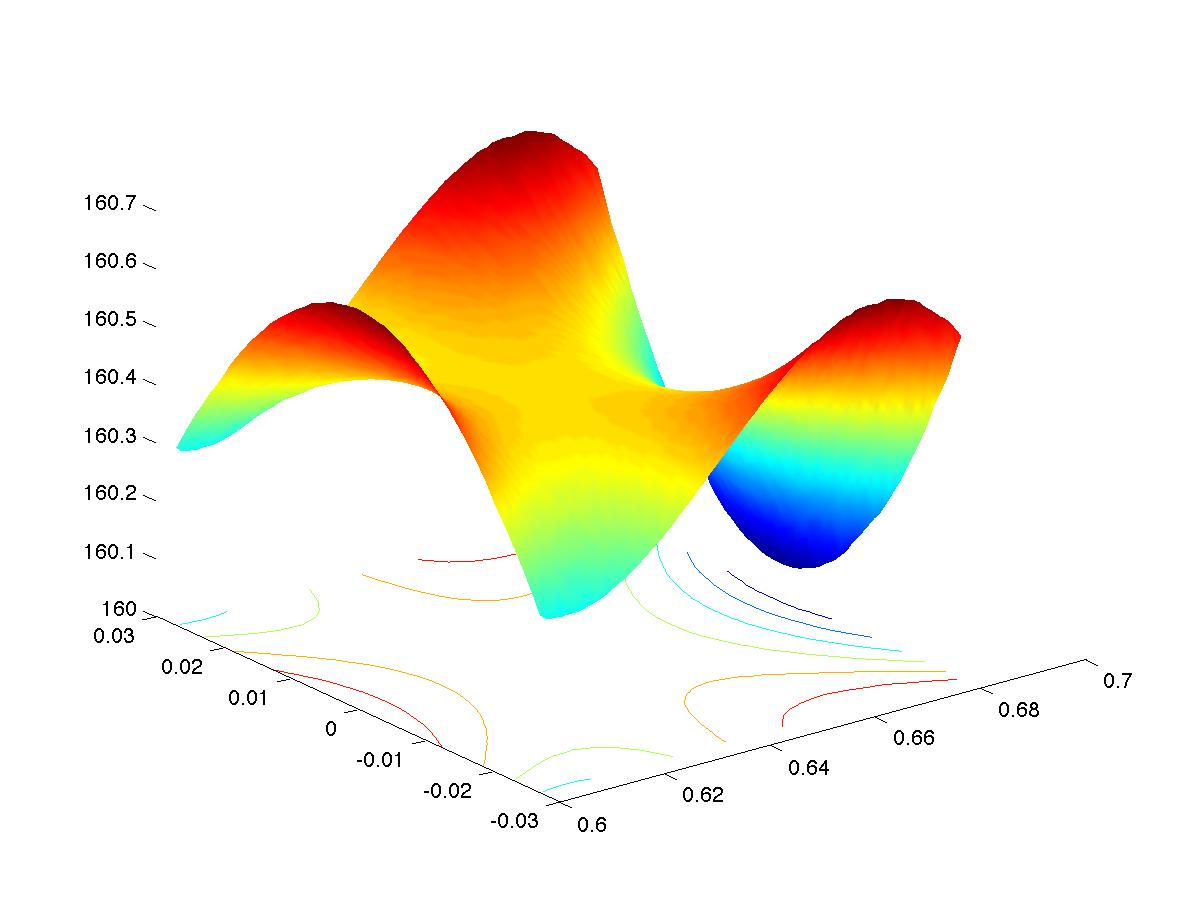}}
\subfigure[$a\mapsto \lambda_{4}^a$, $a\in\left\{(\frac m{1000},\frac n{1000}), 750\leq m\leq 840, 0\leq n\leq 30\right\}$.\label{fig.VPAB4zoom}]{\includegraphics[width=7cm]{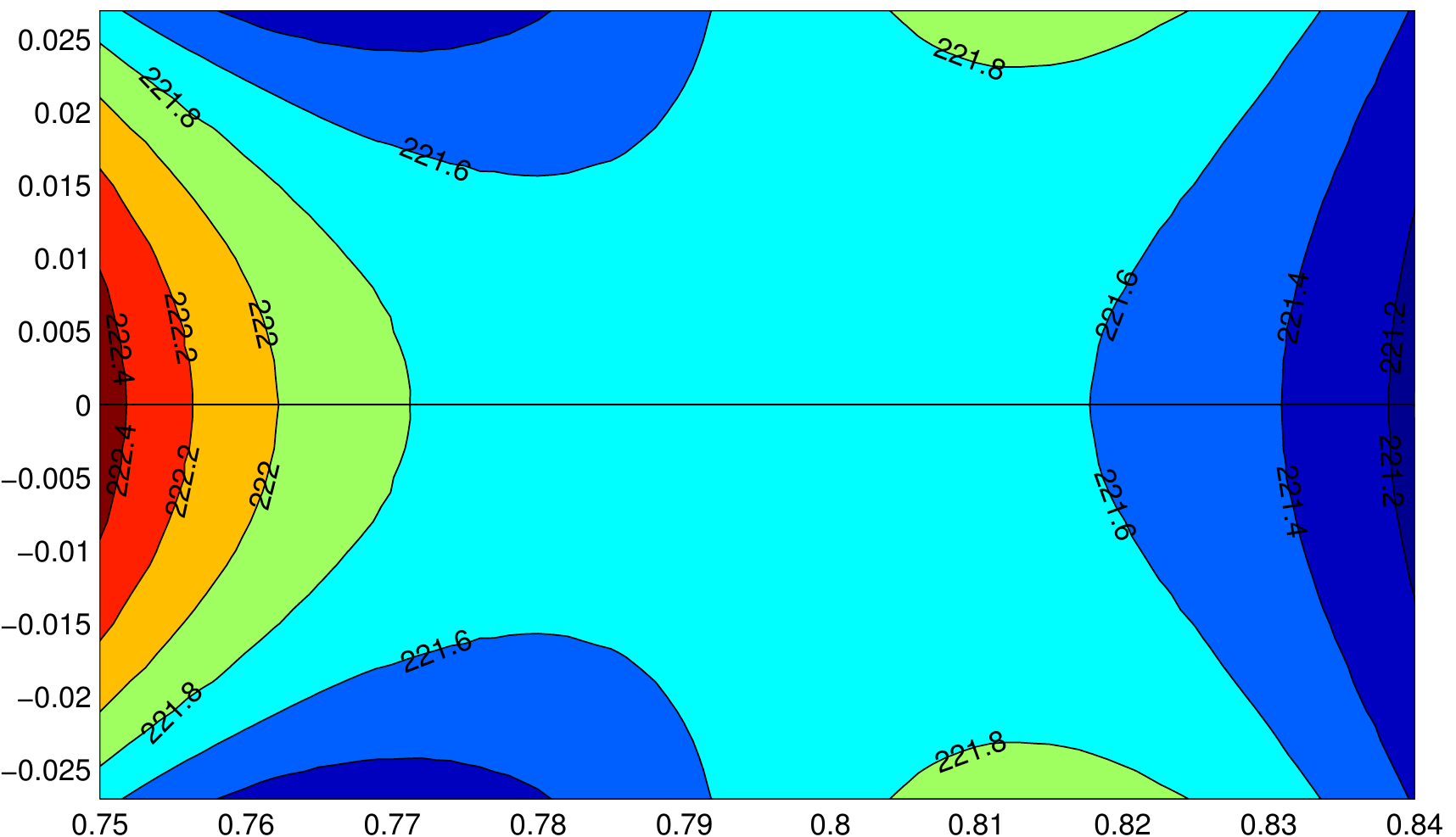}
\includegraphics[width=6.3cm]{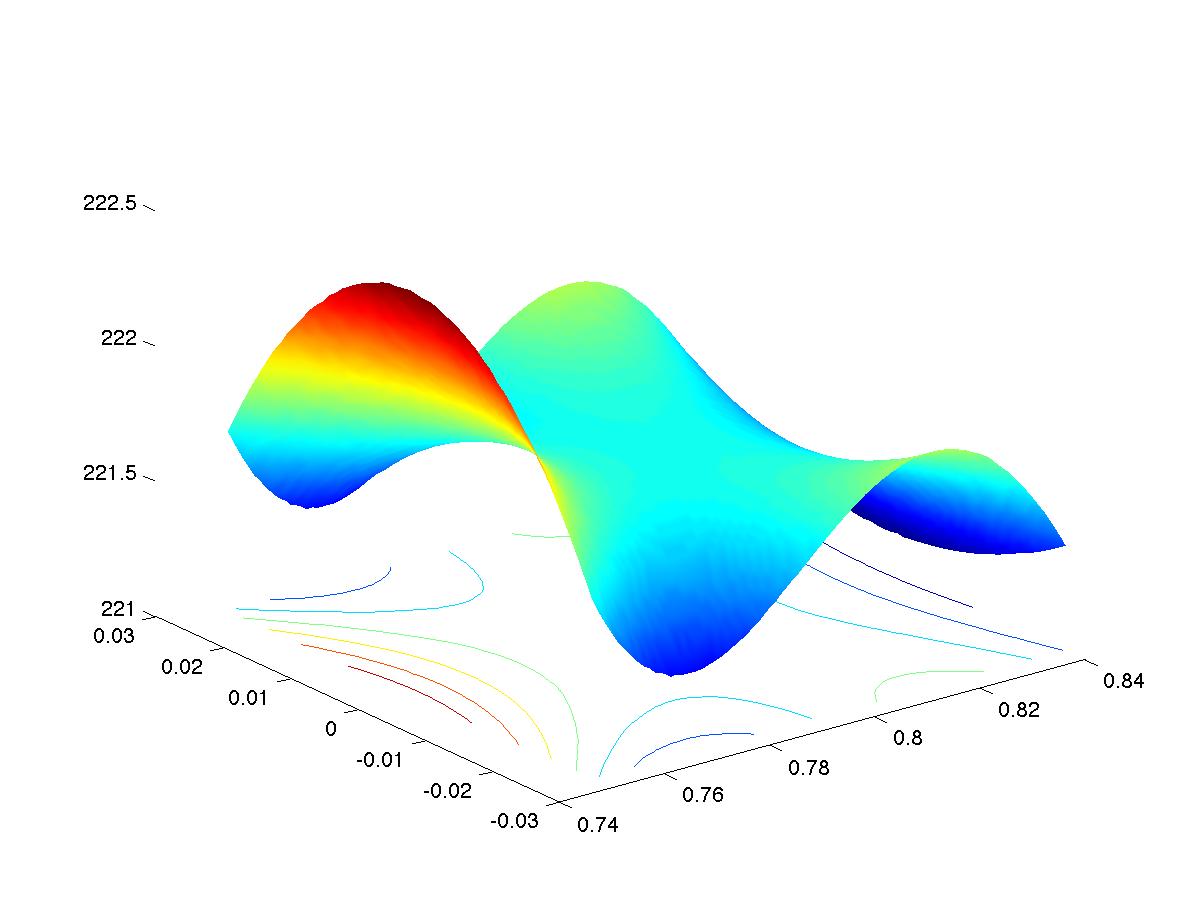}}
\subfigure[$a\mapsto \lambda_{5}^a$, $a\in\left\{(\frac m{1000},\frac n{1000}), 450\leq m\leq 530, 0\leq n\leq 30\right\}$.\label{fig.VPAB5zoom}]{\includegraphics[width=7cm]{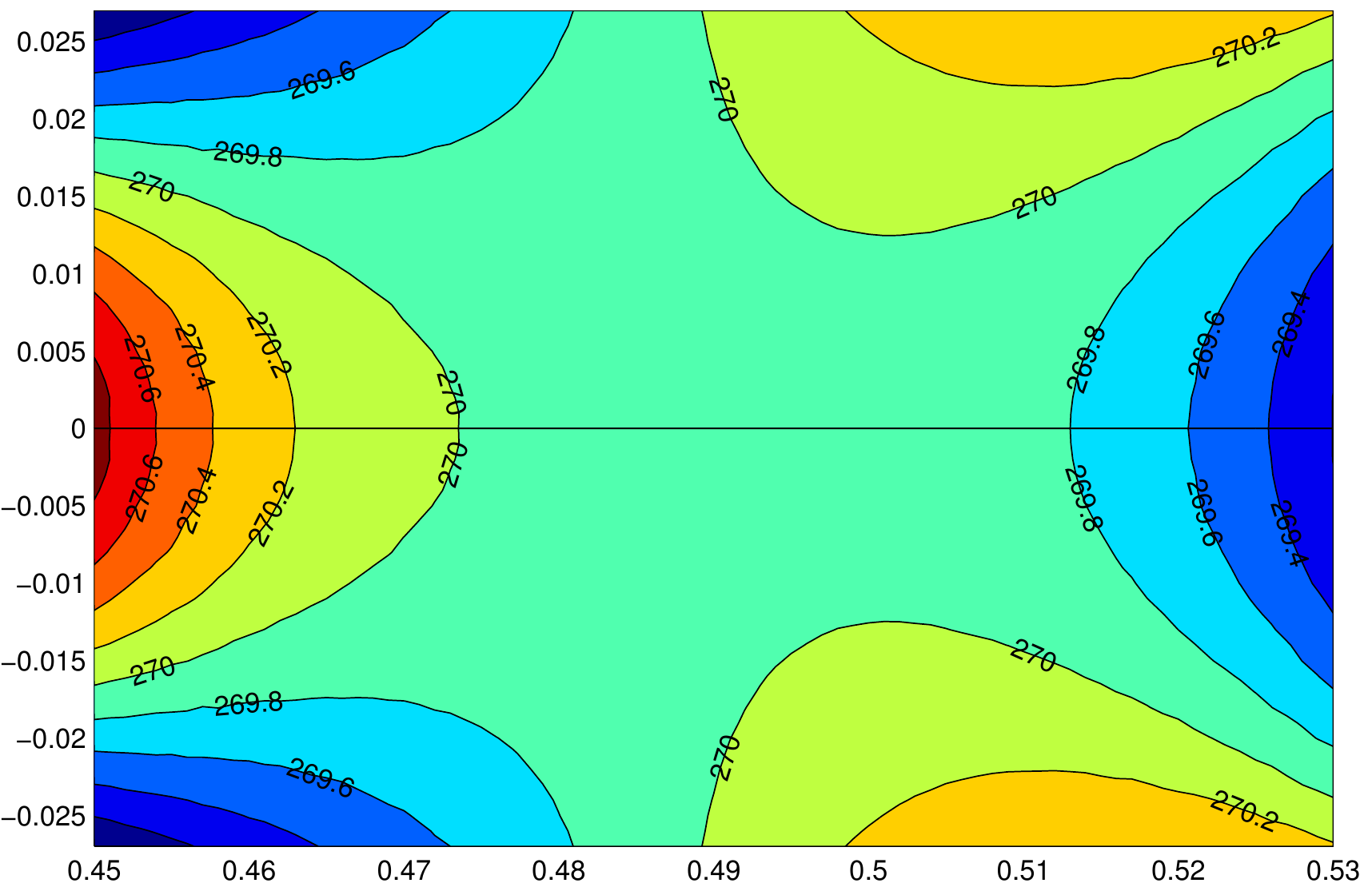}
\includegraphics[width=6.3cm]{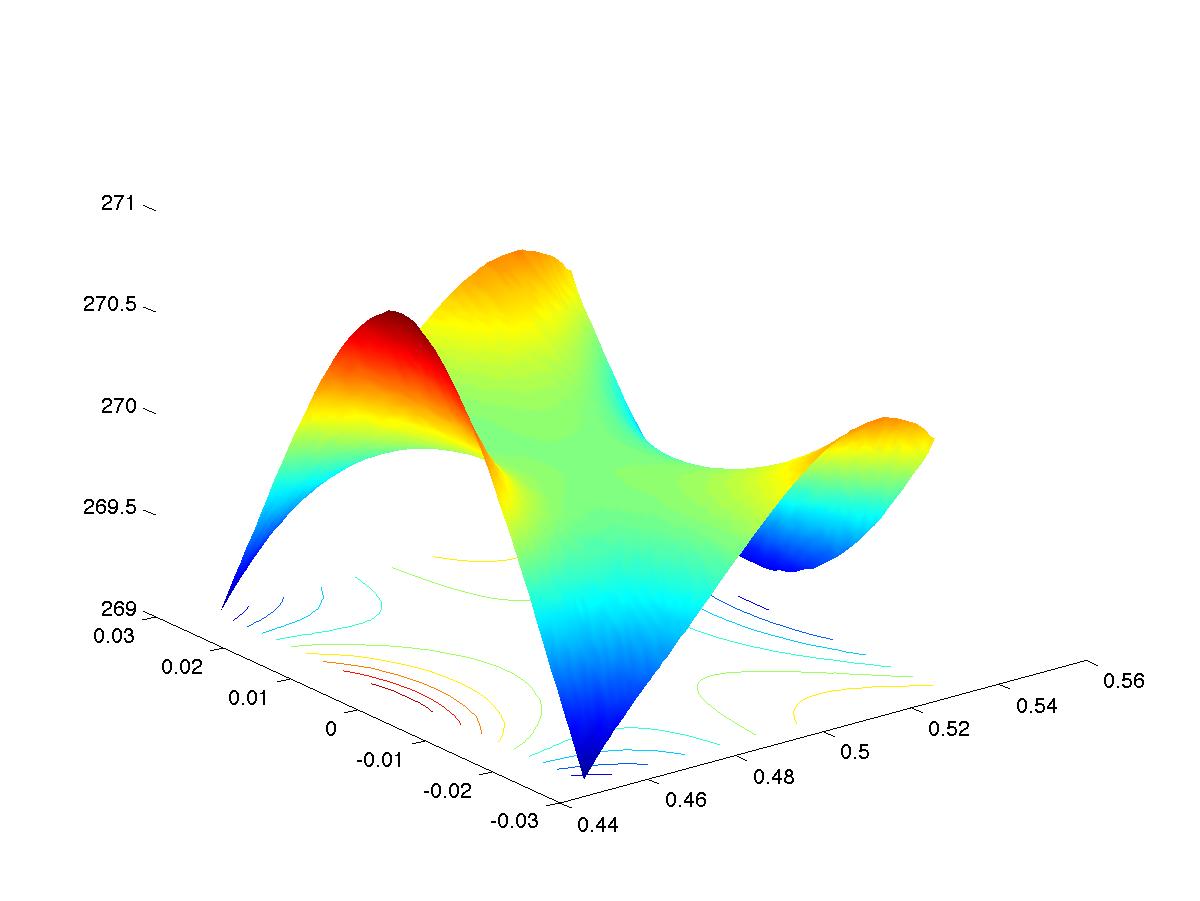}}
\caption{$\lambda_{j}^a$ vs. $a$ for $a$ around the inflexion point $a_{(j)}$, $j=3,4,5$.\label{fig.VPAB3-5zoom}}
\end{center}
\end{figure}

We remark that computing the first twelve eigenvalues of $(i\nabla+A_{a})^2$ on $\Sigma_{\pi/4}$, we have never found an eigenfunction for which five or more nodal lines end at a singular point $a$.

As we have already remarked, all the local maxima and minima of $\lambda_j^a$ in Figure~\ref{fig.VPABy0} correspond to non-simple eigenvalues. Plotting the nodal lines of the corresponding eigenfunctions, we have found that they all have a zero of order $1/2$ at $a$, i.e. one nodal line ending at $a$. Nonetheless, this is not a general fact: in performing the same analysis in the case $\Omega$ is a square $[0,1]\times[0,1]$, we have found that the third and fourth eigenfunctions have a zero of order $3/2$ at the center $a=(\frac 12,\frac 12)$, see Figure~\ref{fig.VecpABsqcentre}, which is in this case a maximum of $a\mapsto\lambda_3^a$ and a minimum of $a\mapsto\lambda_4^a$, see Figures~\ref{fig.ABsq1},~\ref{fig.ABsq2}.
We observe in Figure~\ref{fig.ABsq1} that the first and second derivatives of $\lambda_3^a$ and of $\lambda_4^a$ seem to vanish at the center $a=(\frac 12,\frac 12)$.

\begin{figure}[h!t]
\begin{center}
\subfigure[$a\mapsto \lambda_{j}^a$, $a\in\left\{(\frac n{100},\frac n{100}), 0< n< 200\right\}$, $1\leq j\leq 9$ (square).\label{fig.VPABMed}]{\includegraphics[height=4.8cm]{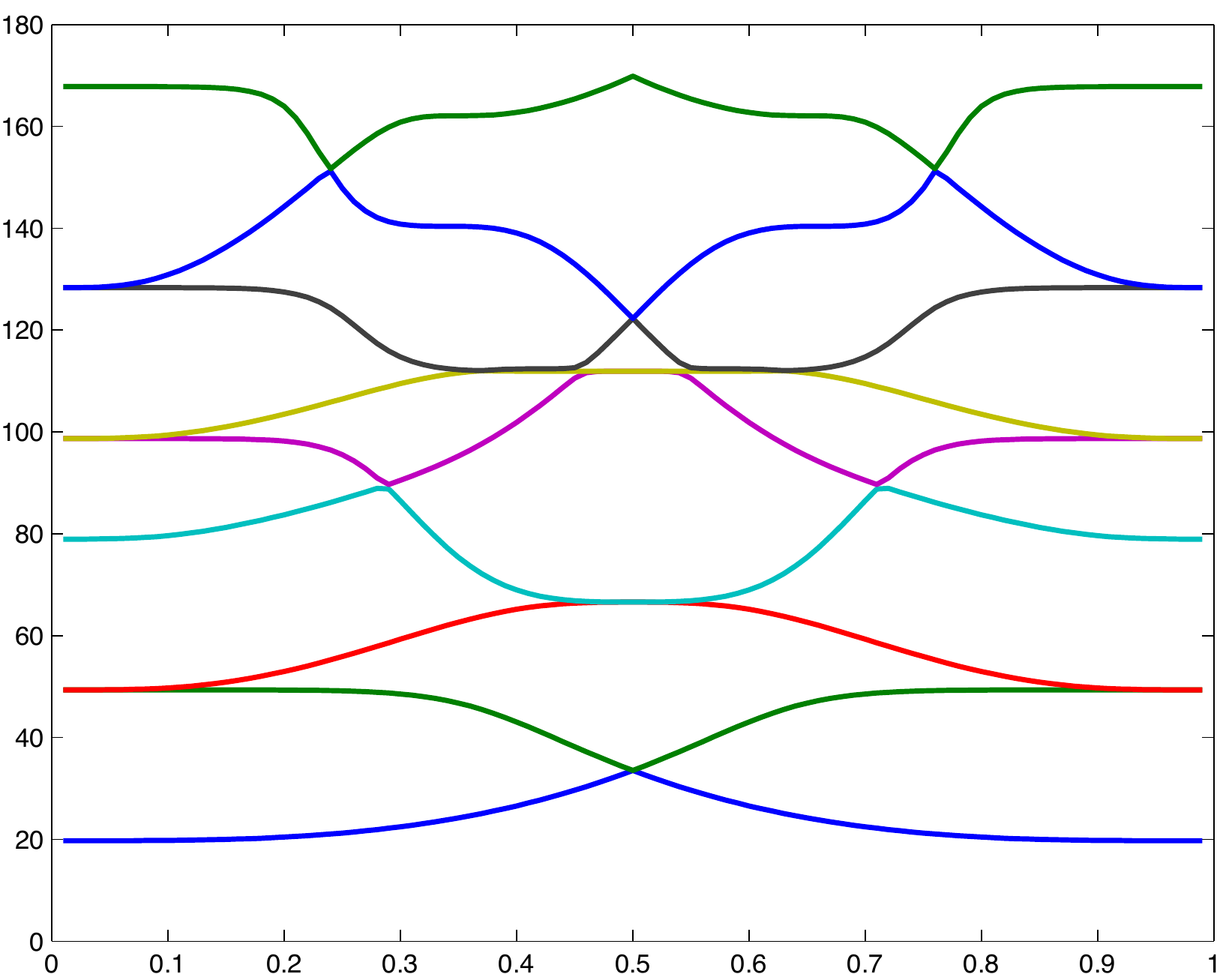}}
\subfigure[$a\mapsto \lambda_{j}^a$, $a\in\left\{(\frac n{100},\frac12), 0< n< 200\right\}$, $1\leq j\leq 9$ (square).\label{fig.VPABDiag}]{\includegraphics[height=4.8cm]{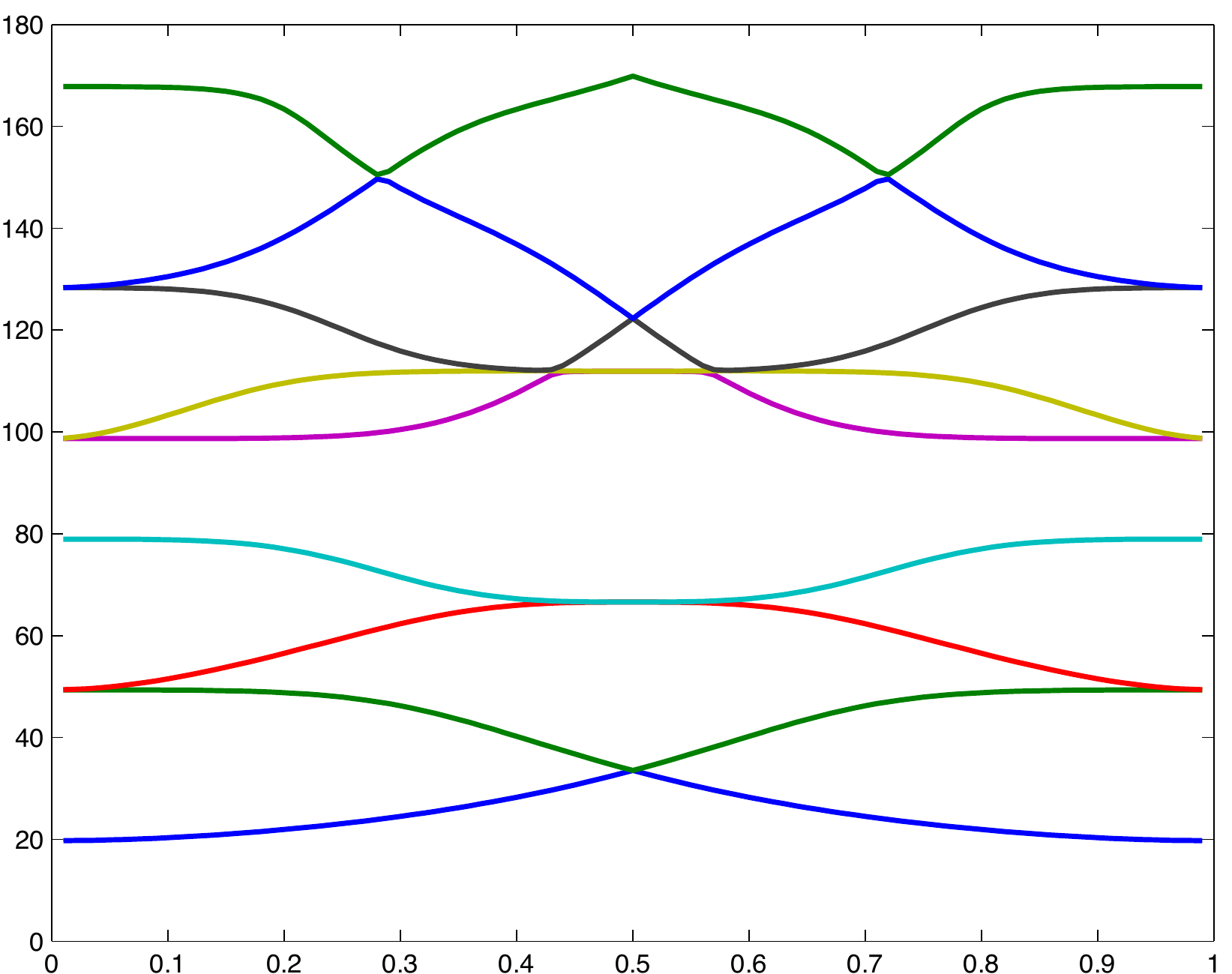}}
\caption{$a\mapsto \lambda_{j}^a$ for $a$ along the perpendicular bisector or the diagonal of a square\label{fig.ABsq1}}
\end{center}
\end{figure}


\begin{figure}[h!]
\begin{center}
\subfigure[$\lambda_{3}^a$ vs. $a$]{\includegraphics[height=3.3cm]{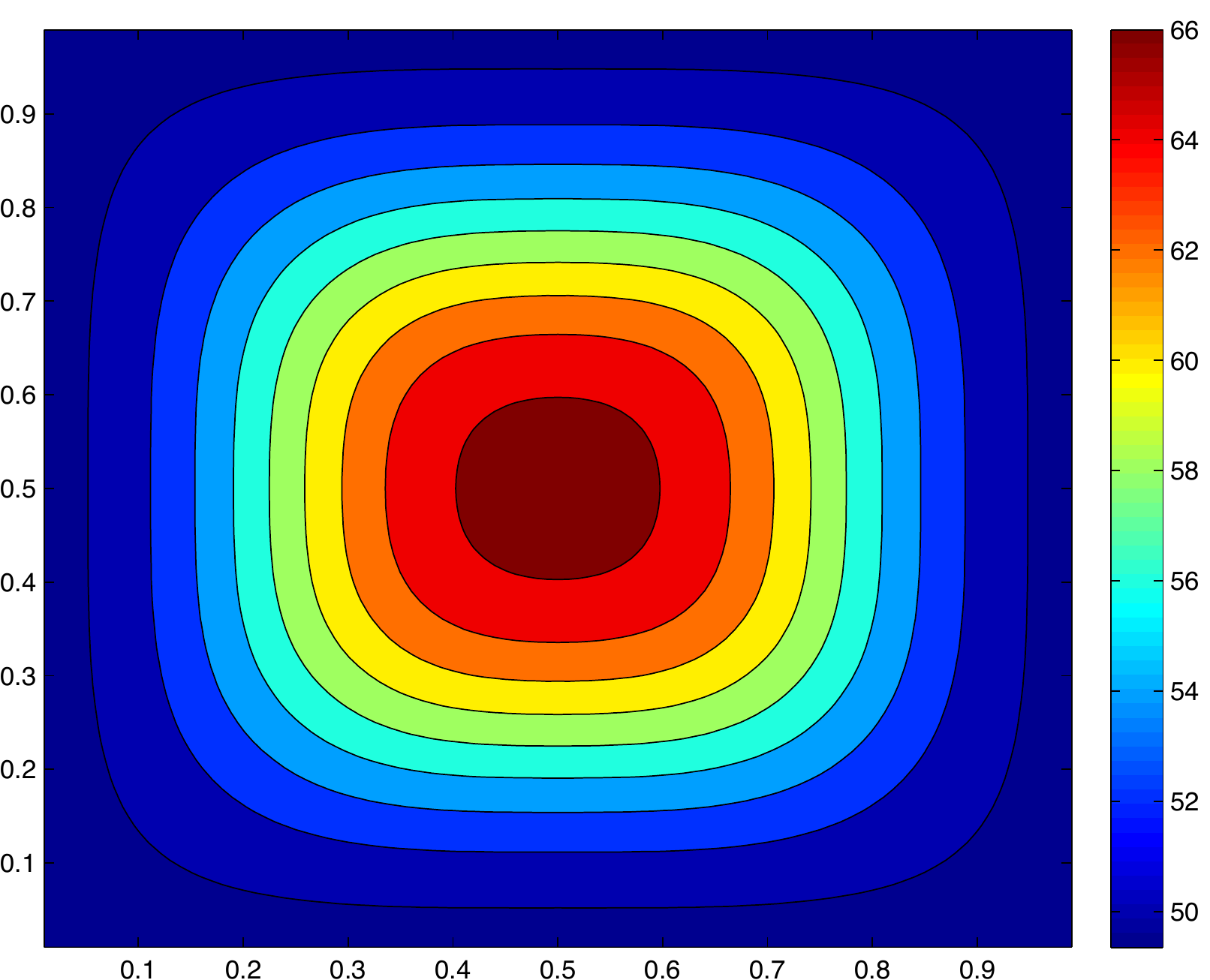}}
\subfigure[$\lambda_{4}^a$ vs. $a$]{\includegraphics[height=3.3cm]{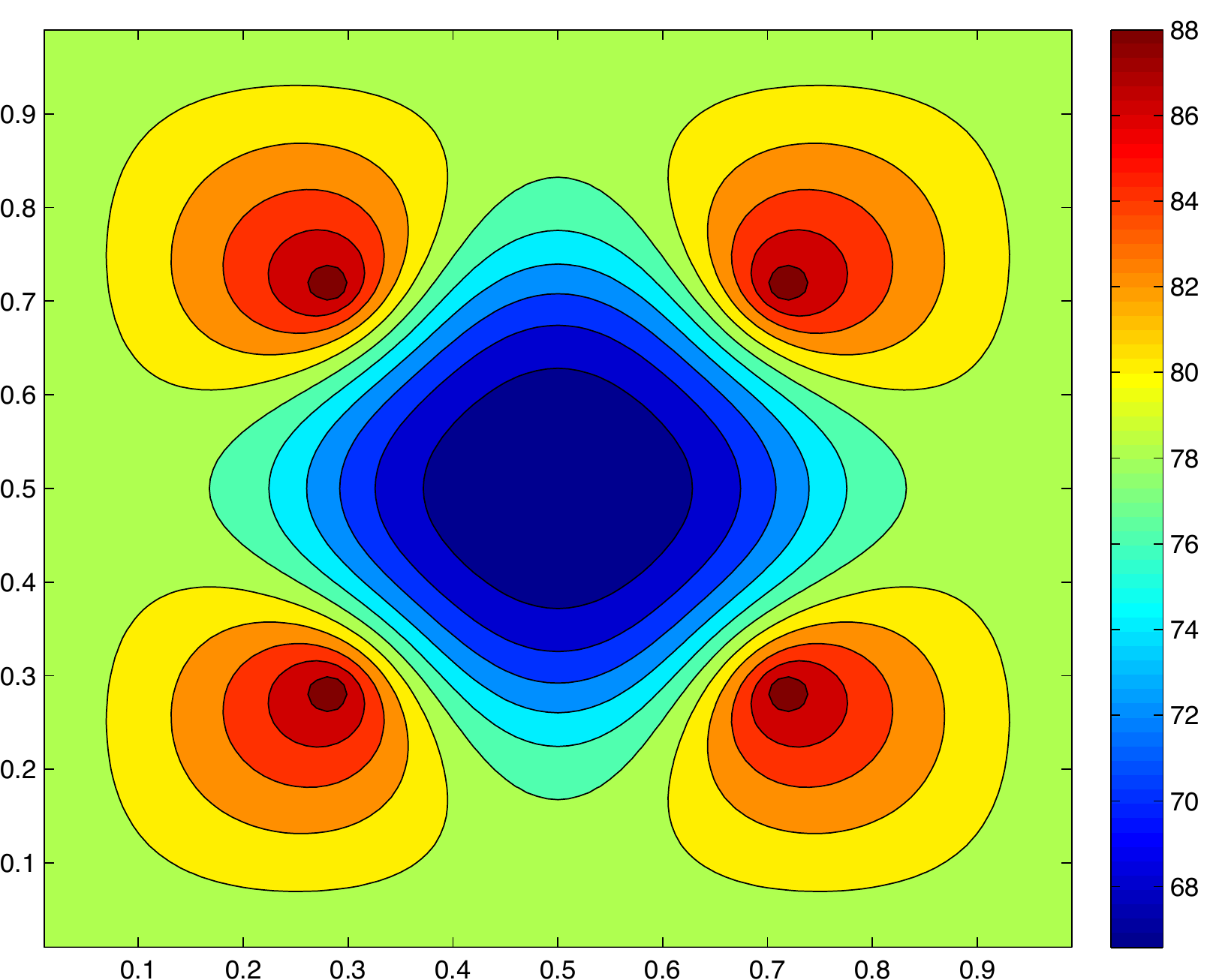}}
\subfigure[$\lambda_{5}^a$ vs. $a$]{\includegraphics[height=3.3cm]{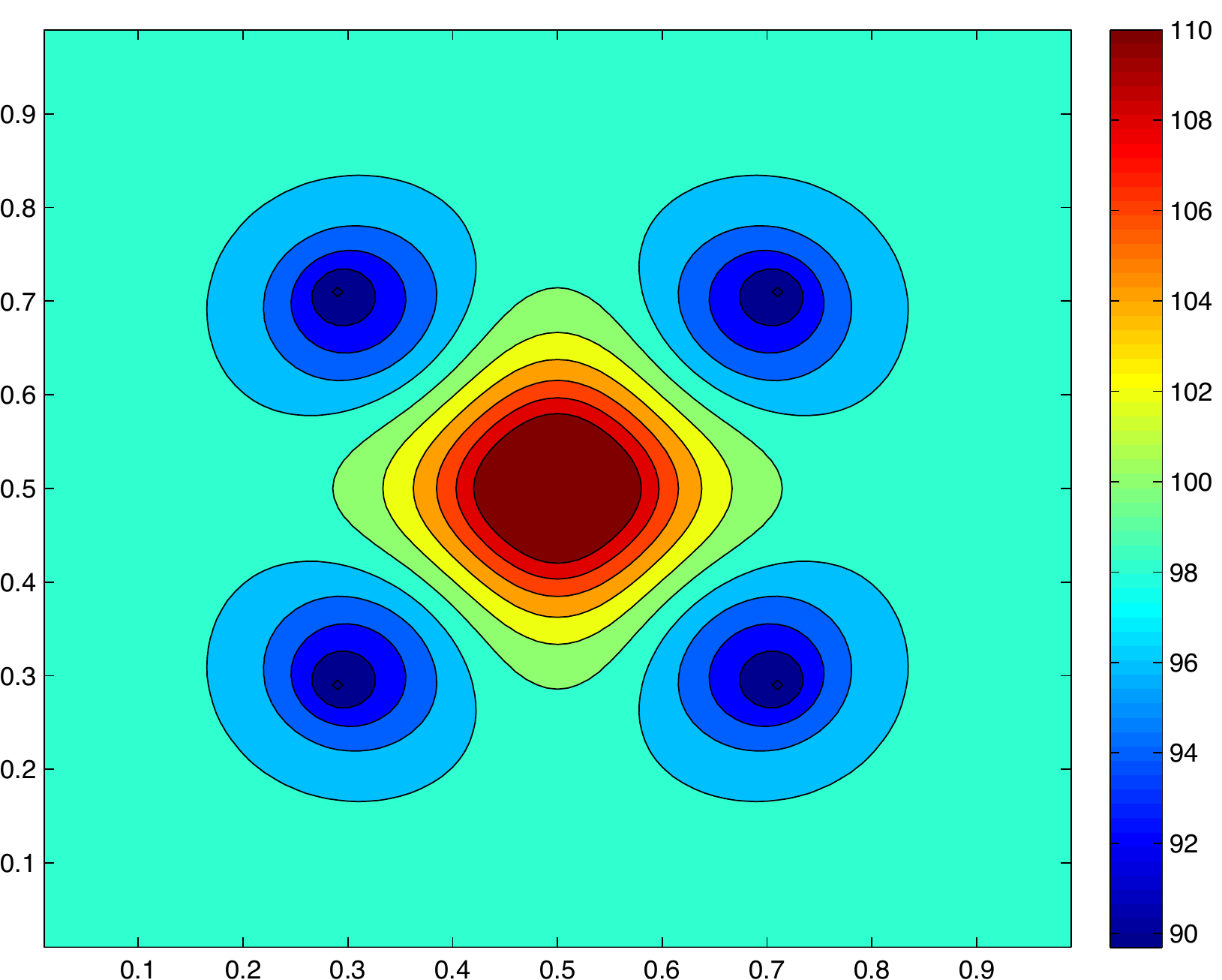}}
\caption{Eigenvalues of $(i\nabla+A_{a})^2$ in $[0,1]\times[0,1]$, $a\in \Pi_{50}$.\label{fig.ABsq2}}
\end{center}
\end{figure}


\begin{figure}[h!]
\begin{center}
\subfigure[$\lambda_{3}^a$, $a=(\frac{1}{2},\frac{1}{2})$]{\includegraphics[height=3cm]{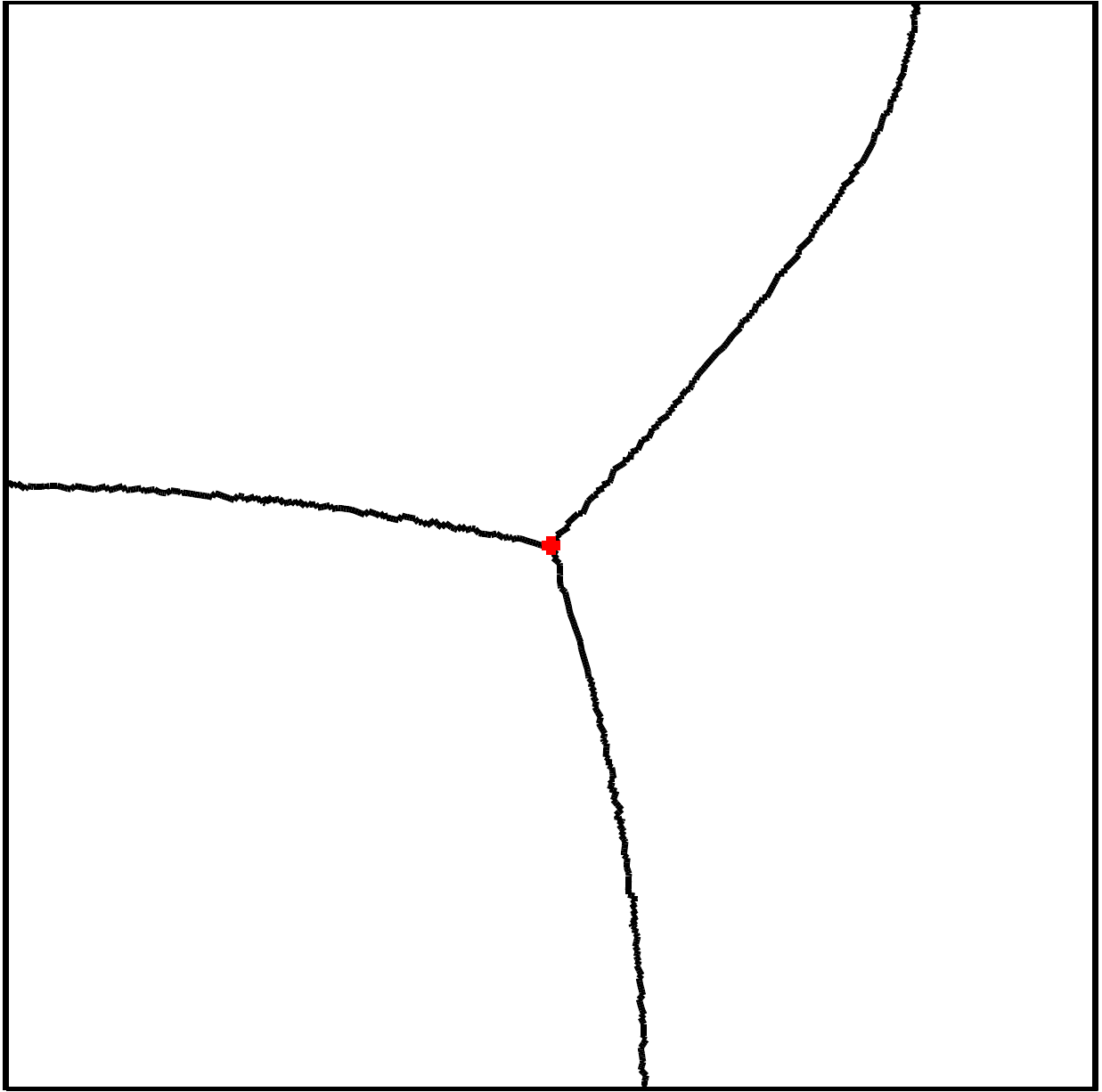}}$\qquad$
\subfigure[$\lambda_{4}^a$, $a=(\frac{1}{2},\frac{1}{2})$]{\includegraphics[height=3cm]{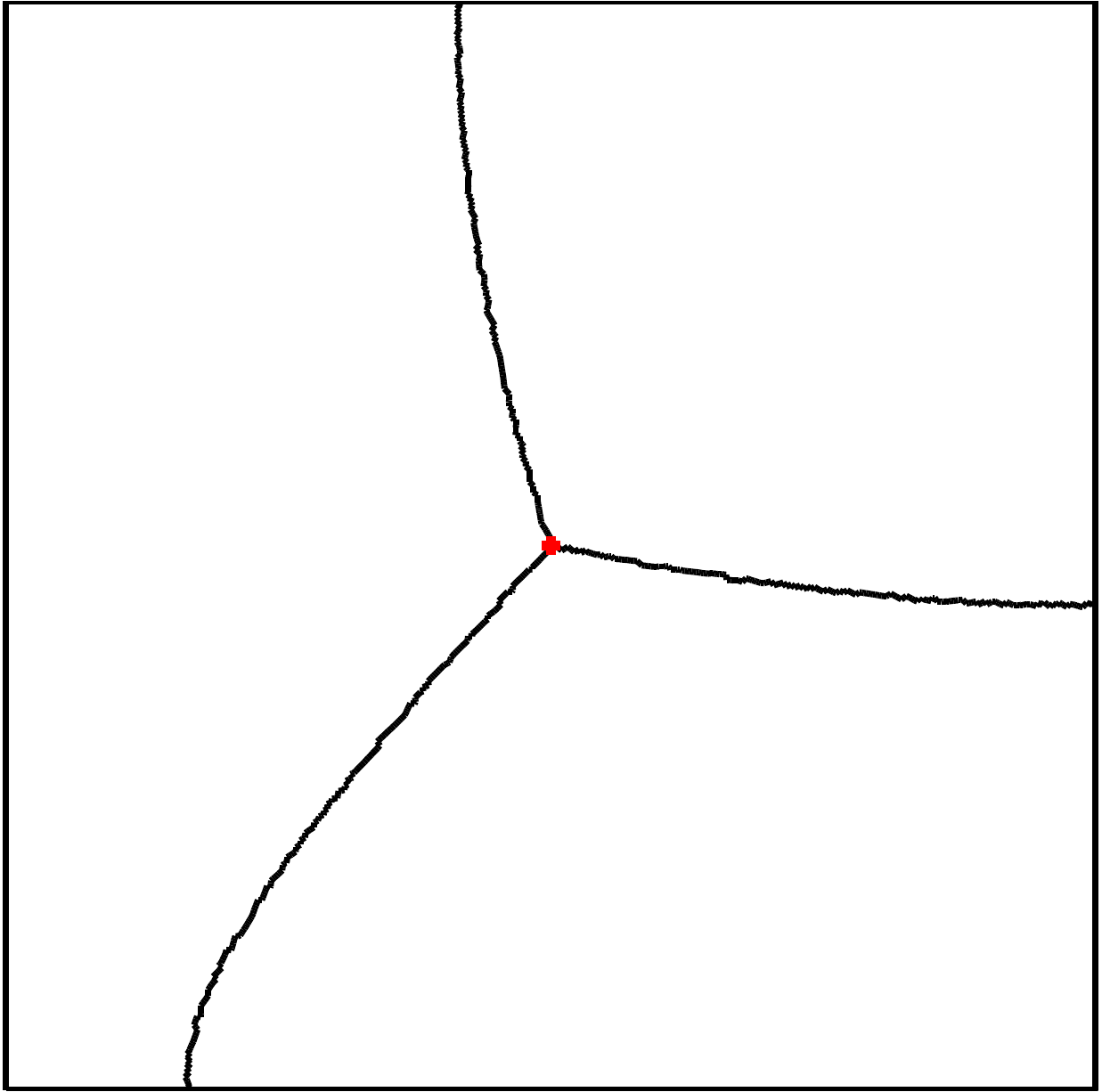}}
\caption{Nodal lines of an eigenfunction associated with $\lambda_{j}^a$, $j=3,4$, $a=(\frac 12,\frac 12)$.\label{fig.VecpABsqcentre}}
\end{center}
\end{figure}

\newpage \ 


\small

\noindent\verb"bonnaillie@math.cnrs.fr"\\
\noindent IRMAR, ENS Rennes, Univ. Rennes 1, CNRS, UEB, av. Robert Schuman, 35170 Bruz (France)\\

\noindent\verb"benedettanoris@gmail.com"\\
INdAM-COFUND Marie Curie Fellow \\
\noindent Laboratoire de Math\'ematiques, Universit\'e de Versailles-St Quentin, 45 avenue des \'Etats-Unis, 78035 Versailles cedex (France) \\

\noindent\verb"manonys@gmail.com"\\
Fonds national de la Recherche scientifique-FNRS \\
\noindent D\'epartement de Math\'ematiques, Universit\'e Libre de Bruxelles (ULB), Boulevard du triomphe, B-1050 Bruxelles (Belgium) \\
\noindent Dipartimento di Matematica e Applicazioni, Universit\`a degli Studi
di Milano-Bicocca, via Bicocca degli Arcimboldi 8, 20126 Milano (Italy)\\

\noindent\verb"susanna.terracini@unito.it"\\
Dipartimento di Matematica ``Giuseppe Peano'', Universit\`a di Torino, Via Carlo Alberto 10, 20123 Torino (Italy)

\end{document}